\documentclass[11pt]{amsart}
\usepackage{amsmath,amsthm}
\usepackage{amssymb} 
\usepackage{amscd} 
\usepackage{graphicx} 
\usepackage{epstopdf}
\usepackage{epsfig}
\usepackage{latexsym}
\usepackage[english,french]{babel}
\usepackage{amsfonts}
\usepackage{color}
\usepackage{array}
\usepackage{verbatim}
\usepackage[numbers]{natbib}
\renewcommand{\cite}[1]{\citeauthor*{#1} [\citeyear{#1}]}

\usepackage{amsthm,url,xspace}
\definecolor{couleurCitations}{rgb}{0,0,0.85}
\definecolor{couleurRef}{rgb}{0.75,0,0}
\usepackage[pdftex,,nativepdf=true,colorlinks=true,linkcolor=couleurRef,citecolor=couleurCitations,a4paper=true,bookmarks=true,plainpages=true,
backref=page, pagebackref=true,urlcolor= blue] {hyperref}

\def\1{\bold 1}

\def\N{\mathbb{N}}
\def\R{\mathbb{R}}
\def\Z{\mathbb{Z}}

\def\S{\mathbb{S}}
\def\Ss{\mathcal{S}}
\def\D{\mathbb{D}}
\def\F{\mathcal{F}}
\def\G{\mathcal{G}}
\def\T{\mathcal{T}}
\def\L{\mathcal{L}}
\def\C{\mathcal{C}}

\def\dps{\displaystyle}
\def\1n{\hbox{$\{1,\dots,n\}$}}
\def\bsl{\backslash}
\def\t{\times}

\newtheorem{LM}{Lemma}[section]
\newtheorem{THM}[LM]{Theorem}
\newtheorem{Q}[LM]{Question}
\newtheorem{PR}[LM]{Proposition}
\newtheorem{COR}[LM]{Corollary}
\newtheorem{CL}[LM]{Claim}

\newtheorem{RK}[LM]{Remark}

\newtheorem{DEF}[LM]{Definition}

\def\restriction#1#2{\mathchoice
              {\setbox1\hbox{${\displaystyle #1}_{\scriptstyle #2}$}
              \restrictionaux{#1}{#2}}
              {\setbox1\hbox{${\textstyle #1}_{\scriptstyle #2}$}
              \restrictionaux{#1}{#2}}
              {\setbox1\hbox{${\scriptstyle #1}_{\scriptscriptstyle #2}$}
              \restrictionaux{#1}{#2}}
              {\setbox1\hbox{${\scriptscriptstyle #1}_{\scriptscriptstyle #2}$}
              \restrictionaux{#1}{#2}}}
\def\restrictionaux#1#2{{#1\,\smash{\vrule height .8\ht1 depth .85\dp1}}_{\,#2}}

\begin{document}

 \selectlanguage{english}
\title{Compact leaves in Reebless or taut foliations} 
\author[Shanti Caillat-Gibert]{Shanti Caillat-Gibert}
\email{shanti@cmi.univ-mrs.fr}
\address{CMI, Universit\'e de Marseille-Provence,
13453-Marseille cedex 13, France}
\date{\today}
\keywords{
Reeb component, Taut foliation, Reebless foliation, non-taut foliation, torus leaf, incompressible torus, turbulization, spiraling
}
\subjclass{Primary 57M25, 57M50, 57N10, 57M15}
\begin{abstract}
Torus leaves play a crucial role in the theory of foliations. For example non-taut foliations admit a torus leaf (see \cite{Go}).\\
In this paper, we study all the foliations near a torus leaf, and try to understand why sometimes it is taut, or non-taut (and Reebless). We focus on some crucial examples to understand non-taut and Reebless foliations; and foliations admitting a non-separating torus leaf. It relies on the study of turbulization and spiraling, with generalizations.

\end{abstract}

\maketitle

\tableofcontents

\section{Introduction}
In this paper, all the manifolds $M$ are $3$-dimensional, compact, connected and irreducible.\\
The foliations studied on $M$ are of codimension one (i.e the leaves are $2$-dimensional). Sometimes we will consider foliations on surfaces, when it will be specified.\\

Since the works of G. Reeb and S.P. Novikov, we know that all manifolds $M$ as above admit a codimension one foliation, (see \cite{Li}). The construction of this foliation gives rise to a \textit{Reeb component}. Foliations whitout Reeb component (or \textit{Reebless}) are more interesting, because they give deep information on the manifold $M$, for example $\pi_{1}(M)$ is infinite or $M\cong \mathbb{S}^2\times \mathbb{S}^1$ (see \cite{No}).\\
D. \cite{Ga}, improved the theory by introducing the notion of \textit{taut} foliation.\\
 It is well known that taut foliations are Reebless; here we generalize this fact in Proposition \ref{sep_torus} (a first version was already in \cite{B1}, or in [\citeauthor{Godb} lemma \textbf{3.8}]  for closed manifold, or manifolds with only one torus boundary component).\\

\begin{PR}\label{sep_torus}
Let $M$ be a $3$-manifold with a transversely orientable foliation $\mathcal{F}$.\\
If the boundary of $M$ is a union of leaves with the same transverse orientation or if  $\mathcal{F}$ contains a compact separating leaf, then $\mathcal{F}$ is not taut. 
\end{PR}

D. Gabai showed that non-trivial homology is a sufficient condition to admit a taut foliation. The general problem of existence of a taut foliation in homology spheres is still open, even if many works partially answer the question (see for example \cite{BNR} for graph manifolds, \cite{GM} which is a complete classification for Seifert fibered manifolds and \cite{RSS} for hyperbolic manifolds). \\
 
Proposition \ref{sep_torus} has the following corollary~:
\begin{COR}
Any taut, transversely oriented foliation in a rational homo-\\logy sphere, has no compact leaf.
\end{COR}

\begin{proof}
Indeed, suppose it admits one compact leaf.  A  rational homology sphere cannot admit any non-separating surface (this induces non-trivial homology). Hence this compact leaf is separating in a closed manifold. By Proposition \ref{sep_torus}, the foliation cannot be taut; a contradiction.\\
\end{proof}

This fact is crucial for showing the main result of \cite{GM}.\\

S. \cite{Go} showed that a non-taut foliation always admits a torus leaf (see Section \ref{prop-non-taut}).  We will see examples of non-taut foliations admitting a separating torus leaf, and non-taut foliations admitting a non-separating torus leaf.\\

One goal of this paper is to understand better non-taut and Reebless foliations. Note that together with \citeauthor{No}'s theorem, if a foliation of a closed $3$-manifold (or with boundary leaves) is non-taut and Reebless, then it admits an incompressible torus leaf. Hence a great part of this paper studies foliations near incompressible torus leaves. Note that those non-taut and Reebless foliations never arise in hyperbolic closed manifolds (since they cannot contain incompressible torus). \\ 
In this context we study two geometric processes~: \textit{turbulization} ($\T_{*}$ component) and \textit{spiraling} ($\Ss_{*}$ component) which occur near a torus leaf (or more generally near a closed compact surface).\\
We will see that turbulization and spiraling can give rise to non-taut Reebless foliations. Spiraling was first introduced by \cite{Ga}, and here we give more details to define it; and then we link it to turbulization (which was first defined by G. Reeb).\\
Conversely if a foliation admits a torus leaf then roughly speaking, in a regular neighborhood of this torus there is turbulization or spiraling, which is the aim of next proposition (for precise definitions see Section \ref{prel}).\\

\begin{PR} \label{torus-leaf}
Let $M$ be a manifold admitting a transversely oriented $\mathcal{C}^{2}$-foliation $\mathcal{F}$.\\ 
Assume that  if $M$ is either $T^{2}\t \S ^{1}$ or $ T^{2}\t I$ then $M$ is not only foliated by torus leaves.\\
Then $\mathcal{F}$ contains either a $\mathcal{S}_{*}$, or a $\mathcal{T}_{*}$ component, if and only if $\mathcal{F}$ admits a torus leaf.
\end{PR}

If all the boundary components of $M$ (with a transversely oriented foliation $\F$) are torus leaves; we say that $\F$ has a \textit{bad orientation} if the transverse orientation on all the torus boundary leaves is the same (all point inward or all point outward); otherwise we say that it is a \textit{good orientation} (at least two torus leaves have opposite orientation, one inward and the other outward).\\
Then, we will see that the tautness of the foliation is deeply linked to the good or bad transverse orientation as suggests the following theorem.
\begin{THM}\label{main}
Let $M$ be a manifold with a transversely oriented $\C^{1}$-foliation~$\F$.\\
Assume that the boundary of $M$ is a union of torus leaves.\\
Assume also that $\F$ does not admit neither interior torus leaf, nor embedded annulus whose induced foliation by $\F$ is a Reeb annulus.\\
 Then, $\F$ is taut if and only if $\F$ has a good orientation.
\end{THM}
%
%

\medskip
\noindent
 {\sc Schedule of the paper.}\\
We organize the paper as follows.\\

In Section \ref{prel}, we recall basic definitions and notations.

Section \ref{Sec-Reeb-turb} introduces the well-known Reeb's component, and the geometric process of turbulization by two different interesting ways.

In Section \ref{spir} we define the geometric process of spiraling and generalize it under condition. 

In  Section \ref{torus}, we first give the equivalence of those two geometric processes in a particular case. For this, we describe $\mathcal{C}^{2}$-foliations near a torus leaf and prove Proposition \ref{torus-leaf}.

Section \ref{prop-non-taut} proves Proposition \ref{sep_torus} saying that separating torus leaves or boundary leaves with the same transverse orientation are contained in a non-taut foliation.\\
Furthermore, we explain why each hypothesis is necessary for this Proposition. For this, we consider the Waldhausen manifold and give an example of taut foliation with a single boundary leaf, non-transversely orientable with non-compact leaves on $Q$.\\
Then we prove Corollary \ref{gen-goodman} of Theorem \ref{tore_Go} which states that a non-taut foliation of a closed $3$-manifold (or manifold with boundary leaves) always contains a torus leaf (\cite{Go}). \\
The following of the paper focuses on the two cases: separating torus leaves (Section \ref{non-taut-sec}) and non-separating ones (Section \ref{non-sep-torus}).

Indeed, Section \ref{non-taut-sec} provides a collection of different non-taut foliations as follows.\\
There are the ones with Reeb components, that we cannot remove (example on $\S^{3}$);
and the ones with Reeb components that we can remove, or non-taut and Reebless foliations; (example on $T¬^{3}$).

The aim of Section \ref{non-sep-torus} is to understand why a foliation with a non-separating torus can be either taut or non-taut. We start with key examples, and we generalize by proving Theorem \ref{main} saying that if a foliation of a manifold with torus boundary leaves does not contain embedded Reeb annuli, then it is taut if and only if it has a good orientation (at least two boundary components whose transverse orientation is opposite).\\

\medskip
\noindent
{\sc Acknowledgement.}\\

Proposition \ref{non-taut} is due to an interesting discussion with Andr\'as Juh\'asz, and the author wants to thank him.

\medskip
\noindent
{\sc Perspectives.}\\

We will see in section \ref{non-taut-sec} some examples of non-taut and Reebless foliations, on different manifolds. One open question is the following~:
\begin{Q}
What are the manifolds admitting a non-taut and Reebless transversely oriented foliation, but not admitting a taut foliation?
\end{Q}
 Note that the first examples of such manifolds were found by \cite{BNR} and they are graph manifolds, (we have seen that this question is trivial for closed hyperbolic manifolds).\\
  Note also that the examples given here of non-taut Reebless foliations, concern manifolds admitting (another) taut foliation.\\
 \cite{GM} have given an infinity of examples of Seifert manifolds not admitting a taut foliation, hence we should ask the following :
 \begin{Q}
Is this family admitting a non-taut and Reebless foliation?
\end{Q}
Note also that nothing is known about the existence of taut foliations or of non-taut and Reebless foliation among non-hyperbolic manifolds admitting an hyperbolic submanifold.\\

\section{Preliminaries}\label{prel}

From now on, $M$ will be a compact connected irreducible $3$-manifold, possibly with boundary, and $\mathcal{F}$ will be a codimension one foliation on $M$ considered up to isotopy when it is not indicated.\\
Furthermore we will let $I=[0,1]$, and denote $\mathring{X}$ the interior of $X$, and $\overline{X}$ the closure of $X$, when it makes sense, and let $T^2\cong \mathbb{S}^1\times\mathbb{S}^1$.\\
For all the following the circle $\S^{1}$ is parametrized by $\{e^{i\theta}, \theta\in ]-\pi,\pi]\}$, but for more simplicity we will consider it as $\{\theta\in ]-\pi,\pi]\}$.

\medskip\noindent
{\bf Separating surfaces and non-separating surfaces.}
A properly embedded surface $F$ in a $3$-manifold $M$ is  said to be a
\textbf{\textit{separating surface}} if $M\bsl F$ is not connected;
otherwise, $F$ is said to be a \textbf{\textit{non-separating surface}} in $M$.
If $F$ is a separating surface, we call \textbf{\textit{sides of $F$}} the connected components of $M\bsl F$.

A $3$-manifold is said to be \textbf{\textit{reducible}} if $M$ contains an \textbf{\textit{essential}} $2$-sphere,
i.e. a $2$-sphere which does not bound any $3$-ball in $M$.
Then, either $M$ is homeomorphic to $\mathbb{S}^1\times\mathbb{S}^2$, or $M$ is a non-trivial connected sum.
If $M$ is not a reducible $3$-manifold, we say that $M$ is an \textbf{\textit{irreducible}} $3$-manifold.

\medskip\noindent
{\bf Incompressible torus.}
An embedded torus $T$ in $M$ is said to be \textbf{\textit{incompressible}} if the induced map $\pi_{1}(T) \rightarrow \pi_{1}(M)$ is injective, otherwise we say that $T$ is compressible.\\
Note that in an irreducible manifold, a compressible torus is always separating, while an incompressible torus can be separating or non-separating.

\medskip\noindent
{\bf Transverse orientation.}
Let $M$ be a compact connected $3$-manifold possibly with boundary.\\
Let $\mathcal{F}$ be a codimension one foliation on $M$.\\
A foliation $\mathcal{F}$ of $M$ is \textbf{\textit{transversely orientable}}, if $M$ admits a non-zero continuous vector field, transverse (i.e non-tangent) to all the leaves.\\
If we fix such a non-zero continuous vector field, then $\mathcal{F}$ is said to be \textbf{\textit{transversely oriented}}.

\medskip\noindent
{\bf Reeb annulus.}

First, we define a foliation of $\mathbb{R}^2$.
Let $f~:\mathbb{R}^2\rightarrow\mathbb{R}$

  \begin{center}
 $f~:(x,z)\mapsto(x^2-1)\times\exp(z)$
\end{center}

$f$ is a  submersion, so it defines a foliation $\mathcal{F}$, axially symmetric about the $z$-axis, where~:

\begin{itemize}
\item $f^{-1}(\lbrace 0 \rbrace)$ is a union of two vertical leaf $\{x= 1\}$ and  $\{x= -1\}$.\\

\item $f^{-1}(\lbrace c^2 \rbrace)$ are leaves satifying the equation\\
 $z=2\log(c)-\log(x^2-1)$, for $|x|>1$. \\
When $z\rightarrow+\infty$, $x^2 \rightarrow 1$ so the leaves tend toward the vertical leaves. \\
When $z\rightarrow-\infty$, $x^2 \rightarrow +\infty$.\\
The general shape is $-\log$.\\

\item $f^{-1}(\lbrace -c^2 \rbrace)$ are parabolic leaves satifying the equation $z=2\log(c)-\log(1-x^2)$, for $|x|<1$. They meet the $z$-axis for $z=2\log (c)$. \\
 When $z\rightarrow+\infty$, $x^2\rightarrow 1$, so the leaves tend toward the vertical leaves. \\
\end{itemize}


\begin{figure}[htb!]
\includegraphics[width=9cm]{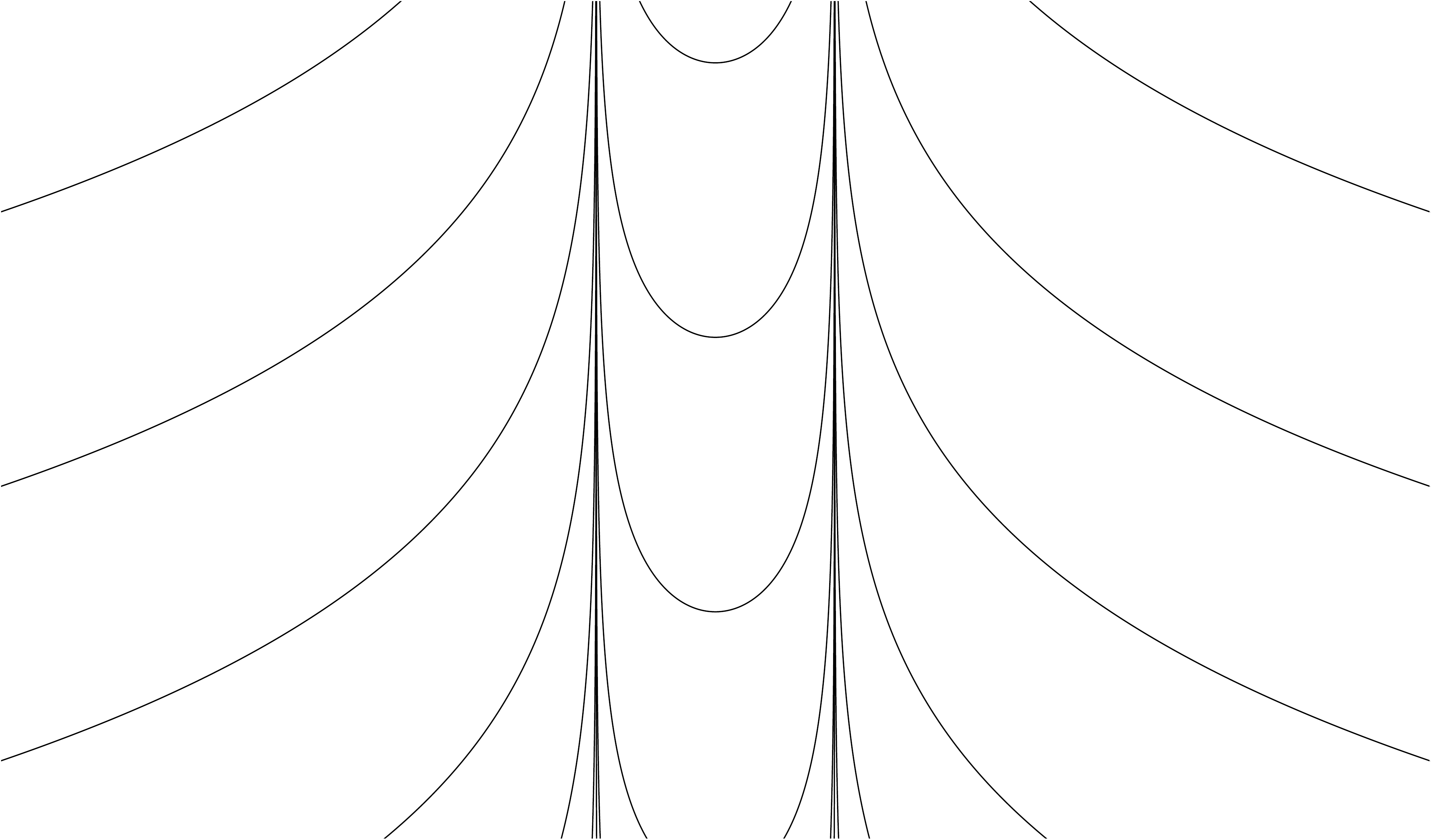}
\caption{ Foliation $\mathcal{F}$ of $\mathbb{R}^2$}\label{R2}
\end{figure}

$\mathcal{F}$ is invariant under integral translations along the $z$-axis; then it induces foliations on an annulus as follows.\\

Consider the restricted foliation of $\mathcal{F}$ on the set $R=\{(x,z)\in  \mathbb{R}^2, -1\leq x\leq 1\}$. The annulus $R/\!\raisebox{-.65ex}{\ensuremath{\sim}}$ where $(x,z)\sim (x,z+k), k\in \Z$ admits an induced foliation by $\F$, and is called \textbf{\textit{Reeb annulus}} as illustrated in Figure \ref{Reeb-annulus}.\\

\begin{figure}[htb!]
\includegraphics[width=8cm]{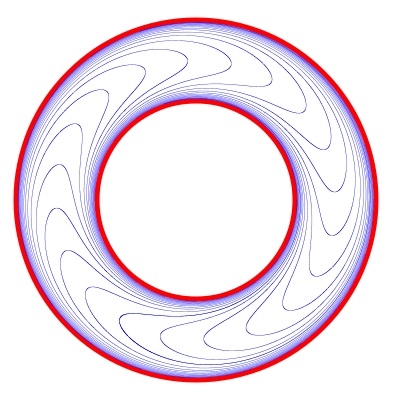}
\caption{ Reeb annulus (from Wikipedia)}\label{Reeb-annulus}
\end{figure}

\medskip\noindent
{\bf Direction of rotation of a spiral foliation.}
Let $X=\{(x,\theta), x\in  I, \theta\in ]-\pi,\pi]\}\cong I\t \S^{1}$ be an annulus foliated with two circle boundary leaf and spiral leaves in the interior of $X$ (see Remark \ref{Reeb-turb} for a definition of this foliation) we call this foliation a \textbf{\textit{spiral foliation}}. Keeping fixed the two boundary components, there is two non-isotopic such foliations drawn in Figure \ref{dir}.

\begin{DEF}
Consider a foliation of $X$ with spiral foliation. Choose any spiral leaf and orient it so that $x$ grows in $I$ (i.e fix the direction of rotation so that $x$ grows). \\
That induces an orientation by continuity on all the leaves of this foliation hence on the circle leaves (which is the same). If this orientation is a clockwise direction of rotation we call that foliation a \textbf{\emph{clockwise spiral foliation}} and the spiral leaves are called \textbf{\emph{clockwise spiral}}, otherwise we call it \textbf{\emph{anti-clockwise spiral foliation}}, and the spiral leaves are called \textbf{\emph{anti-clockwise spiral}}.
\end{DEF}

\begin{figure*}[htb!]
\begin{minipage}[b]{0.48\linewidth}
\centering
\centerline{\includegraphics[width=6cm]{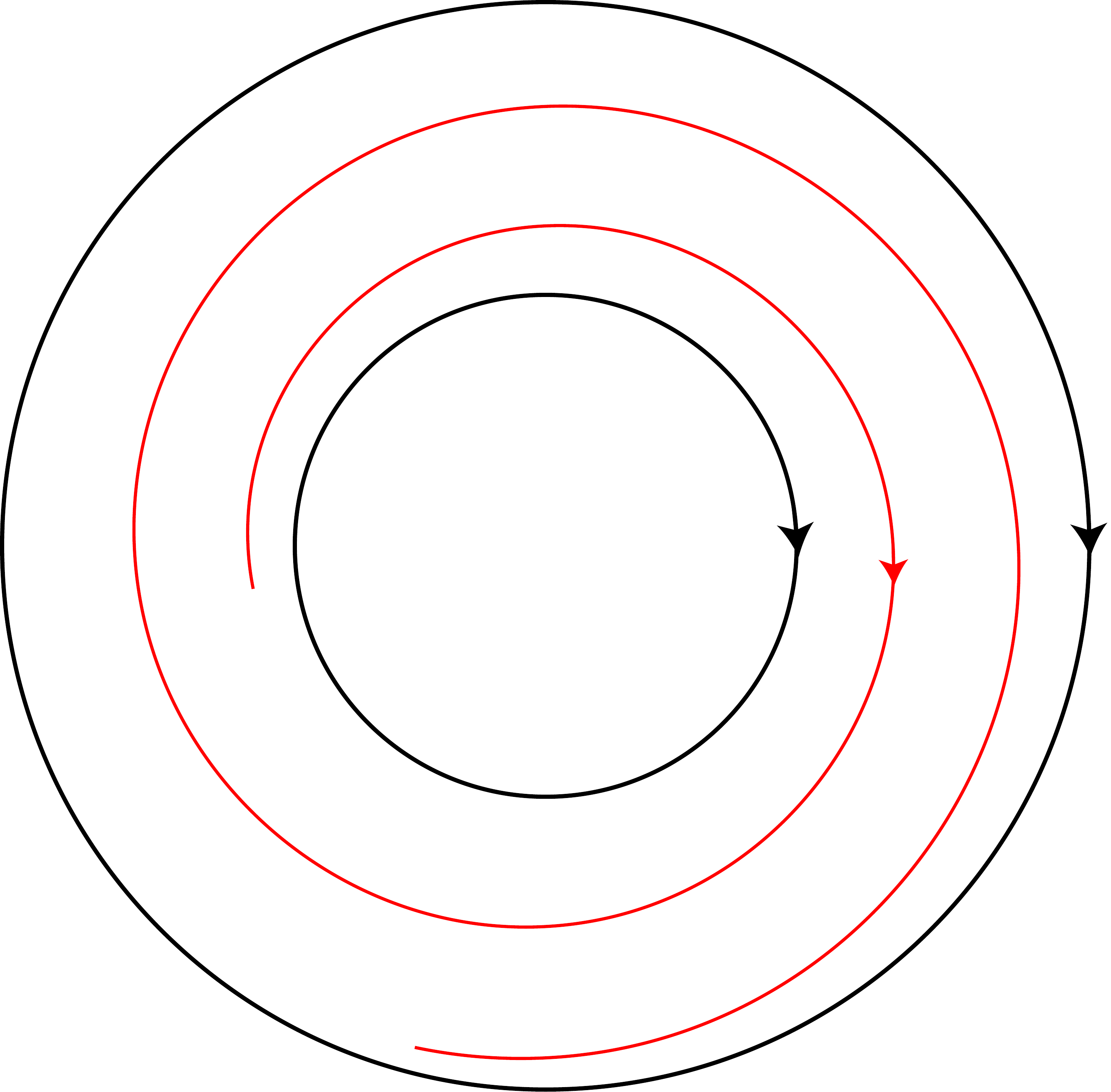}}
\centerline{\footnotesize{(a) Clockwise spiral foliation}}
\end{minipage}
\hfill
\begin{minipage}[b]{0.48\linewidth}
\centering
\centerline{\includegraphics[width=6cm]{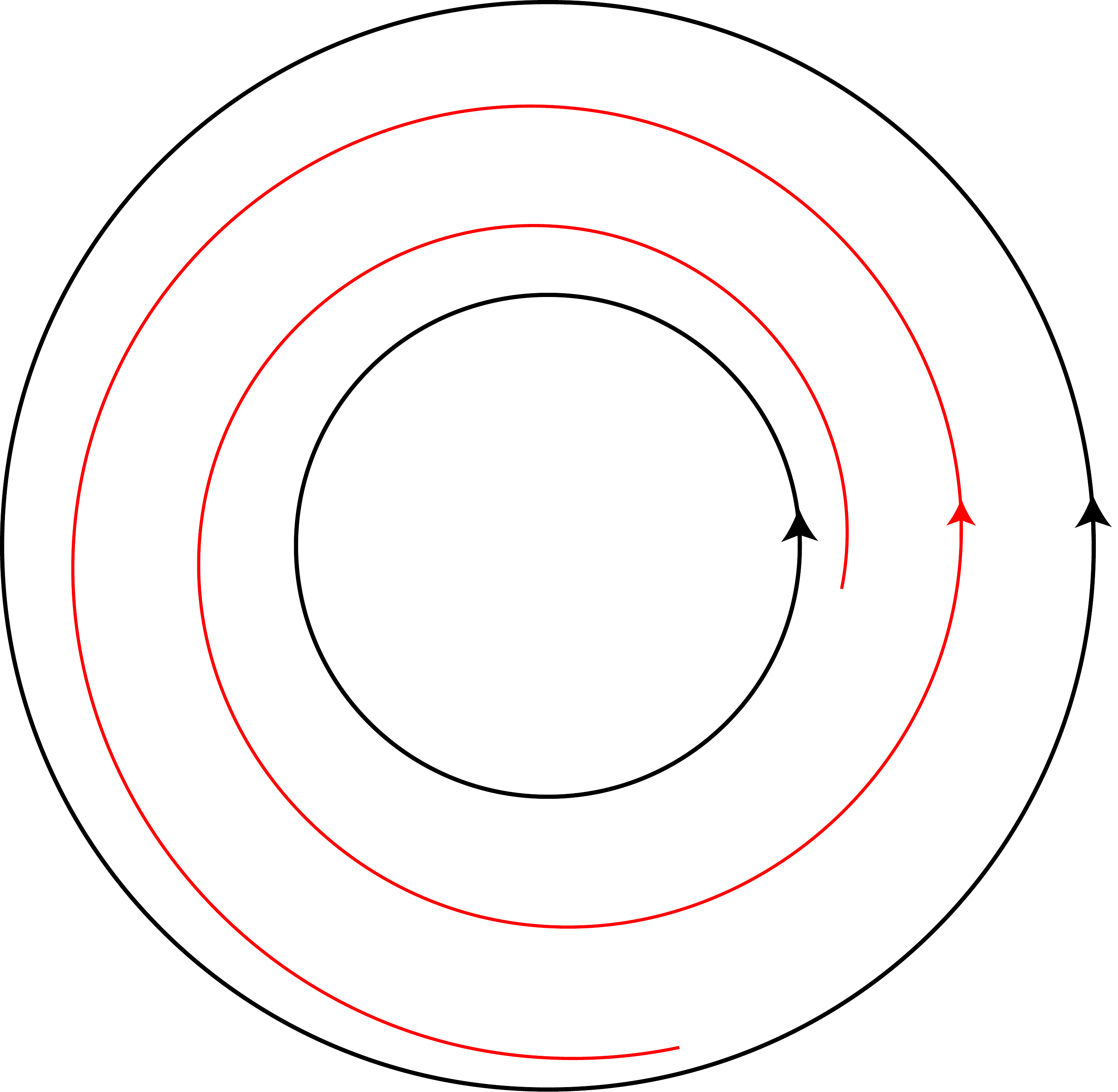}}
\centerline{\footnotesize{(b) anti-clockwise spiral foliation}}
\end{minipage}
\caption{ Direction of rotation of the foliations on $X$  }
\label{dir}
\end{figure*}

\medskip\noindent
{\bf An other foliation of the annulus, denoted by $\C$.}

Now we construct a foliation of the annulus where one boundary component is transverse to the foliation and the other is tangent. The leaves will be homeomorphic to $\R^{+}$, except one circle boundary leaf.\\
Consider the restricted foliation of $\mathcal{F}$ on the set $R_{r}=\{(x,z)\in  \mathbb{R}^2, 1\leq x\leq r\}$, for $r>1$; (or equivalently $\{(x,z)\in  \mathbb{R}^2, r\leq x\leq 1\}$, for $r<1$).\\
The annulus $R_{r}/\!\raisebox{-.65ex}{\ensuremath{\sim}}$ where $(x,z)\sim (x,z+k), k\in \Z$ admits an induced foliation by $\F$ that we will call $\mathcal{C}$.\\
%

$\mathcal{C}$ admits a circle boundary leaf, and the other boundary component is transverse to the foliation. All the non-compact leaves are homeomorphic to the ray $\mathbb{R}^+$ and are limiting along the circle boundary leaf as illustrated in Figure \ref{C}.\\

\begin{figure}[htb!]
\includegraphics[width=4cm]{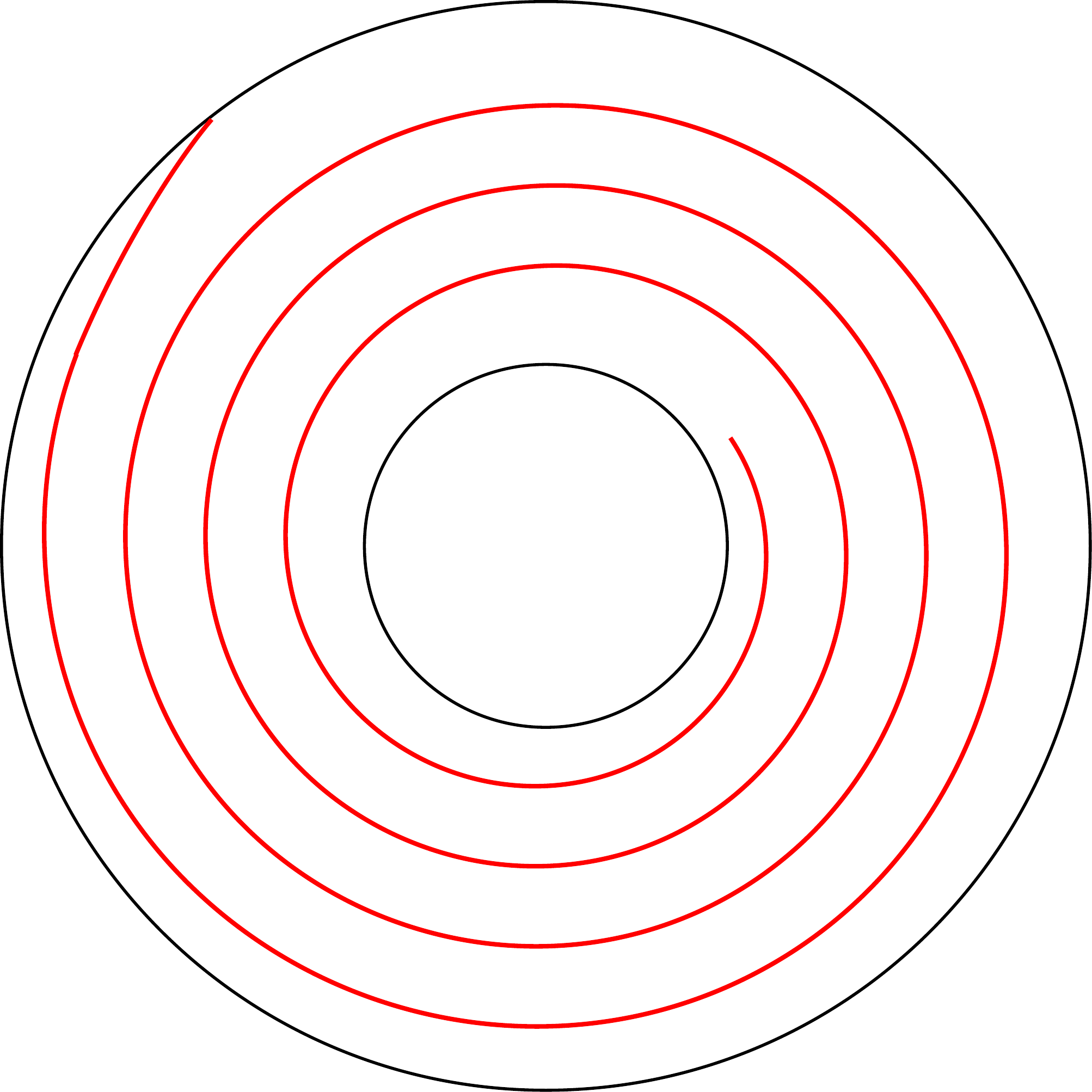}
\caption{One non-compact leaf of the foliation $\mathcal{C}$}\label{C}
\end{figure}

\medskip\noindent
{\bf Taut foliation.}
Let $\mathcal{F}$ be a foliation of a $3$-manifold $M$. An embedded loop, or respectively a properly embedded arc $\gamma$ (if $ \partial M\not = \emptyset $), is called \textbf{\textit{transverse  loop}} or respectively \textbf{\textit{transverse  arc}} if $\forall F\in  \mathcal{F}$ such that $\gamma \cap F\not = \emptyset $, the intersection $\gamma \cap F$ is transverse.\\
$\mathcal{F}$  is \textbf{\textit{taut}}, if for every leaf  $F$ of $\mathcal{F}$ there exists $\gamma$ an embedded transverse loop, or properly embedded transverse arc (if $ \partial M\not = \emptyset $), such that $\gamma \cap F\not = \emptyset $; and if $\mathcal{F}_{|\partial M}$ contains no Reeb annulus.\\

The following theorem is the famous theorem of \cite{Ga} on the existence of taut foliations,
which is stated here for closed $3$-manifolds.

\begin{THM}[D. \cite{Ga}]\label{Gabai}
Let $M$ be a closed $3$-manifold.
If $H_2(M; \mathbb{Q})$ is non-trivial then $M$ admits a taut foliation.
\end{THM}

\medskip\noindent
{\bf Foliated component.}
Suppose $M$ admits a foliation $\mathcal{F}$, and that $\mathcal{F}_{V}$ is a foliation of a submanifold $V$ of $M$. We say that $M$ admits a \textbf{\textit{foliated component}} $\mathcal{F}_{V}$, if the induced foliation by $\mathcal{F}$ on $V$ is isotopic to $\mathcal{F}_{V}$ in $M$.\\

\medskip\noindent
{\bf Foliation preserving homeomorphism.}
Let $M$ be a manifold admitting a foliation $\F$ and $N$ a manifold admitting a foliation $\G$.
An homeomorphism $f : M \rightarrow N$ is a \textbf{\textit{foliation preserving homeomorphism}}, if $f$ sends the leaves of $\F$ on the leaves of $\G$, i.e if $f$ preserves the leaves.\\

\section{Reeb component and Turbulization}\label{Sec-Reeb-turb}
In this section we first define Reeb's component which is a foliation of a solid torus tangent to the boundary, and then we define a particular foliation of $T^2 \times I$, where one torus boundary component is a leaf and the other is transverse to the foliation, which will be called \textit{turbulization} component.\\
  Reeb's component and the process of turbulization were firstly defined by G.\cite{Re}.\\
  Nowadays this construction is very common, and can be found for example in the notes of M. \cite{B1}.\\
  Finally, we define \textit{generalized turbulization}.

%
%
%
%
%
%
%
%
%

\subsection{Reeb component.}

First, we define a foliation of $\mathbb{R}^3$ illustrated in Figure \ref{R3}.\\
Note that this is the foliation of $\mathbb{R}^2$ of Figure \ref{R2} in each vertical plane containing the $z$-axis of $\mathbb{R}^3$.\\ 

Let
 $$f :\mathbb{R}^3\rightarrow\mathbb{R}$$

        \begin{center}
 $(x,y,z)\mapsto(x^2+y^2-1)\times\exp(z)$
\end{center}

$f$ is a  submersion, and defines a foliation $\mathcal{F}$ of $\mathbb{R}^3$, symmetric about the $z$-axis, where~:

\begin{itemize}
\item$f^{-1}(\lbrace 0 \rbrace)$ is a vertical cylinder leaf $C$ centered in $ 0$ with radius $1$.\\

\item $f^{-1}(\lbrace c^2 \rbrace)$ are leaves homeomorphic to a cylinder, because\\ $x^2+y^2=1+c^2\exp(-z)$, hence $x^2+y^2 > 1$. \\
When $z\rightarrow+\infty$, $x^2+y^2 \rightarrow 1$ so the leaves tend toward $C$. \\
When $z\rightarrow-\infty$, $x^2+y^2 \rightarrow +\infty$, i.e. the base of the cylinder is flaring.\\

\item $f^{-1}(\lbrace -c^2 \rbrace)$ are paraboloid leaves which intersect the $z$-axis for\\ $z=2\log (c)$, when $x^2+y^2 < 1$, . \\
 When $z\rightarrow+\infty$, $x^2+y^2 \rightarrow 1$, so the leaves tend toward $C$. \\
\end{itemize}

  \begin{figure*}[htb!]
\includegraphics[width=7cm]{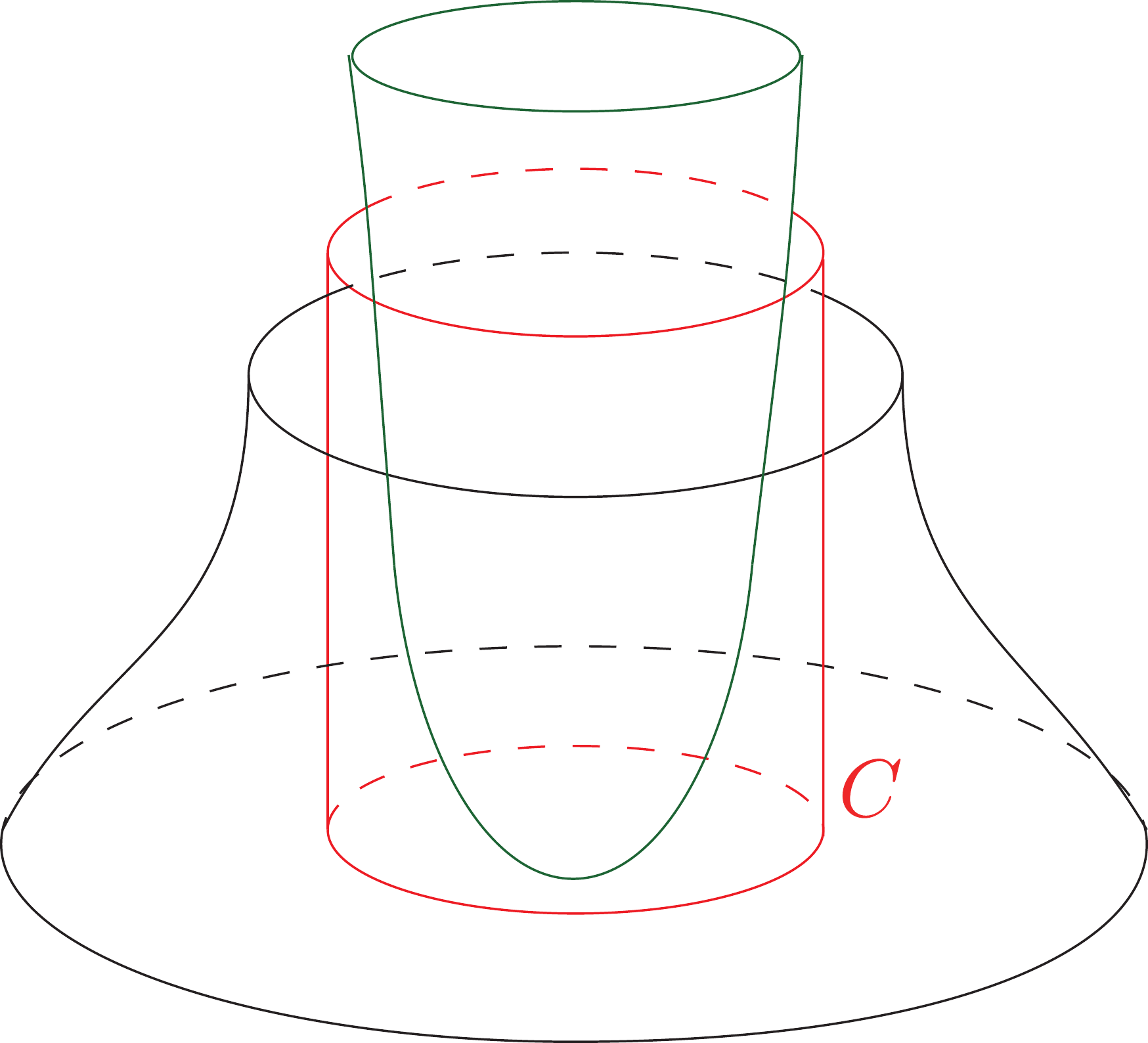}
\caption{Foliation $\mathcal{F}$ of $\mathbb{R}^3$}\label{R3}
\end{figure*}

 Let $\mathcal{F}$ be the restricted foliation on a vertical solid cylinder  $\mathbb{D}^2\times \mathbb{R}$, of radius $r\geq 1$, included in $\mathbb{R}^3$, denoted $C_r$. 
  $\mathcal{F}$ is invariant under integral vertical translations (along the $z$-axis); hence it induces a foliation on the solid torus $\mathbb{D}^2\times \mathbb{S}^1$ denoted $T_r$, of radius $r$.\\
  
  $T_r$ contains $T_1$ which is a solid torus of radius $1$ ($r\geq 1$), and a \textbf{\textit{Reeb component}} is the induced foliation by $\mathcal{F}$ on $T_1$, see Figures \ref{fig:4} and \ref{Reeb}.\\

 \begin{figure*}[htb!]
\includegraphics[width=7cm]{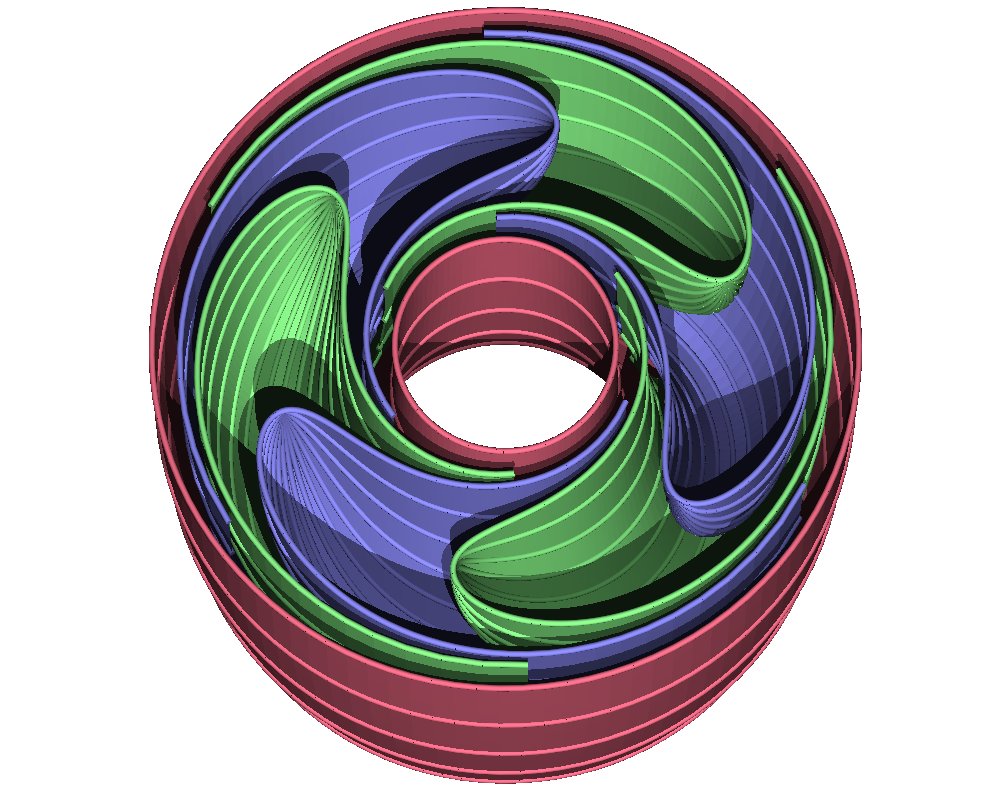}
\caption{(Half) Reeb component (from Wikipedia).}\label{fig:4}
\end{figure*}
Note that a Reeb annulus correspond to a $2$-dimensional Reeb component.\\

\begin{DEF}
Let $\mathcal{F}$ be a foliation of a $3$-manifold $M$.  $\mathcal{F}$ is \textbf{\emph{Reebless}} if it does not admit any Reeb component.
\end{DEF}

\subsection{ Turbulization.}
In this subsection we define by two interesting different ways the \textit{ turbulization component } denoted for all the following by $\T$. We talk about \textit{turbulization} when one torus boundary is foliated by circles. So we first need the following definition.

\begin{DEF}
A \textbf{\emph{circle foliation}} on a torus, (respectively on an annulus), is a foliation of $T^{2}$ (respectively of  $\S^{1}\t I$) where all the leaves are homeomorphic to $\S^{1}$. Hence, the leaves are parallel copies of an essential simple closed curve on $T^{2}$ (respectively on $\S^{1}\t I$).
\end{DEF}

\begin{DEF}
 Let us call \textbf{$\mathcal{T}$} the foliation induced by $\mathcal{F}$ on $T_r\backslash \mathring{T_1}$, for $r>1$ (or equivalently on $T_1\backslash \mathring{T_r}$, for $r<1$). The resulting foliated manifold is homeomorphic to $T^2\times I$, with a torus boundary leaf ($\partial T_{1}$), and another torus boundary component, transverse to the foliation $\mathcal{T}$, which induces a circle foliation on it, as in Figure \ref{fig:5}.\\
 \end{DEF} 
 \begin{DEF}
The foliation $\mathcal{T}$ is trivially transversely oriented. We will denote $\mathcal{T}^+$ (respectively $\mathcal{T}^-)$, the transversely oriented foliation of $T^2\times I$, where the transverse orientation on the torus leaf points out of (respectively into) $T^2\times I$.
\end{DEF}
  
  \begin{figure}[htb!]
\includegraphics[width=9cm]{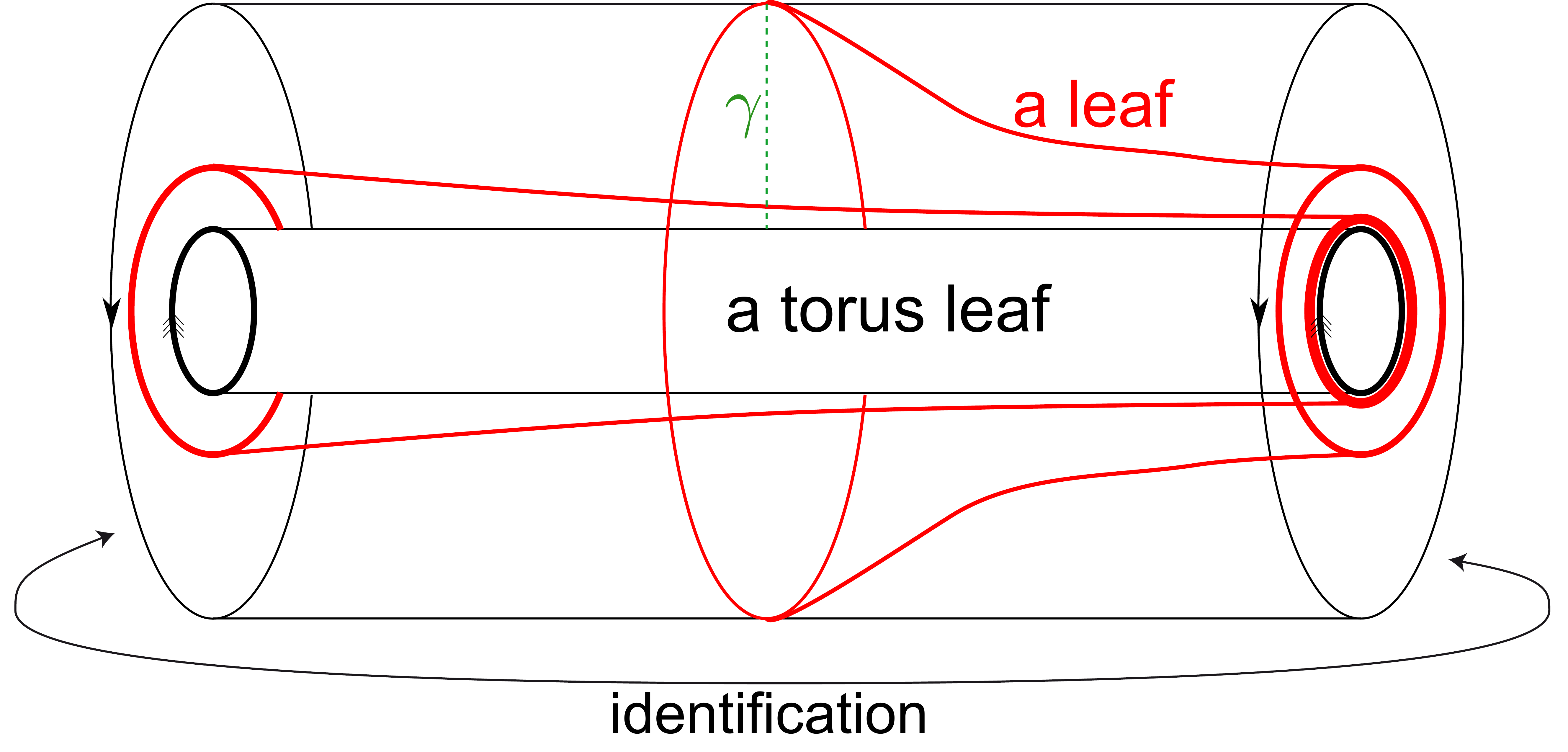}
\caption{Foliation $\mathcal{T}$ of  $T^{2}\times I$~: Turbulization}\label{fig:5}
\end{figure}
  
  \begin{DEF}
  Let $M$ be a manifold with a torus boundary component $T$ and admitting a foliation which induces on $T$ a circle foliation.\\
The process of \textbf{\emph{turbulization}} consists on pasting on $T$ (by homeomorphism) a $T^{2}\times I$ component, foliated by $\mathcal{T}$ (with the notations above).
\end{DEF}

Roughly speaking, the process of turbulization, changes a circle foliation on a torus to a torus leaf, as in the trivial following example in Figure \ref{Reeb}.\\

\begin{figure}[htb!]
\includegraphics[width=8cm]{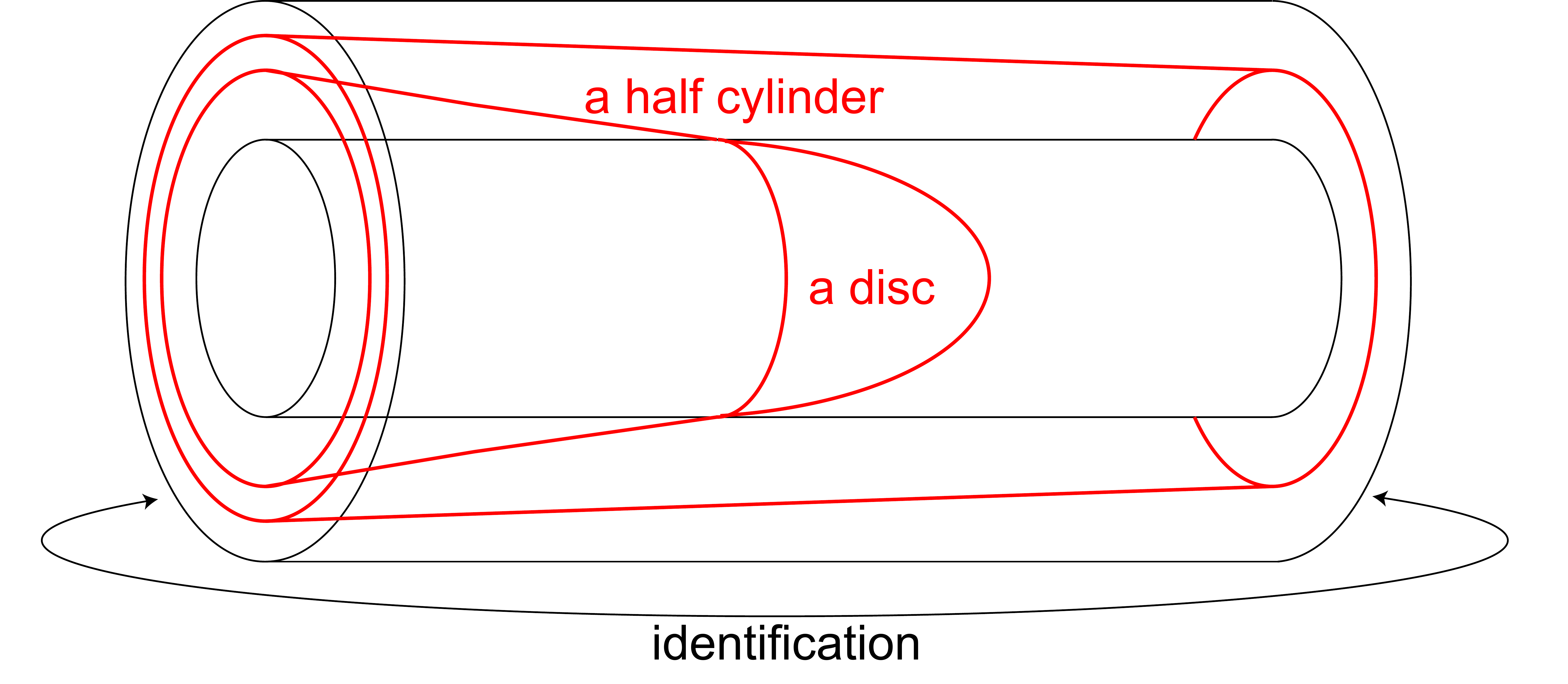}
\caption{Reeb component}\label{Reeb}
\end{figure}

\begin{RK}\label{Reeb-turb}
A Reeb component contains the foliation $\mathcal{T}$.
\end{RK}

Let us give another definition, that we will also use later.\\
Let $A=\{(x,\theta), x\in [0,1], \theta \in ]-\pi,\pi]\}$ be an annulus embedded in $\R^3$, and consider the following foliation called $\F$ on $A\times I$.\\
Let $f$ be a diffeomorphism of the unit interval such that $\{0\}$ and $\{1\}$ are fixed point, and $f$ is strictly increasing.\\
Denote for each $x\in [0,1]$, and $z\in I$, the circle $\lambda^{z}_{x}=\{(x,\theta,z), \theta \in ]-\pi,\pi]\}$ in $A\t I$.\\
Let us define the foliation $\F_{f}$.\\
The leaves of $\F_{f}$ are the annuli $A_{\alpha}$ bounded by $\lambda^{0}_{\alpha}$ and $\lambda^{1}_{f(\alpha)}$ in $A\times I$, for each $\alpha \in [0,1]$, as in Figure \ref{susp} where we have chosen $f(t)>t$. \\
This foliation $\F_{f}$ is called the \textbf{\textit{suspension foliation}} of $f$ along $\lambda_{0}^{0}$ on $A\t I$.\\
More precisely $\dps A_{\alpha}=\bigcup _{z\in [0,1]} \lambda_{[f(\alpha)-\alpha]z+\alpha}^{z}$.\\
Indeed the segment joining $\alpha$ to $f(\alpha)$ for a fixed angle $\theta \in ]-\pi,\pi]\}$, with the chosen coordinates, is defined by the equation $x=[f(\alpha)-\alpha]z+\alpha$ in $A\times I$.\\
Obviously $A_{0}$ and $A_{1}$ are vertical leaves.\\

\begin{figure}[htb!]
\includegraphics[width=12cm]{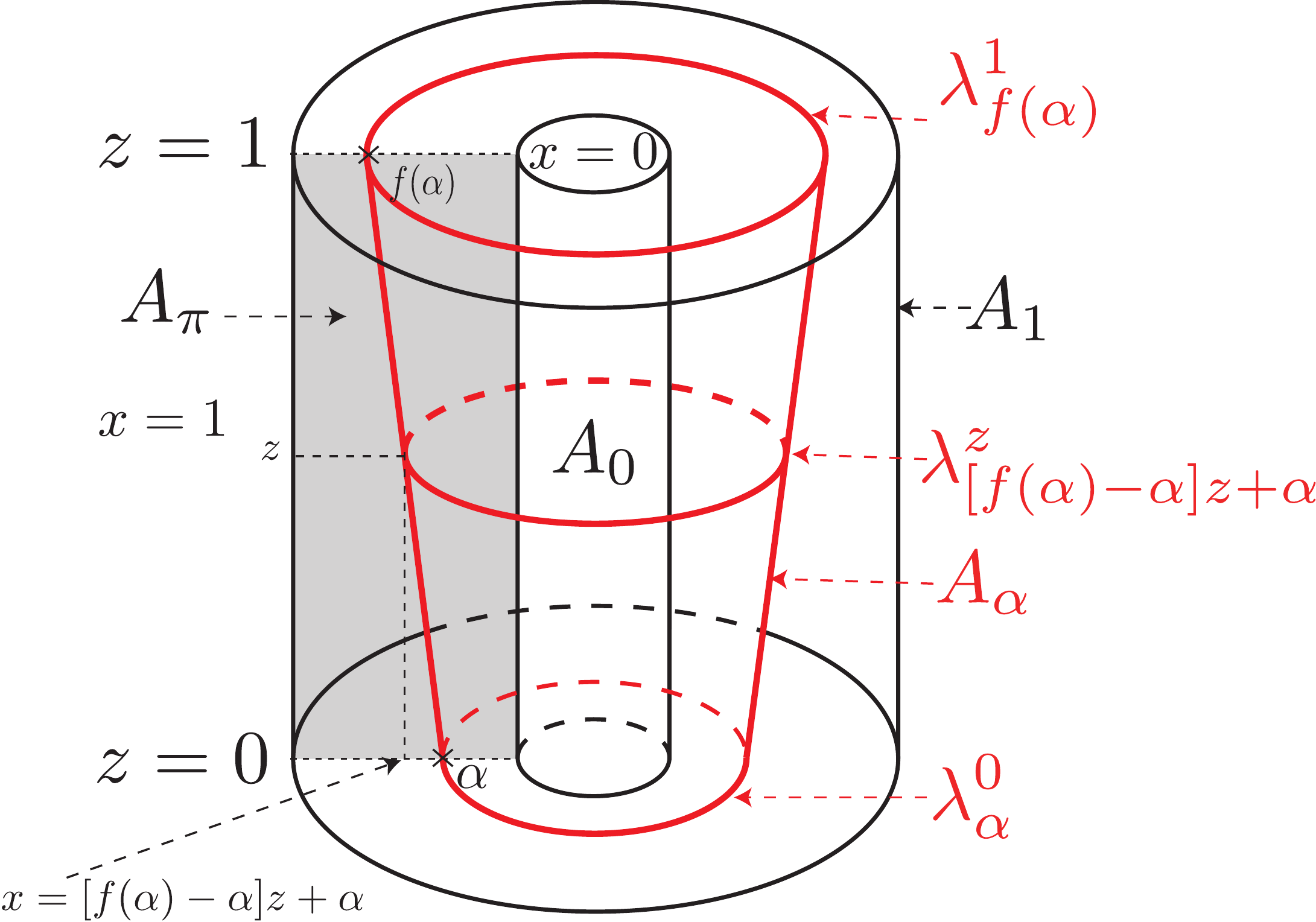}
\caption{Foliation $\F_{f}$ of $A\times I$}\label{susp}
\end{figure}

That leads us to construct a foliation on $T^{2}\times I$ as follows:\\
Consider $T^{2}\times I= (A\times I)/\!\raisebox{-.65ex}{\ensuremath{\sim}}$, where $(x,\theta,0)\sim (x,\theta,1)$.\\
$\F_{f}$ induces a foliation on $T^{2}\times I$ where $T^{2}\times \{0\}$ and $T^{2}\times \{1\}$ are torus leaves.\\
Now, if we choose  $f$ so that $f(t)>t$ or $f(t)<t$, for all $t\in \mathring{I}$, the foliation $\mathcal{T}$ is isotopic to the induced foliation by $\F_{f}$ on $\{(x,\theta,z), x\in [0,\frac{1}{2}], \theta \in ]-\pi,\pi], z\in [0,1]\}/\!\raisebox{-.65ex}{\ensuremath{\sim}}$, or equivalently on $\{(x,\theta,z), x\in [\frac{1}{2},1], \theta \in ]-\pi,\pi], z\in [0,1]\}/\!\raisebox{-.65ex}{\ensuremath{\sim}}$.\\
Indeed the torus $\{(\frac{1}{2},\theta,z), \theta \in ]-\pi,\pi], z\in [0,1]\}/\!\raisebox{-.65ex}{\ensuremath{\sim}}$ is everywhere transverse and admits a circle foliation.\\
Note that if $\exists t_{0} \in \mathring{I} / f(t_{0})=t_{0}$ that induces an interior torus leaf in $\F_{f}$.

\begin{RK}\label{A-pi-turb}
Note that the induced foliation by $\F_{f}$ on the transverse annulus $A_{\pi}=\{(x,\pi,z), x\in [0,1], z\in [0,1]\}/\!\raisebox{-.65ex}{\ensuremath{\sim}}$ in $T^{2}\t I$ is described on Figure~\ref{susp-ann}. 
\end{RK}

\begin{RK}
Note that the choice $f(t)>t$ or $f(t)<t$ does not imply that there is the same direction of rotation along the torus leaf $T=\{(1,\theta,z), \theta \in ]-\pi,\pi], z\in [0,1]\}/\!\raisebox{-.65ex}{\ensuremath{\sim}}$. That is the reason why we give the following definition.
\end{RK}

  \begin{figure*}[htb!]
\begin{minipage}[b]{0.48\linewidth}
\centering
\centerline{\includegraphics[width=6cm]{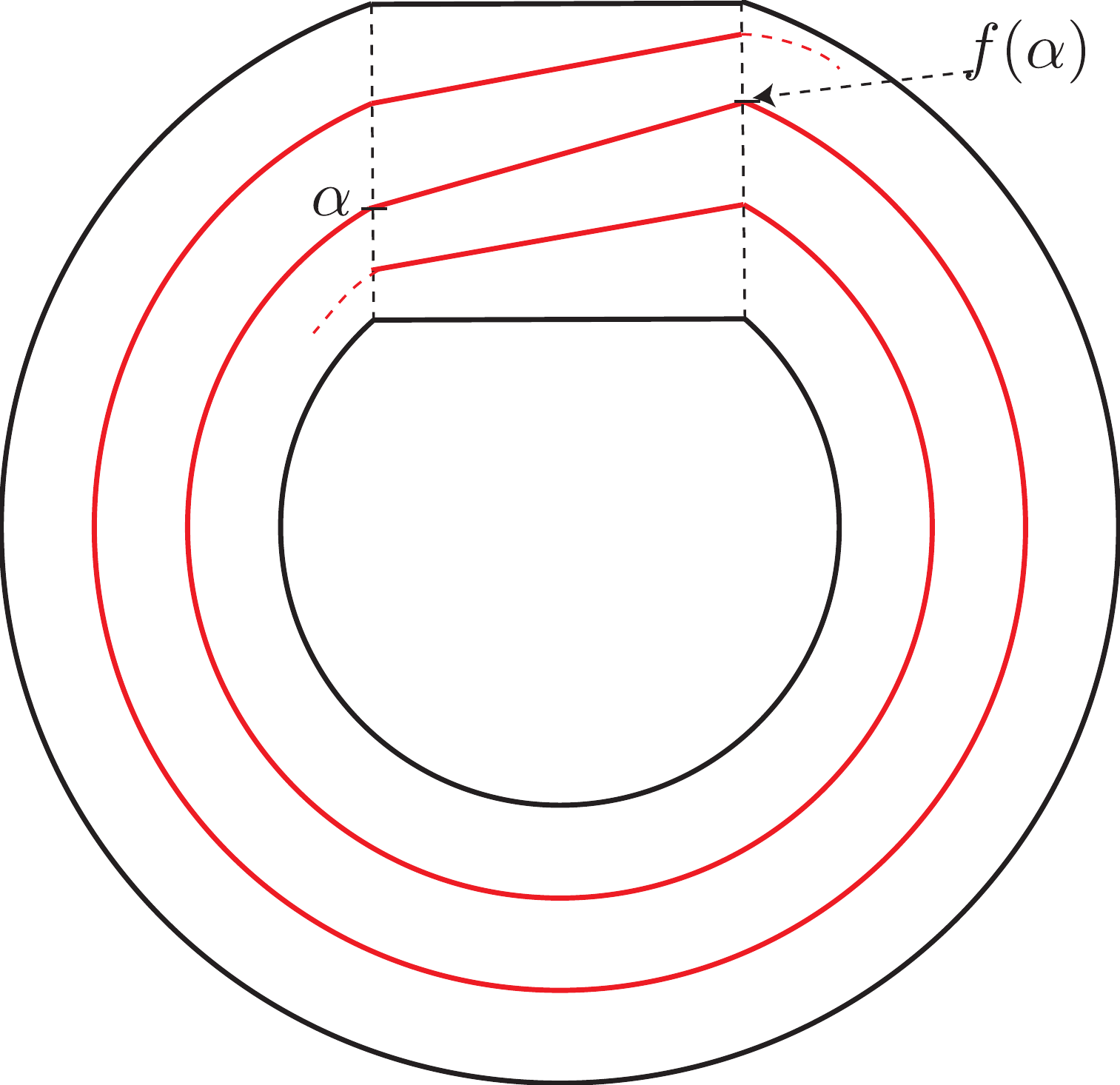}}
\centerline{\footnotesize{(a) }}
\end{minipage}
\hfill
\begin{minipage}[b]{0.48\linewidth}
\centering
\centerline{\includegraphics[width=6cm]{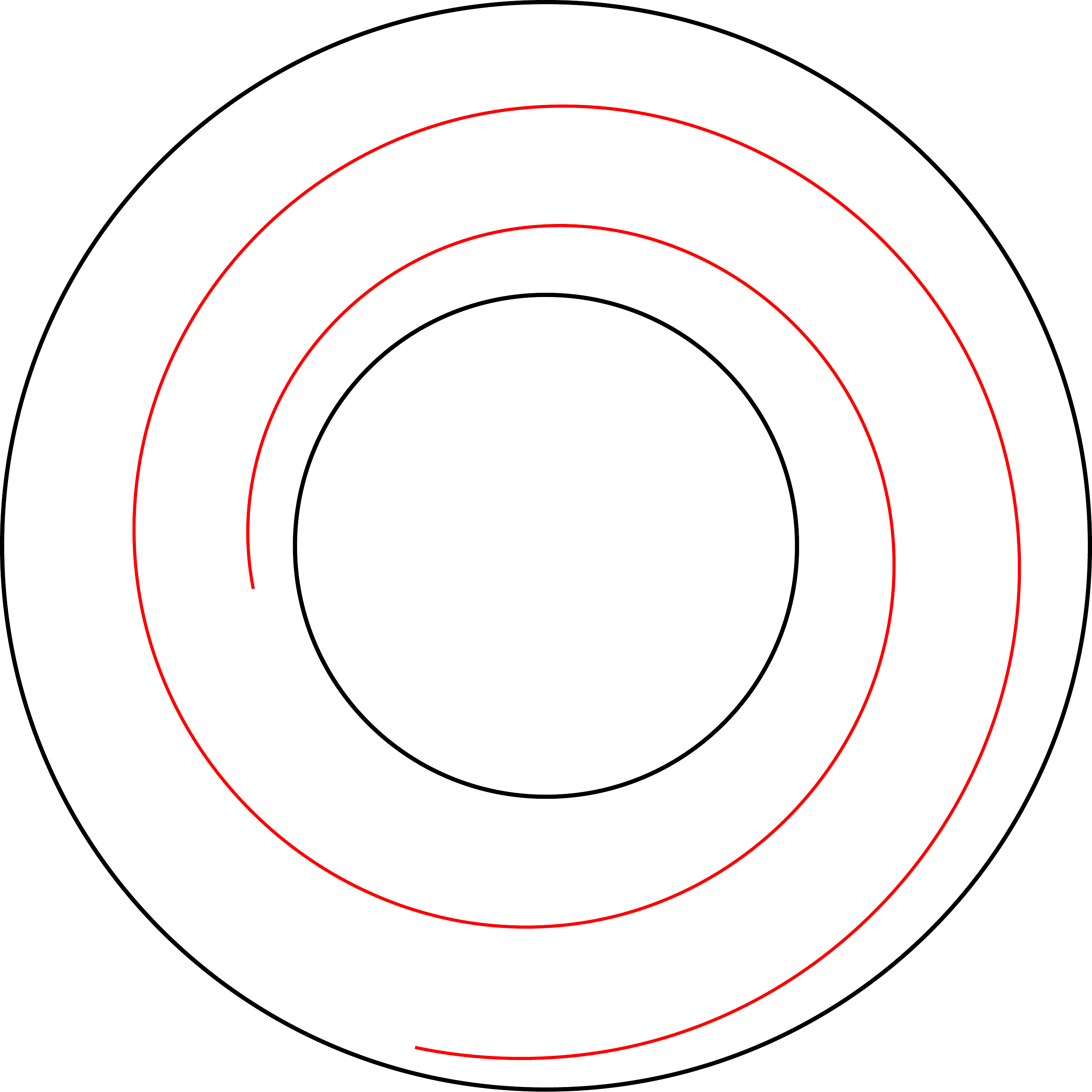}}
\centerline{\footnotesize{(b) }}
\end{minipage}
\caption{ Induced foliation by $\mathcal{F}$ on $A_{\pi}$ (isotopic representations)}
\label{susp-ann}
\end{figure*}

\begin{DEF}
If $f(t)>t$ we say that $\F_{f}$ is a clockwise foliation, and if $f(t)<t$, we say that $\F_{f}$ is a anti-clockwise foliation.
\end{DEF}

\begin{RK}
Note that if $\F_{f}$ is a clockwise (respectively anti-clockwise) foliation then the induced foliation on $A_{\pi}$ is a clockwise (respectively anti-clockwise) spiral foliation.
\end{RK}

Turbulization has a lot of applications; one of the most famous is the following Theorem from W.B.R \cite{Li} and also showed independently by S.P \cite{NZ} (helped by H. Zieschang)~:

\begin{THM}
Every $3$-manifolds admit a codimension one foliation, possibly with a Reeb component.
\end{THM}

\begin{proof}(idea)
We may recall that every closed $3$-manifold $M$ is obtained by deleting a tubular neighborhood of a link $L$ in $\mathbb{S}^3$, and by gluing it back differently. Let us consider $\mathbb{S}^3$ foliated by two Reeb's component glued along their torus leaf. We can isotope $L$ so that it meets transversely the leaves of this foliation. By choosing a thin enough tubular neighborhood of $L$, denoted by $N(L)$, we may assume that the induced foliation on $N(L)$ is the one by disks transverse to the boundary of $N(L)$. Then we can remove (the interior of) $N(L)$,  and glue some $\mathcal{T}$ components along each boundary component of $\mathbb{S}^3 \backslash \mathring{N(L)}$, i.e we apply the process of turbulization. Then we obtain a manifold with torus boundary leaves, and it remains to glue Reeb's component along those boundary leaves by the well chosen way, to obtain $M$ with a foliation (with Reeb components).
\end{proof}

\subsection{Generalized turbulization}
Turbulization can be defined in a more general context.\\
The idea of turbulization is to extend a foliation of $T^{2}\times \{1\}$ (either by circle or dense on the torus) in $T^{2}\times I$ to obtain $T^{2}\times \{0\}$ as a torus leaf.\\
In the preceding paragraph we have done it for a circle foliation on $T^{2}\times \{1\}$, here we want to do a similar construction for a dense foliation on $T^{2}\times \{1\}$.
Indeed, when the foliation on a torus is $\C^{2}$, A. \cite{De} and C.L. \cite{Si} showed that either there is a circle leaf or the foliation is dense. When there is a circle leaf there are two cases, either this is a circle foliation, or there are spiral leaves between circles leaves. The last case will be taken in account with spiraling, while the former case has already been studied.\\

Let us formulate it more precisely.\\

One way to define generalized turbulization is as follows~:\\
Consider $A\times I= \{(x,\theta,z), x\in [0,1], \theta \in ]-\pi,\pi], z\in [0,1]\}$, ($A$ is an annulus).\\
For each $z\in [0,1]$ set $A_{z}= \{(x,\theta,z), x\in [0,1], \theta \in ]-\pi,\pi]\}$.\\
Now foliate each $A_{z}$ by the foliation $\mathcal{C}$ (of Figure \ref{C}), to obtain a foliation on $A\times I$ by $\mathcal{C}\times I$.\\

This foliation is invariant by rotation along the $z$-axis; hence that induces a foliation of $T^{2}\times I$ by identifying $A_{0}$ to $A_{1}$ by a foliation preserving homeomorphism $f$ from $A_{0}$ to $A_{1}$ such that $f(\{(0,\theta,0),\theta\in]-\pi,\pi]\})=\{(0,\theta,1),\theta\in]-\pi,\pi]\}$, and $f$ sends a spiral leaf on a spiral leaf, for example any rotation.\\ 
Consider $T^{2}\times I\cong (A\times I)/\!\raisebox{-.65ex}{\ensuremath{\sim}}$  where $((x,\theta),0)\sim (f(x,\theta),1)$.\\
Note that here $T^{2}\times \{x\}\cong  \{(x,\theta,z), \theta \in ]-\pi,\pi], z\in [0,1]\}/\!\raisebox{-.65ex}{\ensuremath{\sim}}$, for each $ x\in [0,1]$.\\
\begin{DEF}
We call this foliated component a generalized turbulization component, and we denote it by $\mathcal{T}_{*}(f)$, as illustrated in Figure \ref{T-*}.\\
\end{DEF}

\begin{figure}[htb!]
\includegraphics[width=7cm]{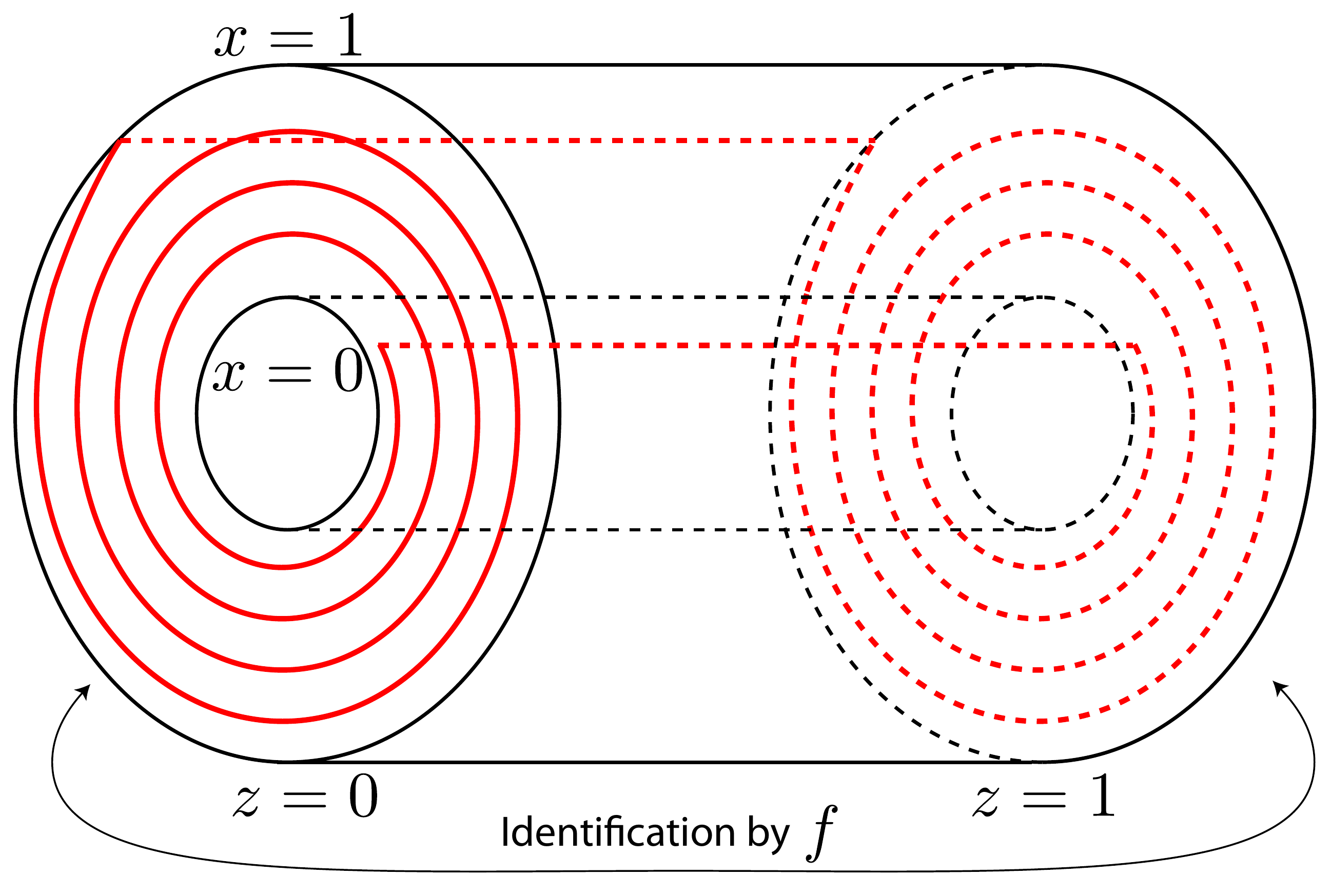}
\caption{Generalized turbulization~: foliation $\mathcal{T}_{*}(f)$}\label{T-*}
\end{figure}

There are two crucial examples :\\
When $f$ is a rotation rational rational angle, the non-compact leaves are homeomorphic to $\R^{+}\times \S^{1}$, and there is one compact leaf~: the torus $T^{2}\times \{0\}$. Moreover the leaves of the circle foliation on $T^{2}\times \{1\}$ have rational slopes (recall that those are essential simple closed curves).\\
Then, note that $\mathcal{T}_{*}(f)$ (of Figure \ref{T-*}) and $\mathcal{T}$ (of Figure \ref{fig:5}) are homeomorphic.
When $f$ is a rotation of irrational angle, there is no cylinder leaf, i.e all the non-compact leaves are homeomorphic to $\R^{+}\times \R$, and the induced foliation on $T^{2}\times \{1\}$ is dense (corresponds to irrational slopes).\\

\section{Spiraling} \label{spir}

Turbulization extends a circle foliation or a dense foliation on a torus $T^{2}\t \{0\}$ in a foliation of the $3$-manifold $T^{2}\times I$ such that $T^{2}\t \{1\}$ is a leaf (turbulization).\\
The goal of spiraling (here) is to extend a foliation on the torus by spirals and circles.\\

For this we use the construction of D. \citeauthor{Ga}, who defines spiraling in a more general context which is the following.\\
Let $S_{g}$ be a closed orientable compact genus $g\geq1$ surface. We start with a foliation $\F$ on $S_{g}\times \{0\}$ which has a two dimensional leaf and an annulus with a one dimensional foliation tangent to its boundary.\\
Then, we extend it to a foliation of $S_{g}\times I$, where $S_{g}\times \{1\}$ is a leaf.\\

Here we explain this construction providing more details and generalize it when $g=1$.\\

We first construct a foliation of $S_{g}\times I$ that we will call \textit{spiraling component}. We proceed into four steps which are subsections.\\
\begin {itemize}
\item Step 1 : Notations and conventions, we fix $\delta$ and $\lambda$ two essential simple closed curve on $S_{g}$ such that $\#(\lambda \cap \delta)=1$.
\item Step 2 : Suspension foliation along $\lambda$. 
\item Step 3 : Superposition along $\delta$.
\item Step 4 : Repeating infinitely many times Step $3$ (infinite induction).\\
\end{itemize} 

Subsection \ref{attach} applies this construction to extend a foliation $\F$ of a manifold $M$ with a boundary component $S_{g}$, such that $\F_{|S_{g}}$ admits a transverse annulus to a foliation of $M\cup (S_{g}\t I)$ such that $S_{g}=S_{g}\t \{0\}$ and $S_{g}\t\{1\}$ is a leaf.

Finally, Subsection \ref{gen-spi} generalizes spiraling when $g=1$ and when $S_{g}\t \{0\}$ admits Reeb annulus.
\subsection{Step 1}

We consider a closed compact surface of genus $g\geq 1$, denoted $S_{g}$, with a non-separating simple closed curve $\delta$ embedded in $S_{g}$. We set $A_{\delta}=\delta\times I$ a regular neighborhood of $\delta$ in $S_{g}$ and we make the confusion between $\delta$ and $\delta\times \{0\}$.\\
Let $\lambda$ be simple closed curve embedded in $S_{g}$ whose geometric intersection number with $\delta$ is one. Note that it induces that $\lambda$ is non-separating. Similarly we will denote $A_{\lambda}=\lambda\times I$ a regular neighborhood of $\lambda$ in $S_{g}$ and we make the confusion between $\lambda$ and $\lambda\times \{0\}$.\\

\begin{figure}[ht!]
\includegraphics[width=10cm]{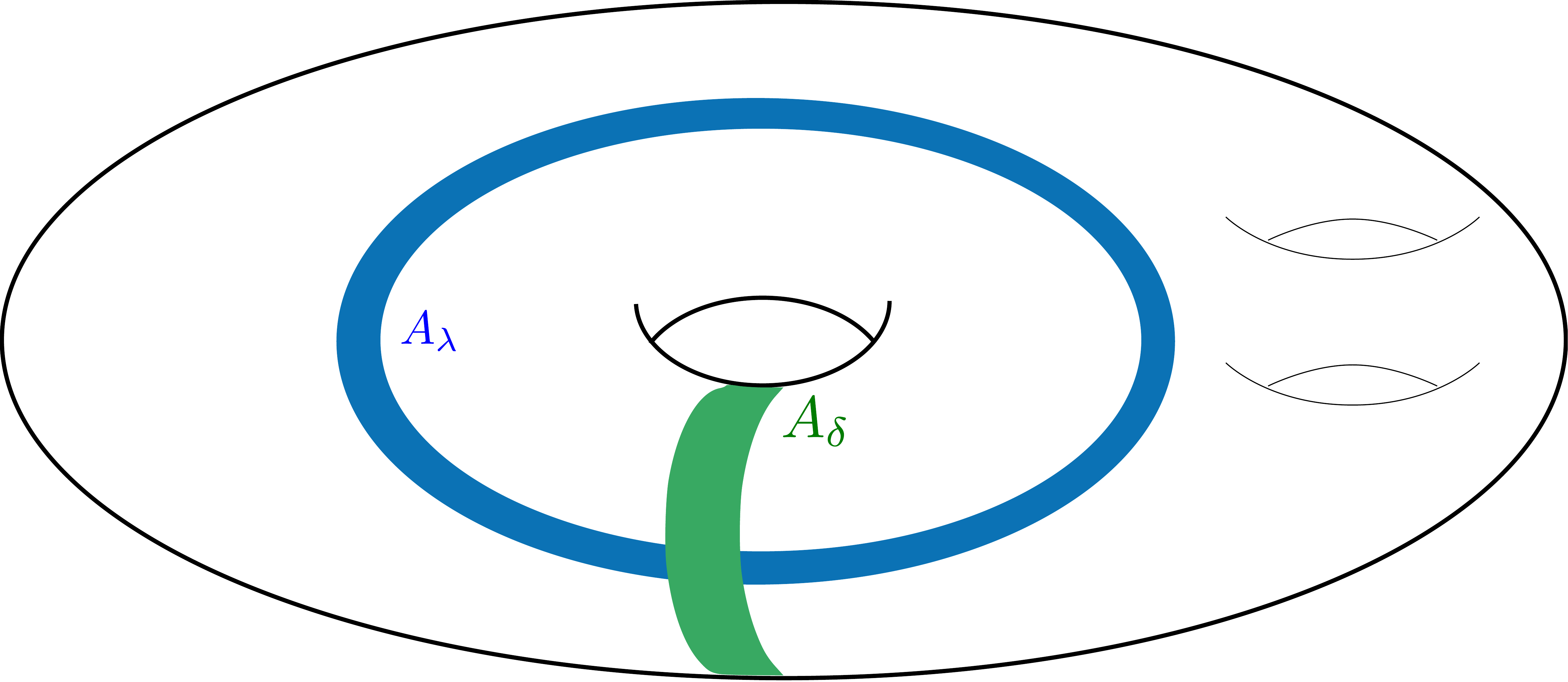}
\caption{A choice of $\lambda$ and $\delta$.}\label{lambda}
\end{figure}

Consider the product foliation on $(S_{g}\bsl A_{\lambda})\t I$ and denote the leaves by $Q_{t}=(S_{g}\bsl A_{\lambda})\t \{t\}, t\in I$.

\subsection{Step 2}

Here we construct a particular foliation of $S_{g}\t I$, where the two boundary components are leaves, and the interior leaves are non-compact.\\

Let $f$ be a strictly increasing diffeomorphism of $I$ such that $f(0)=0$ and $f(1)=1$.
 Consider the suspension foliation of $f$ along $\lambda\t\{0\} \t \{0\}$, in $\lambda\t\{0\} \t I$, where the leaves are the annuli cobounded by $\lambda \t \{0\} \t \{t\}$ and $\lambda \t \{1\} \t \{f(t)\}$ in $A_{\lambda}\t I$, (see Figure \ref{susp}).
Now extend the product foliation on $S_{g}\t I$ by gluing this foliated component $A_{\lambda}\t I$ by the identity on $(S_{g}\bsl A_{\lambda})\t I$, and we denote the resulting foliation of $S_{g}\t I $ by $\F_{f}$.
To draw $\F_{f}$ more easily, we represent $A_{\lambda}$ differently in Figure \ref{A-lambda}.\\
Note that in all the following figures we have chosen $f(t)>t$.

\begin{figure}[ht!]
\includegraphics[width=14cm]{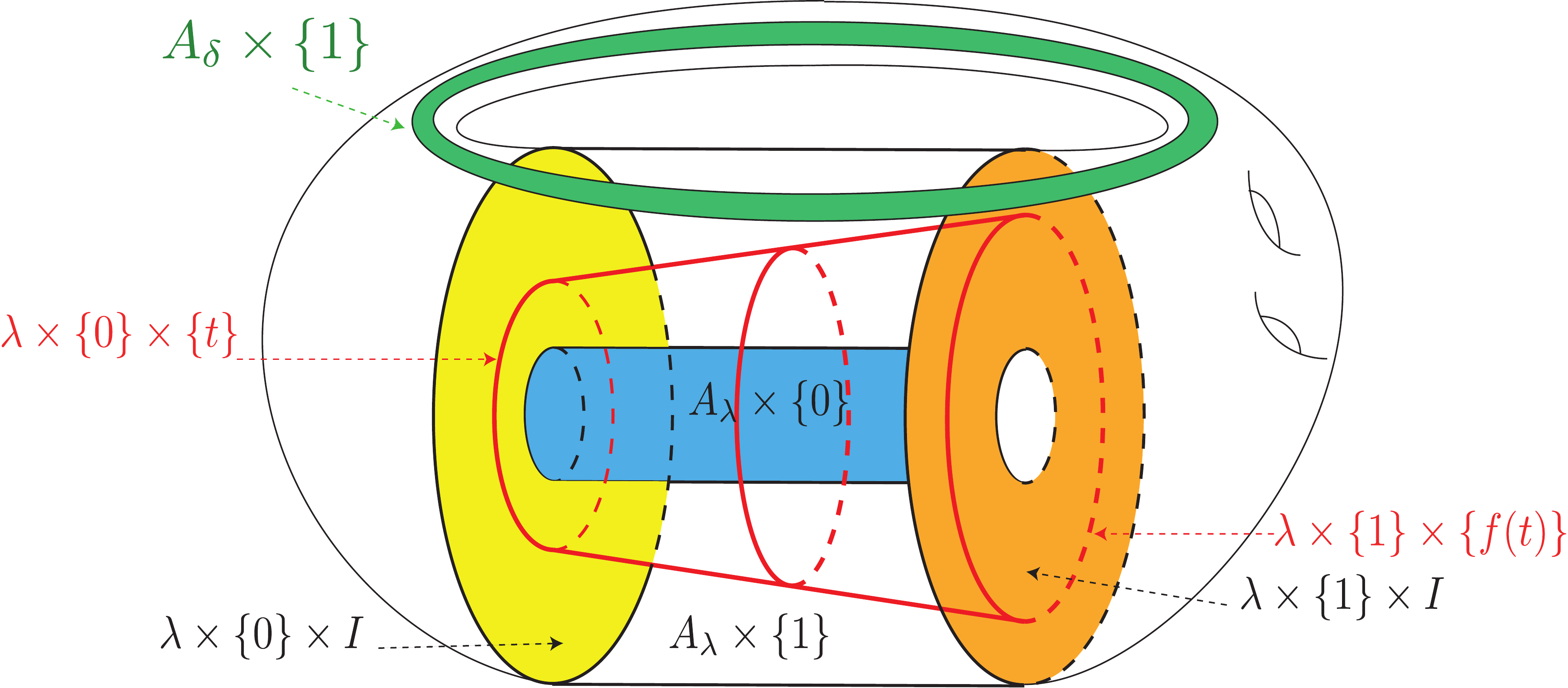}
\caption{Suspension foliation along $A_{\lambda}\t I$ in $S_{g}\t I$}\label{A-lambda}
\end{figure}

Note that for all $t\in \mathring{I}$, this extension adds to $Q_{t}$ two annuli : one is joining $Q_{t}$ to $Q_{f(t)} $ and the other annulus is joining $Q_{t}$ to $Q_{f^{-1}(t)}$.\\
Of course this extension adds an annulus to $Q_{t}$, for $t\in \{0,1\}$, which induces that $S_{g}\t \{0\}$ and $S_{g}\t \{1\}$ are leaves of $\F_{f}$, and the other leaves tends toward those two.\\

To represent $\F_f$, we can draw a transverse cut of this foliation, i.e along $\delta\t \{0\}\t I$, as in Figure \ref{transverse-A-delta}. Note that here, $\delta\t \{0\}\t I$ plays the role of $A_{\pi}$ of Figure \ref{susp}, see also Figure \ref{A-delta}.

\begin{figure}[ht!]
\includegraphics[width=10cm]{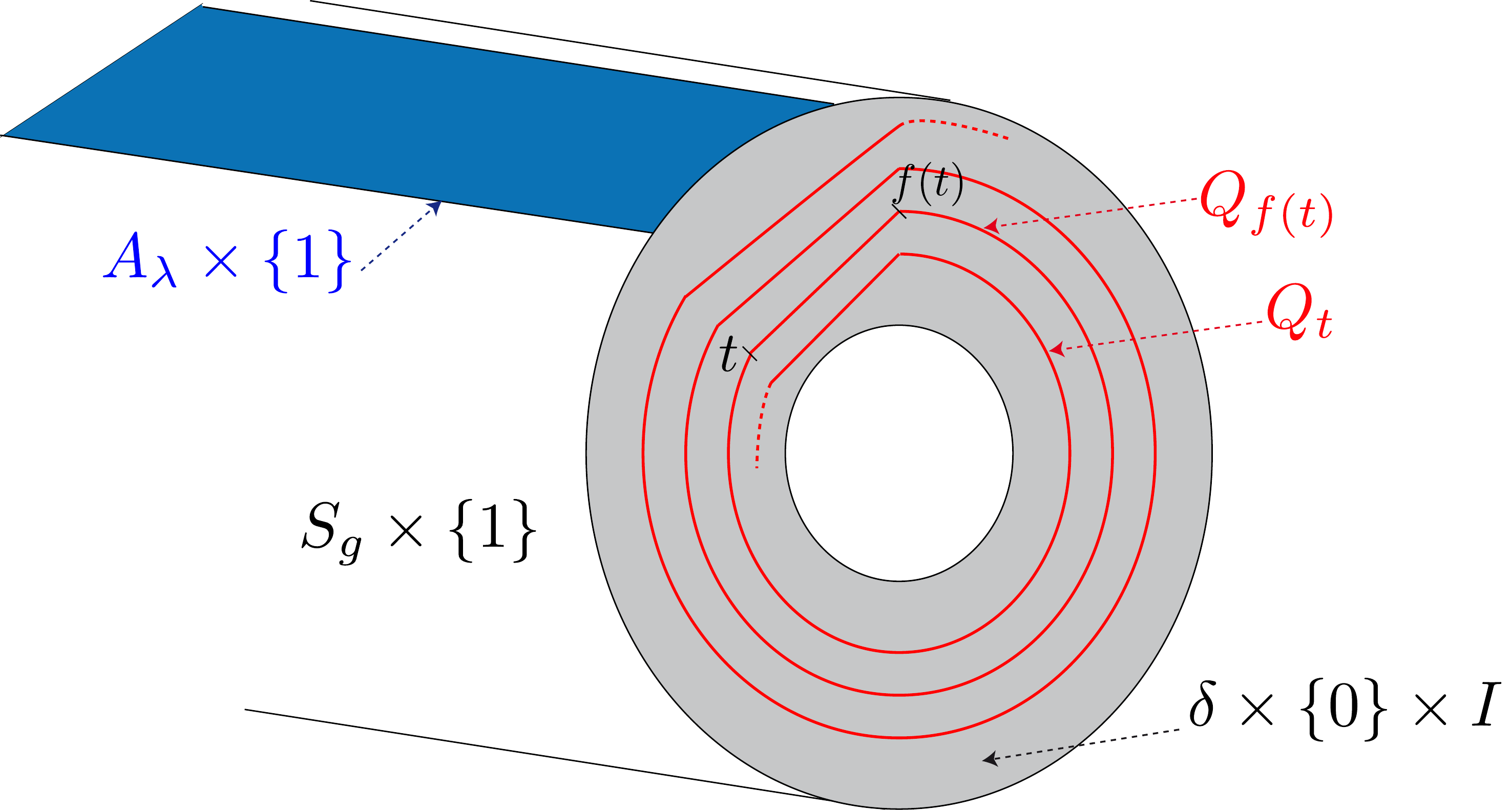}
\caption{Transverse cut along $\delta\t \{0\}\t I$}\label{transverse-A-delta}
\end{figure}
 
\begin{figure}[ht!]
\includegraphics[width=14cm]{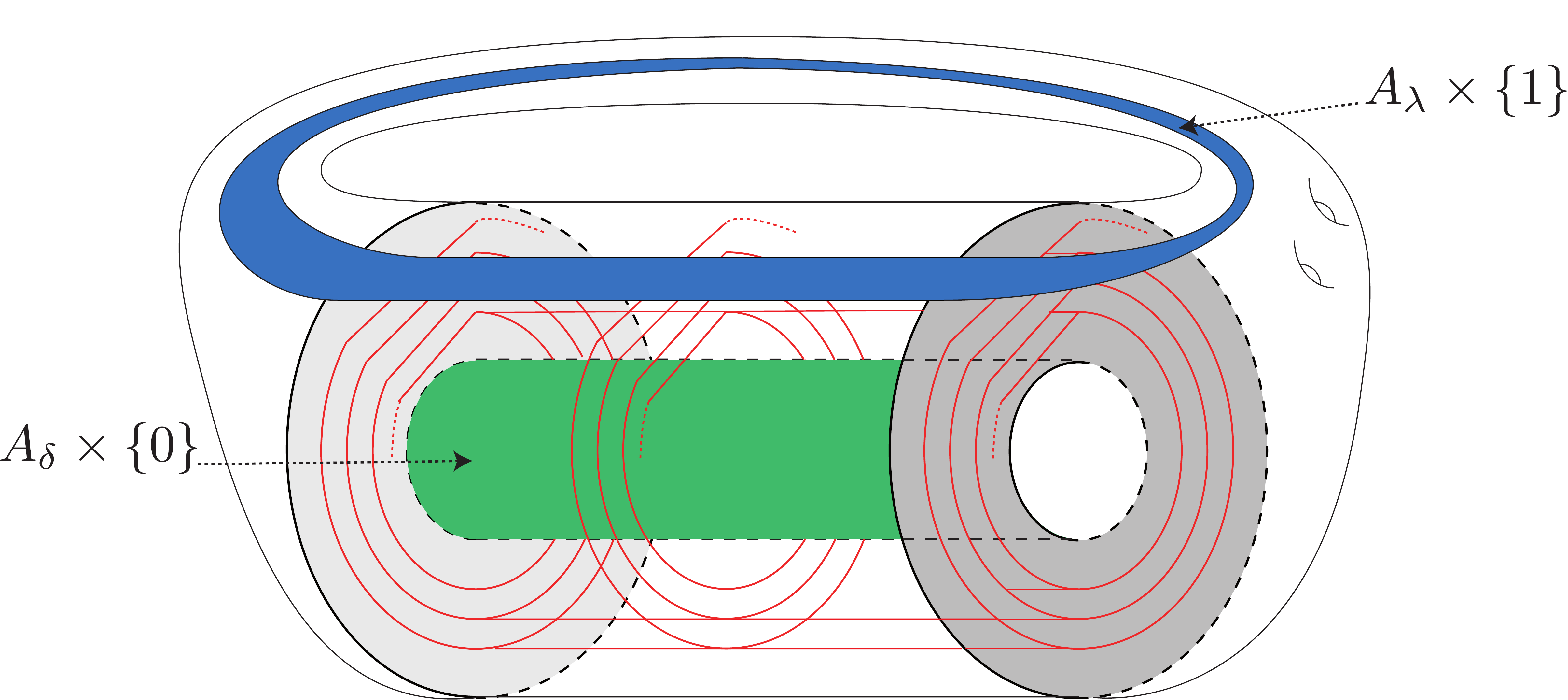}
\caption{Foliation $\F_f$ along $A_{\delta}$}\label{A-delta}
\end{figure}

\subsection{Step 3}

The goal of this step is to paste together $(S_{g}\bsl A_{\delta} )\t I_{0}$ and $(S_{g}\bsl A_{\delta} )\t I_{1}$, and to extend it nicely; where $I_{0}=[0,1/2]$, and $I_{1}=[1/2,3/4]$.\\

Consider $(S_{g}\bsl A_{\delta} )\t I_{0}$ and $(S_{g}\bsl A_{\delta} )\t I_{1}$ both with the foliation $\F_f$ of Figure \ref{A-delta}.\\
Glue them along $(S_{g}\bsl A_{\delta})\t \{1/2\}$ to obtain a manifold homeomorphic to $(S_{g}\bsl A_{\delta} )\t [0,3/4]$.\\
Now we do the following extension (see Figure \ref{super}).

\begin{figure}[ht!]
\includegraphics[width=12cm]{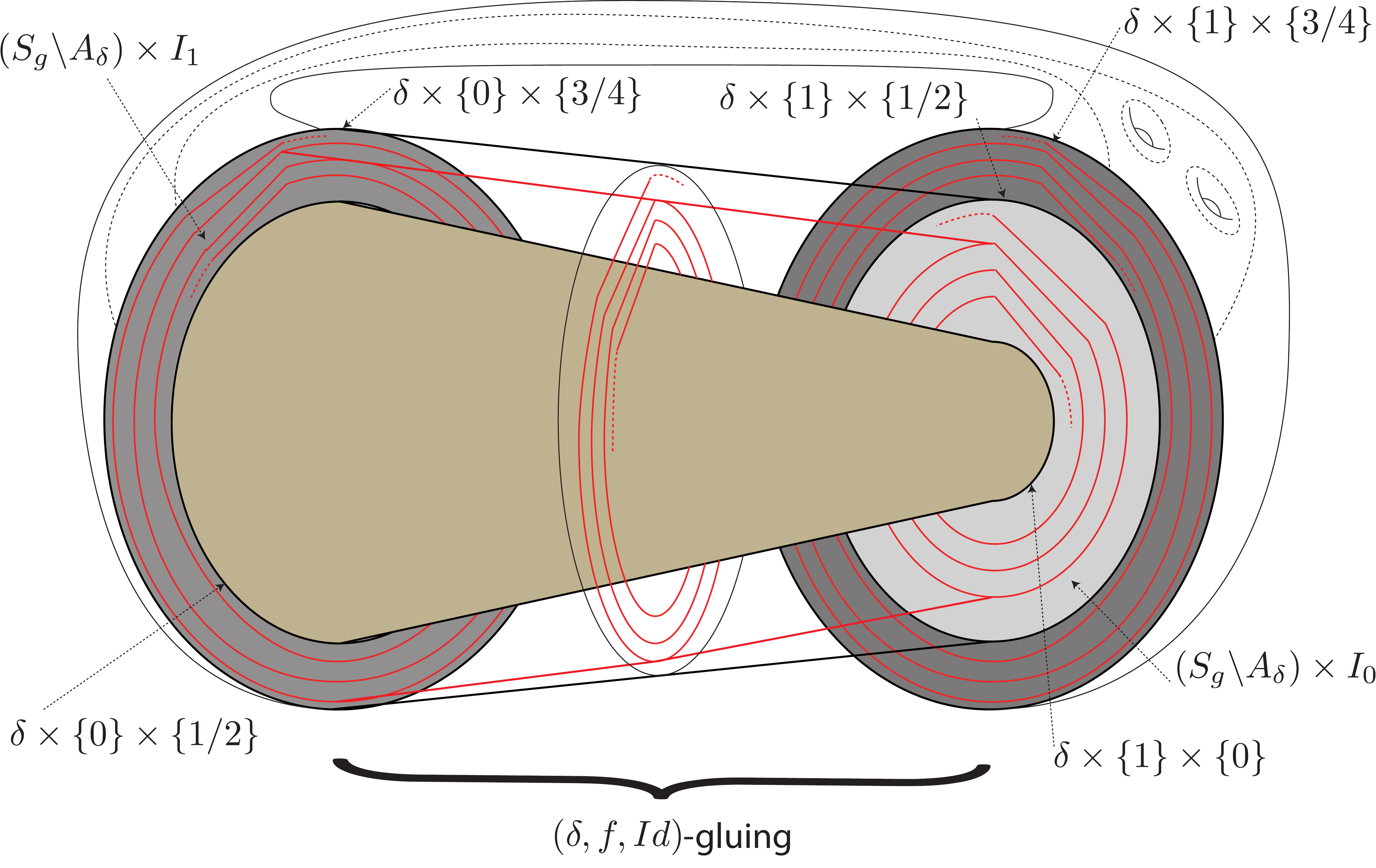}
\caption{($\delta,f,Id)$-gluing between $(S_{g}\bsl A_{\delta})\t I_{0}$ and $(S_{g}\bsl A_{\delta})\t I_{1}$}\label{super}
\end{figure}

We identify $\delta\t \{0\}\t I_{1}$ to $\delta\t \{1\}\t I_{0}$ by a foliation preserving homeomorphism $h$ (i.e sending an interior leaf on an interior leaf, and sending $\delta\t \{0\}\t \{3/4\}$ on $\delta\t \{1\}\t \{1/2\}$, and $\delta\t \{0\}\t \{1/2\}$ on $\delta\t \{1\}\t \{0\}$, for example any rotation).\\
In particular, this amounts to glue one annulus called $A_{\delta}^{3/4}$ between $\delta\t \{0\}\t \{3/4\}$ and $\delta\t \{1\}\t \{1/2\}$; and an other annulus $A_{\delta}^{1/2}$ between $\delta\t \{0\}\t \{1/2\}$ and $\delta\t \{1\}\t \{0\}$, which connects the compact leaves $(S_{g}\bsl A_{\delta})\t \{0\}$, $(S_{g}\bsl A_{\delta})\t \{1/2\}$ and $(S_{g}\bsl A_{\delta})\t \{3/4\}$.\\
There may exist annuli connecting the circles; which correspond to possible interior fixed points of $f$.\\
Moreover, this amounts to glue bands $\R\t I$ where one boundary spiral $\R\t \{0\}$ lies on $\delta\t \{0\}\t I_{1}$ and the other boundary spiral $\R\t \{1\}$ lies on $\delta\t \{1\}\t I_{0}$, so that the foliation matches.\\

We denote this extension by a \textbf{\textit{$(\delta,f,h)$-gluing}} between $(S_{g}\bsl A_{\delta})\t I_{0}$ and $(S_{g}\bsl A_{\delta})\t I_{1}$, and the foliation is called $\F(f,I_{0},I_{1},h)$.\\

Note that we can make another choice to make this extension. Indeed, we can also identify $\delta\t \{0\}\t I_{0}$ and $\delta\t \{1\}\t I_{1}$ by a foliation preserving homeomorphism $h$ similarly. That gives another direction of rotation along the boundary leaf.\\
If we do the first choice we call it a \textbf{\textit{clockwise $(\delta,f,h)$-gluing}} otherwise if we make the second choice we call it a \textbf{\textit{anti-clockwise $(\delta,f,h)$-gluing}}.
But for more simplicity when the direction of rotation does not matter we will just say a \textit{$(\delta,f,h)$-gluing}.

The boundary of the resulting manifold has two connected components as shown in Figure \ref{super-bord} :
\begin{itemize}
 \item $S_{g}^{0}=\delta\t \{0\}\t [0,1/2]\cup A_{\delta}^{1/2}\cup (S_{g}\bsl A_{\delta})\t \{0\}$; where $\delta\t \{0\}\t [0,1/2]$ is transverse and $A_{\delta}^{1/2}\cup (S_{g}\bsl A_{\delta})\t \{0\}$ is tangent to the foliation  $\F(f,I_{0},I_{1},h)$.
  \item $S_{g}^{3/4}=\delta\t \{1\}\t [1/2,3/4]\cup A_{\delta}^{3/4}\cup (S_{g}\bsl A_{\delta})\t \{3/4\}$; where $\delta\t \{1\}\t [1/2,3/4]$ is transverse and $A_{\delta}^{1/2}\cup (S_{g}\bsl A_{\delta})\t \{3/4\}$ is tangent to the foliation  $\F(f,I_{0},I_{1},h)$.

\end{itemize}

\begin{figure}[ht!]
\includegraphics[width=16cm]{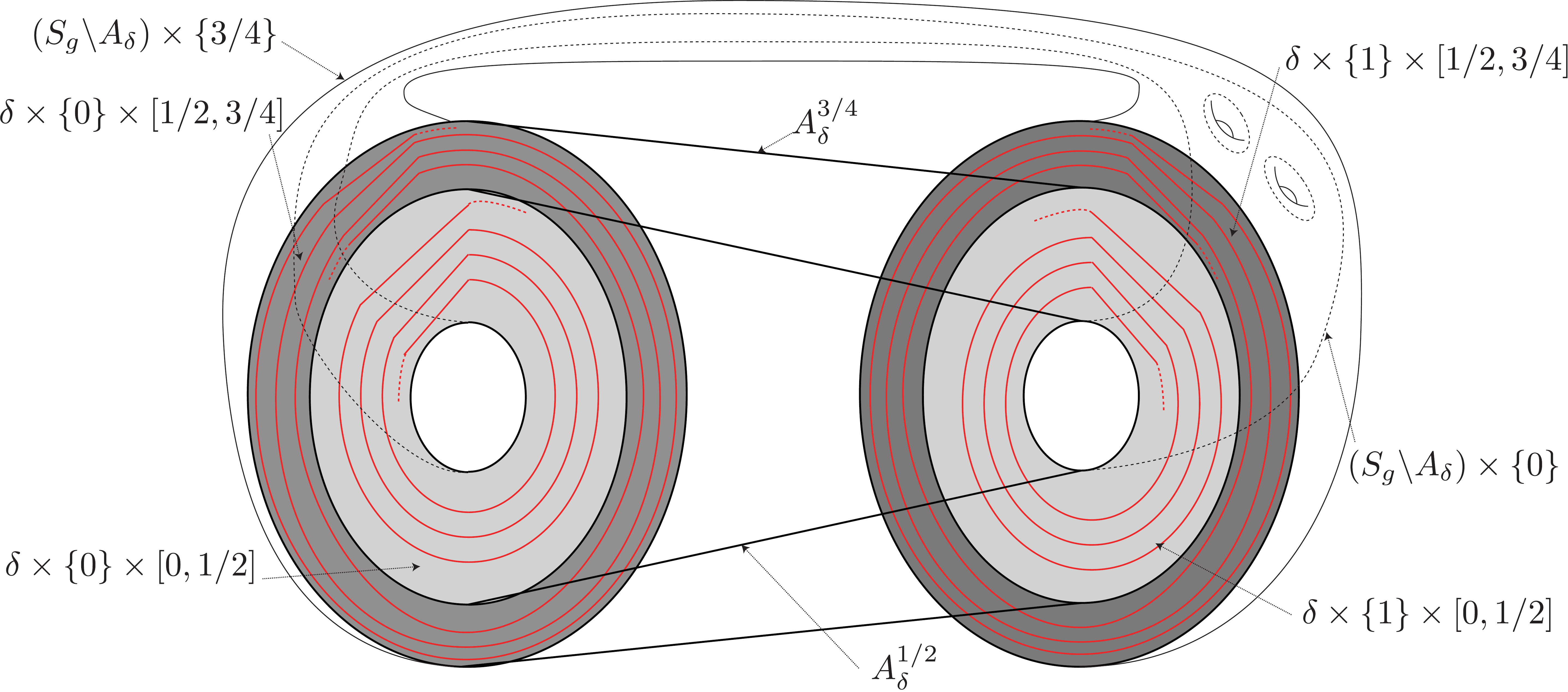}
\caption{Boundary of the foliation $\F(f,I_{0},I_{1},Id)$}\label{super-bord}
\end{figure}

\subsection{Step 4}

The aim of this step is to repeat infinitely many times Step $3$, in order to obtain a foliation of $S_{g}\t I$ where $S_{g}\t \{1\}$ is a leaf and $S_{g}\t \{0\}$ is foliated as $S_{g}^{0}=\delta\t \{0\}\t [0,1/2]\cup A_{\delta}^{1/2}\cup (S_{g}\bsl A_{\delta})\t \{0\}$ defined in Step~$3$.\\

Some notations :\\
Set $i_{0}=0$.\\
Let $n\in \N ^{*}$, we set :
$$\dps i_{n}= \sum _{k=1}^{n} \frac{1}{2^{k}}$$
and $I_{n}= [i_{n}, i_{n+1}]$, for all $n\in \N $.\\
Note that $\dps \lim_{n\rightarrow +\infty} \sum _{k=1}^{n} \frac{1}{2^{k}}=1$; hence $\dps \overline{\bigcup_{n\in\N}I_{n}}= I$.\\
\begin{figure}[ht!]
\includegraphics[width=8cm]{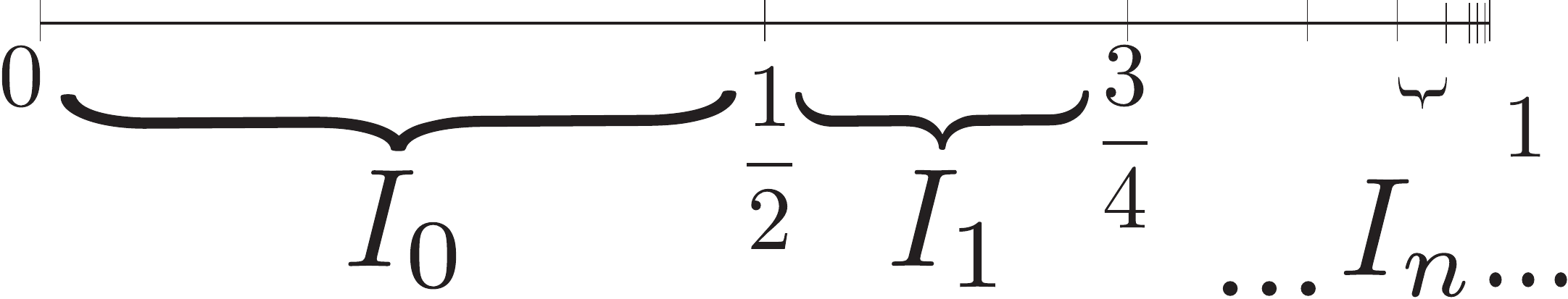}
\caption{Intervals $I_{n}$}\label{interval}
\end{figure}

For all $n\in \N$, consider $(S_{g}\bsl A_{\delta}) \t I_{n}$ with the foliation $\F_f$ defined in Step $2$.\\

Let $h_{n}, n\in \N$ be foliation preserving homeomorphisms between $\delta\t \{0\}\t I_{n+1}$ and $\delta\t \{1\}\t I_{n}$ sending an interior leaf on an interior leaf, and sending $\delta\t \{0\}\t \{i_{n+1}\}$ on $\delta\t \{1\}\t \{i_{n}\}$, and $\delta\t \{0\}\t \{1/2\}$ on $\delta\t \{1\}\t \{0\}$.\\

For each $n\in \N$, apply clockwise $(\delta , f, h_{n})$-gluing (defined in Step $3$) between $(S_{g}\bsl A_{\delta})\t I_{n}$ and $(S_{g}\bsl A_{\delta})\t I_{n+1}$, constructing the foliation $\F(f,I_{n},I_{n+1},h_{n})$, and consider the closure of this manifold, to obtain $S_{g}\t I$ with the clockwise foliation $$\dps \F(f,h_{n},n\in\N)=\overline{\bigcup_{n\in\N} \F(f,I_{n},I_{n+1},h_{n})}$$\\
Since $\dps \overline{\bigcup_{n\in\N}I_{n}}= I$, the homeomorphisms $h_{n},n\in\N$ can be considered as a single homeomorphism $h$ of $I$; so for more simplicity we denote this foliation by $\F(f,h)$ which will be called the clockwise foliation $\F(f,h)$.\\

We similarly define a foliation by considering only anti-clockwise $(\delta , f, h_{n})$-gluing for all $n\in \N$ to obtain a anti-clockwise foliation $\F(f,h)$.\\

This amounts to consider $\dps (S_{g}\bsl A_{\delta}) \t \overline{\bigcup_{n\in\N}I_{n}}$, and to extend the foliation by pasting annuli called $A_{\delta}^{i_{n}}$ for $n\in \N ^{*}$ between $\delta\t \{0\}\t \{i_{n}\}$ and $\delta\t \{1\}\t \{i_{n-1}\}$, and bands $\R\t I$ between the spirals $\R\t \{0\}$ on $\delta\t \{0\}\t \{i_{n}\}$ and the spirals $\R\t \{1\}$ on $\delta\t \{1\}\t \{i_{n-1}\}$ with respect to $h_{n}$.\\

When $n$ tends towards the infinity, $ i_{n}$ tends toward $1$, so we attach an annulus $A_{\delta}^{1}$ between $\delta\t \{0\}\t \{1\}$ and $\delta\t \{1\}\t \{1\}$, hence $S_{g}\t \{1\}$ is a leaf.\\
Moreover, as in Step $3$, $S_{g}^{0}=\delta\t \{0\}\t I_{0}\cup A_{\delta}^{1/2}\cup (S_{g}\bsl A_{\delta})\t \{0\}$ is the second boundary component homeomorphic to $S_{g}\t \{0\}$, where $\delta\t \{0\}\t I_{0}$ is transverse and $A_{\delta}^{1/2}\cup (S_{g}\bsl A_{\delta})\t \{0\}$ is tangent to the foliation  $\F(f,h)$.

Note that the leaf starting in $\delta\t \{0\}\t \{0\}$ is homeomorphic to the half-infinite ladder as shown in Figure \ref{ladder}.
\begin{figure}[ht!]
\includegraphics[width=10cm]{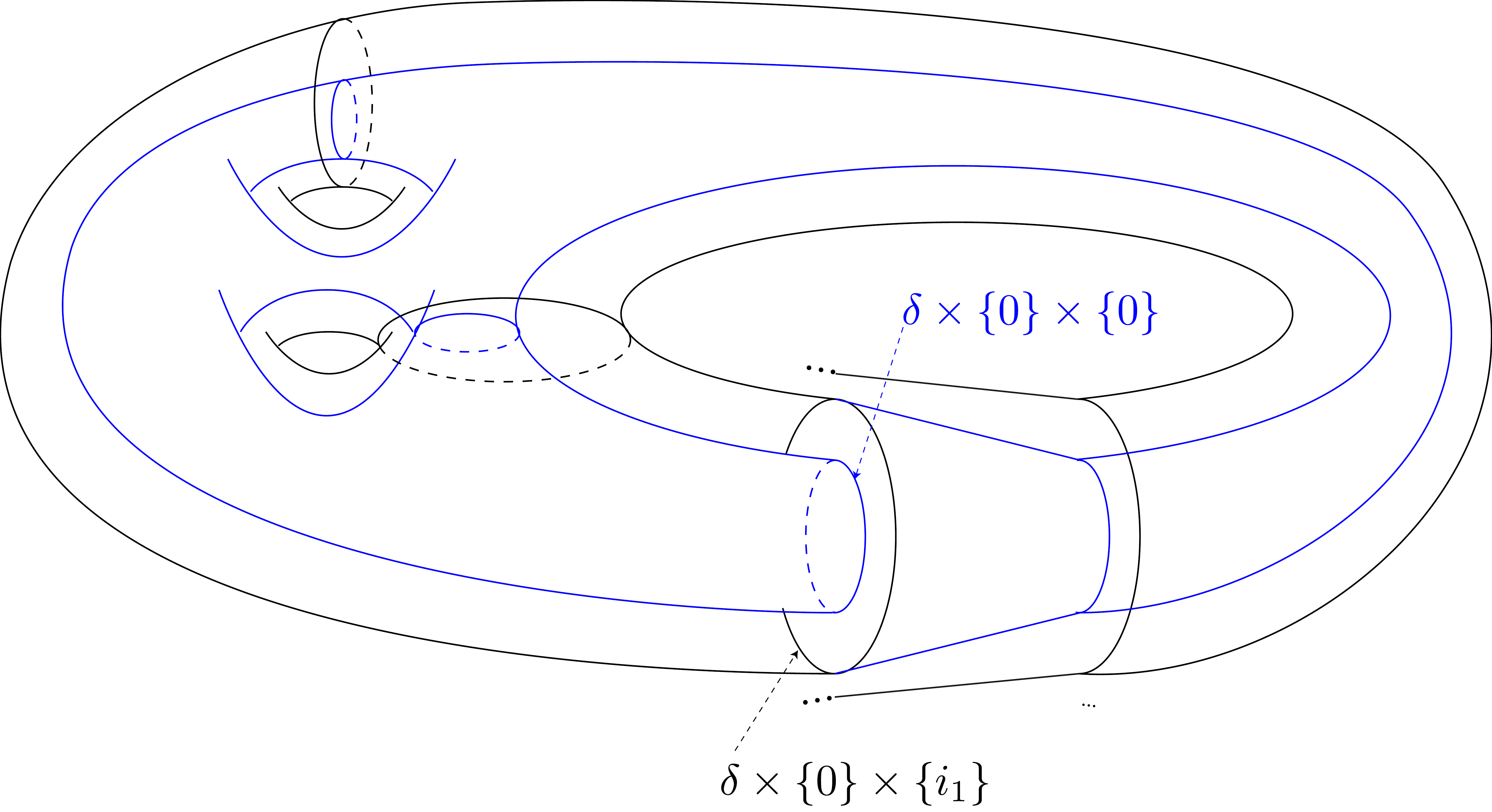}
\caption{Interior leaf starting in $\delta\t \{0\}\t \{0\}$}\label{ladder}
\end{figure}

\begin{RK}\label{infini}
Note that the induced foliation on the transverse annulus $ \dps X=\delta\times \{0\}\times \overline{\bigcup_{n\in\N}I_{n}}$ has an infinite number of circle leaves (which are $\dps\bigcup_{n\in\N} \delta\times \{0\}\times \{i_{n}\}$). Between two consecutive such circle leaves there is the suspension foliation induced by $f$.\\
\end{RK}

Now we can define Spiraling.
\begin{DEF}
That construction of $\mathcal{F}(f,h)$ is called \textbf{\emph{spiraling}}.\\
We say that the nearby leaves of $S_{g}\times\{1\}$ in the clockwise (respectively in the anti-clockwise) foliation $\mathcal{F}(f,h)$ are \textbf{\emph{clockwise spiraling}} (respectively \textbf{\emph{ anti-clockwise spiraling}}) along $S_{g}\times\{1\}$. \\
We will denote the component $S_{g}\t I$ with the foliation $\mathcal{F}(f,h)$ by $\mathcal{S}_g(f,h)$, or when there is no ambiguity $\mathcal{S}_g$, with possibly adding the direction of rotation (clocwise or anti-clockwise). \\
This foliation is of course transversely orientable; let $\mathcal{S}_g^+$ (respectively $\mathcal{S}_g^-$) be the foliation $\mathcal{F}(f,h)$ where the transverse orientation on the closed compact boundary leaf $S_{g}\times \{1\}$, points out (respectively into) $ S_{g} \t I$. 
\end{DEF}

\begin{RK}
If $f=Id$ there is no spiral leaves, there are only half-infinite cylinders ($g=1$), or half-infinite ladders ($g>1$), and the foliation does not depend on $h$. In this case we will denote this foliation by $\mathcal{S}_g(Id)$.
\end{RK}

\subsection{Attaching components of spiraling}\label{attach}

Consider a $3$-manifold $M$ with a foliation $\F$, admitting a boundary component homeomorphic to $S_{g}$. Assume that $\F_{|S_{g}}$ has circle and spiral leaves contained in an annulus $A$ with circle boundary leaves, and that $\F_{|S_{g}\bsl A}$ is a leaf.\\
We want to extend $\F$  in a neighborhood $S_{g}\t I$ (where $S_{g}=S_{g}\t \{0\}$), so that $S_{g}\t \{1\}$ is a leaf.\\

By the above construction it suffices to choose $f$ such that the foliation on $\delta\t \{0\}\t [0,1/2]$ and the foliation on $A$ are homeomorphic, and glue $\delta\t \{0\}\t [0,1/2]$ on $A$, and $S_{g}\bsl A$ on $(S_{g}\bsl A_{\delta})\t \{0\} \cup A_{\delta}^{1/2}$ defined in Step~$3$, (see Figure \ref{super-bord}).\\
It remains to make the good choice of $f$.\\

Recall that $A$ is foliated by spirals and circles, and denote by $\G$ its foliation. \\
$\mathcal{G}$ is a foliation of an annulus with circles boundary leaves, because $S_{g}\backslash A$ is a leaf.\\
Hence $\G$ is isotopic to a union of the annuli of Figure \ref{anneau}.

 \begin{figure*}[htb!]
\begin{minipage}[b]{0.33\linewidth}
\centering
\centerline{\includegraphics[width=4.5cm]{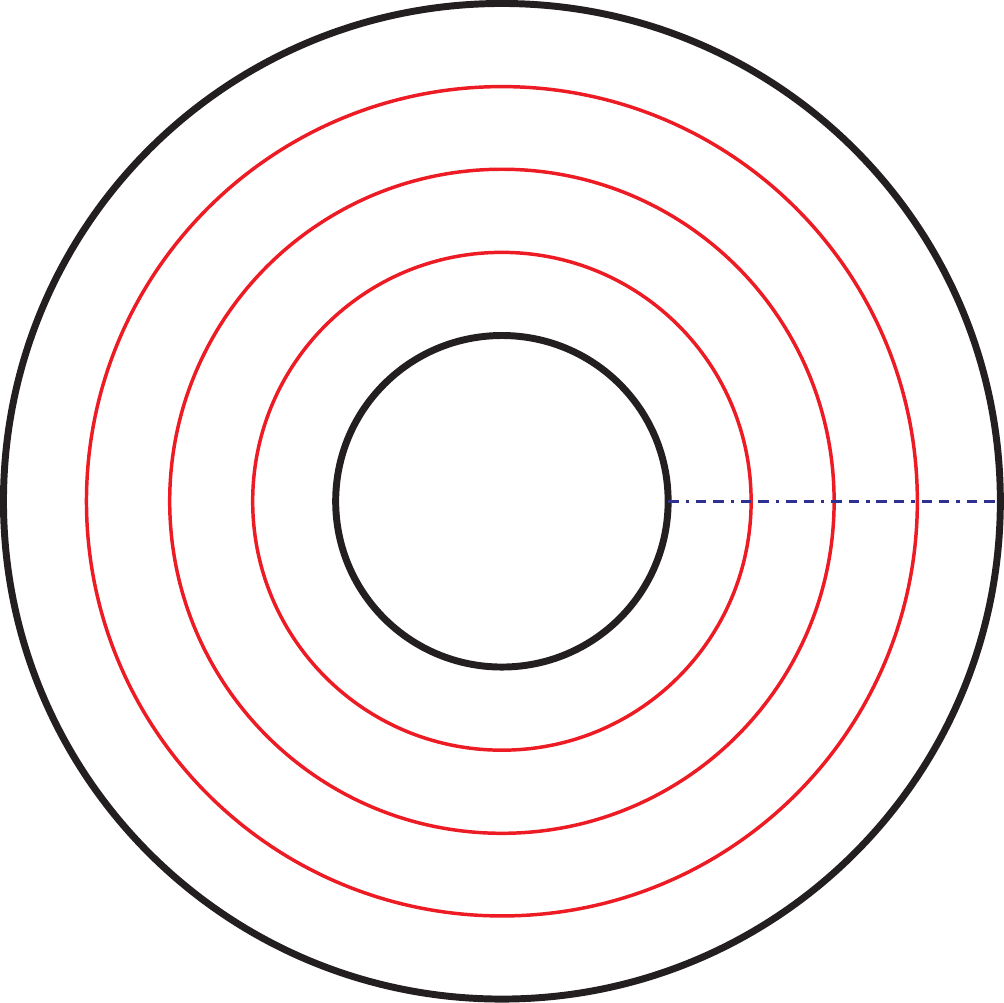}}
\centerline{\footnotesize{(a) Circle foliation}}
\end{minipage}
\hfill
\begin{minipage}[b]{0.33\linewidth}
\centering
\centerline{\includegraphics[width=4.5cm]{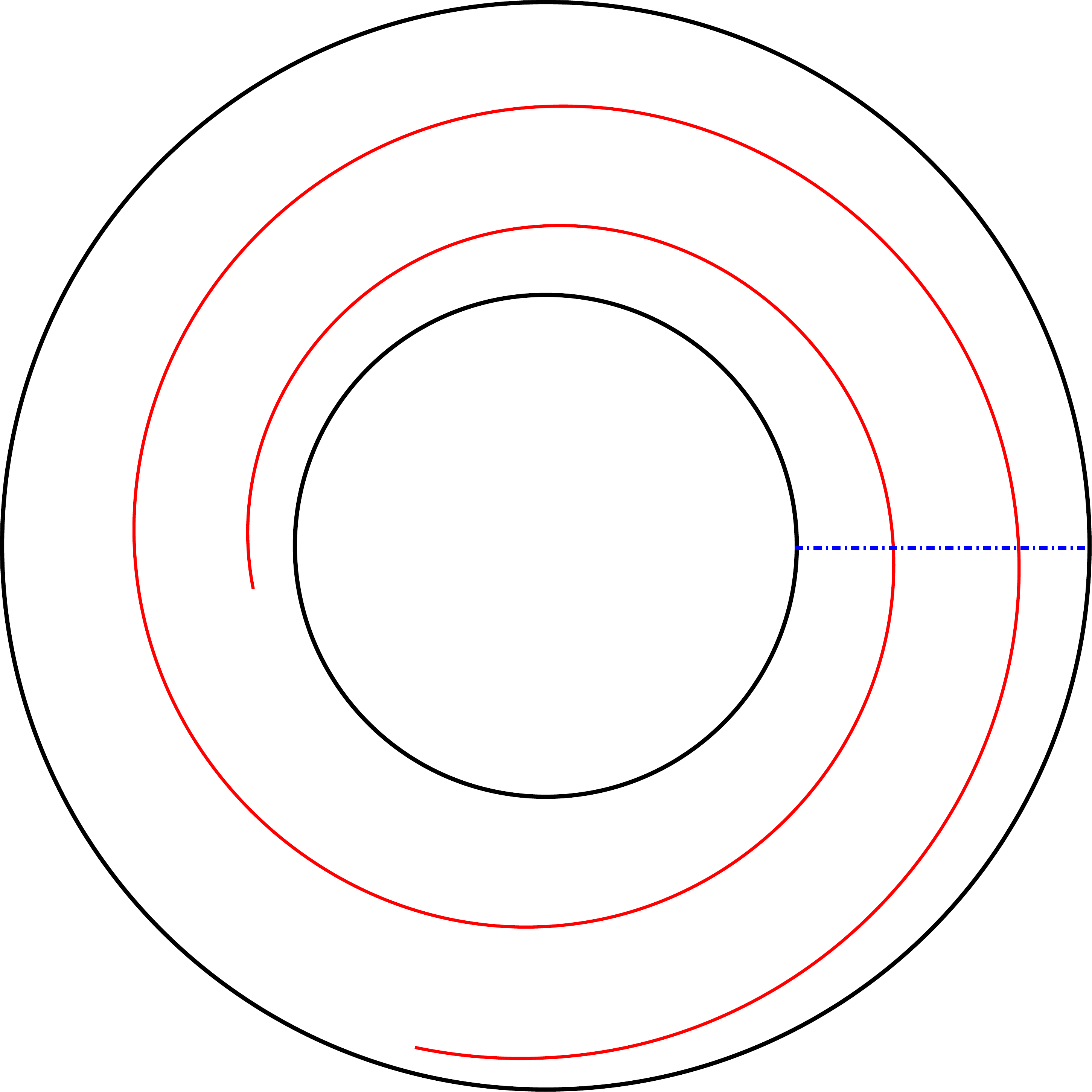}}
\centerline{\footnotesize{(b) Spiral foliation}}
\end{minipage}
\hfill
\begin{minipage}[b]{0.33\linewidth}
\centering
\centerline{\includegraphics[width=4.5cm]{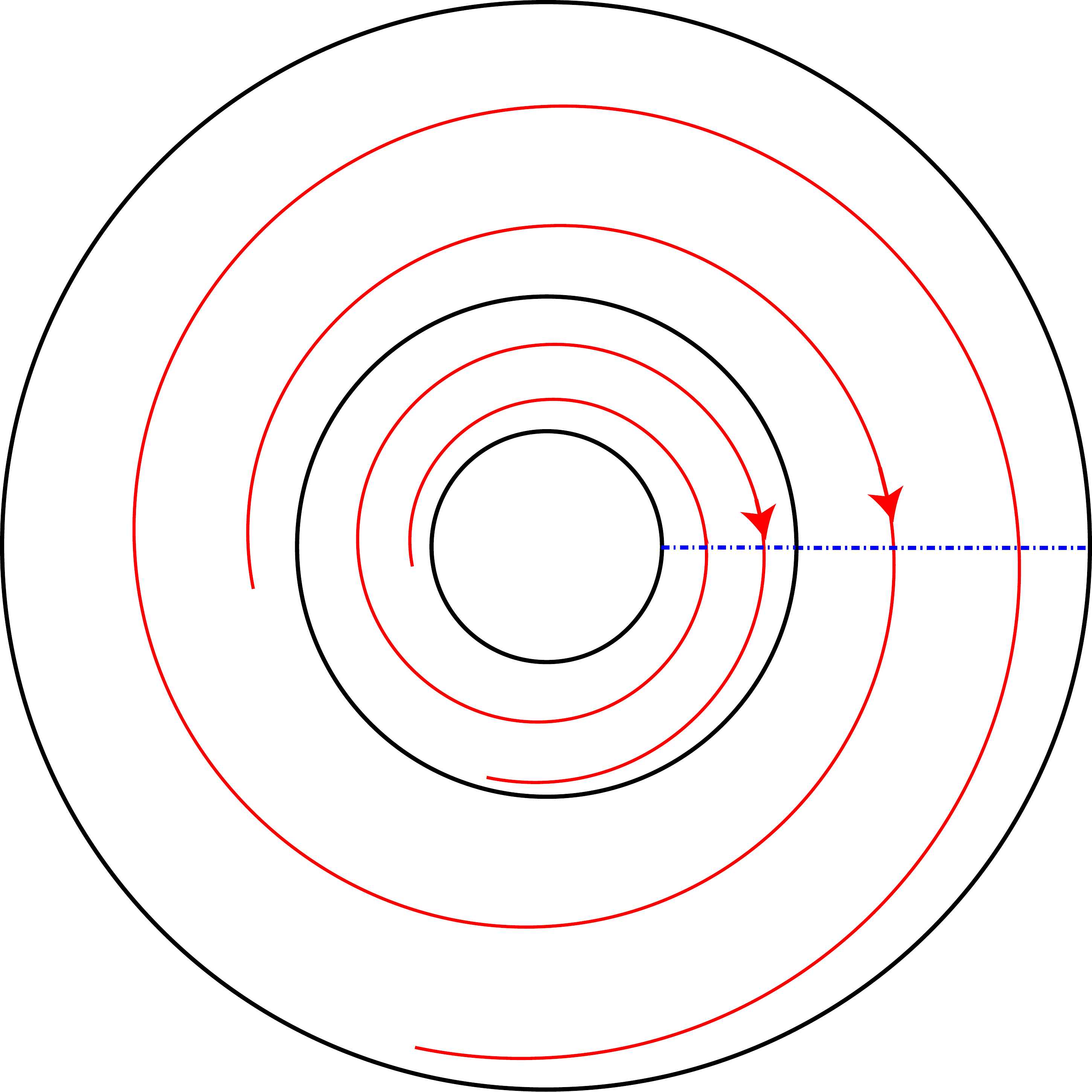}}
\centerline{\footnotesize{(c) Other possibility}}
\end{minipage}
\hfill
\begin{minipage}[b]{0.33\linewidth}
\centering
\centerline{\includegraphics[width=4.5cm]{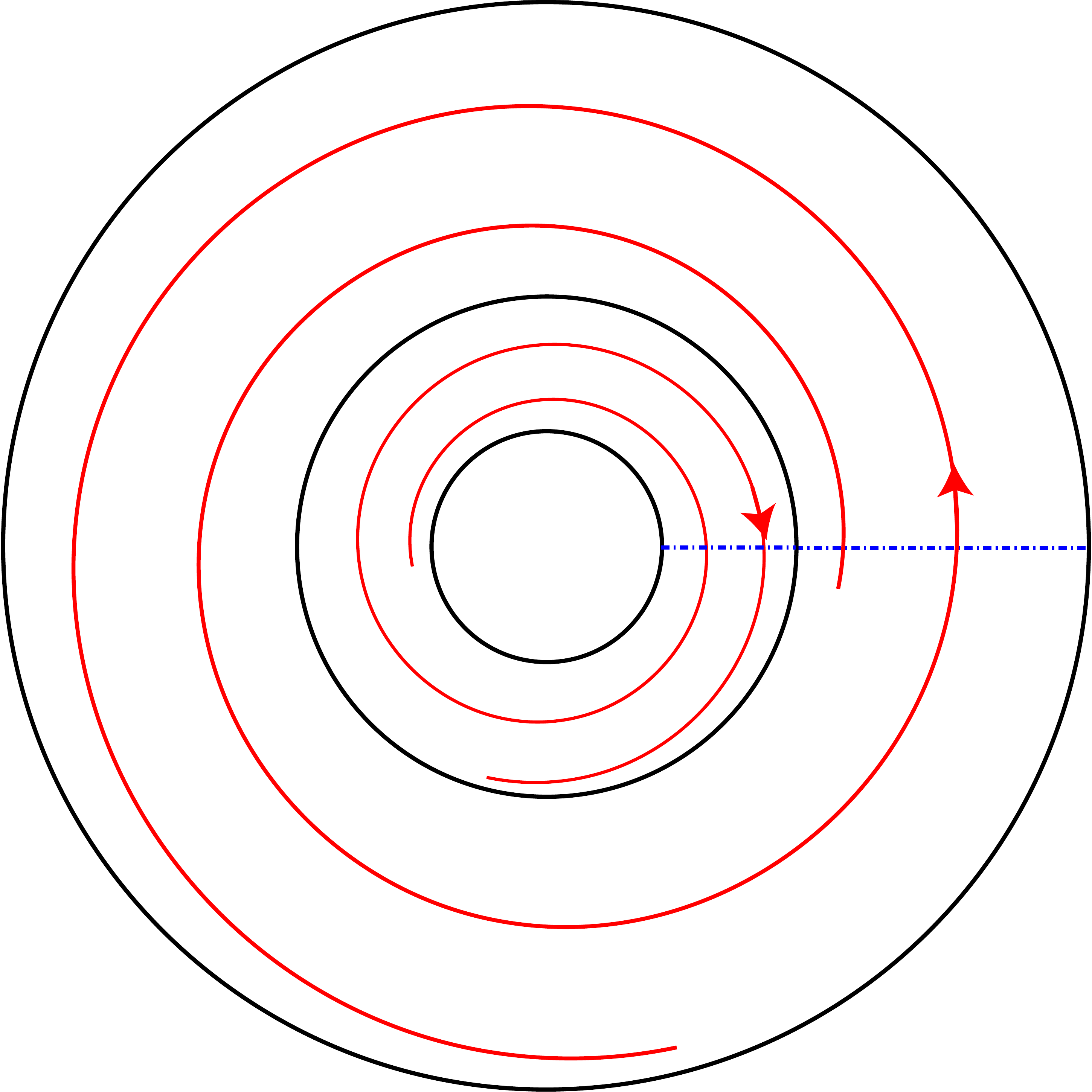}}
\centerline{\footnotesize{(c) Other possibility}}
\end{minipage}
\caption{ Taut foliations of the annulus with boundary leaves.}
\label{anneau}
\end{figure*}

Considering that $A=\S^{1}\t I$, denote by $\tau=\{*\}\t I$ a transverse arc parametrized by $[0,1]$. \\
Obviously, if $F$ is a compact leaf of $\G$ (hence a circle) then $F\cap \tau=\{*\}\t \{x\}$, for a single $x\in I$, and if $F$ is non-compact, there are infinitely many such $x$.\\
Let us call $X=\{ x\in I /\exists F\in \G, F\cap \tau = \{*\}\t \{x\}\}$. By definition of $\G$, $\{0,1\}\subset X$.\\
Note that possibly there exists $0\leq a< b\leq 1$ such that $[a,b]\subset X$.\\
 If $X=I$, $\mathcal{G}$ has trivial holonomy, i.e. $\G$ is a circle foliation.\\
 
Let $f$ be a smooth, increasing map from $I$ to $I$, such that $\restriction{f}{X}$ is the identity on $X$; and assume that $f$ is strictly increasing out of $X$, and that $f$ gives rise to the good direction of rotation. That is to say that for each spiral $F$ of $\G$, there exists $x\in F\cap\tau\cap ( I\bsl X)$, and we denote by $I_{F}$ the connected component of $x\in I\bsl X$. If $F$ is a clockwise spiral, we set $f(t)>t, t\in I_{F}$, otherwise ($F$ is anti-clockwise) we set $f(t)<t,t\in I_{F}$.\\
 Up to isotopy, f is the holonomy map of $\G$, but we will not use holonomy here.
 
By constructing the suspension foliation along $\lambda$ (Step $2$), we create spiral and circle leaves on $\delta\t \{0\}\t [0,1/2]$. More precisely, we create circle leaves when $f$ is the identity i.e when there are circles on $\G$, and spiral leaves out of $X$, with the corresponding direction of rotation; which is exactly the expected foliation (up to isotopy).

\begin{RK}\label{or-spir}
Note that given a transverse orientation on the tangent part $S_{g}\backslash A$, say outward, (respectively inward), spiraling amounts to glue on $S_{g}$ a component $\mathcal{S}_g^+$ (respectively $\mathcal{S}_g^-$).
\end{RK}

%

%
%
%
%
%
%
%
%

\subsection{Reeb annulus spiraling}\label{gen-spi}

 When $g=1$, we can also define spiraling for non-taut foliations, i.e when the induced foliation $\mathcal{G}$ on $A$ admits Reeb annuli, as soon as it has boundary circle leaves. \\

 We keep the previous notations of Step $4$, and recall that $\dps \overline{\bigcup_{n\in\N}I_{n}}= I$.\\
 Indeed, here $S_{1}$ is a $2$-torus that we denote $T$, hence $T\bsl A$ is an annulus so we have the following representation~:\\
 
$I\times \mathbb{S}^1 \t I \cong\{ (x,\theta,t), x \in I, \theta \in ]-\pi,\pi], t\in I \}$, and denote for $t\in I$ $\delta\times \{0\}\times \{t\}=\{ (0,\theta,t), \theta \in ]-\pi,\pi] \}$.\\

Now foliate each annuli $\delta\times \{0\}\times I_{n}, n\in \mathbb{N}$, and $\delta\times \{1\}\times I_{n}, n\in \mathbb{N}$, by a foliation isotopic to $\mathcal{G}$ (the foliation is $\G$ up to dilatation).\\
Consider $T\times I \cong ( I\times \mathbb{S}^1\t I)/\!\raisebox{-.65ex}{\ensuremath{\sim }}$\\
where $(0,(\theta,t)) \sim  (1, h_{n}(\theta,t))$, for given foliation preserving homeomorphisms $h_{n}$ sending $\delta\times \{0\}\times I_{n}$, on $\delta\times \{1\}\times I_{n}$, for each $n\in \mathbb{N}$, depending on the integer $n$ such that $t\in I_{n}$. As above all the homeomorphisms $h_{n}$ can be seen as a single homeomorphism from  $\delta\times \{0\}\times I$, on $\delta\times \{1\}\times I$ since $\dps \overline{\bigcup_{n\in\N}I_{n}}= I$.\\
Denote for each $t\in I$ and $x \in I$, $\delta_{x}^{t}=\{ (x,\theta, t), \theta\in ]-\pi,\pi]\}\cong  \mathbb{S}^1$.

With those coordinate, we assume $A\cong \delta\times \{0\}\t I_{0}=\{(0,\theta,t), \theta\in ]-\pi,\pi], t\in I_{0}\}\subset T\times I$; and $T\backslash A\cong \{z \in\delta_{x}^{\frac{1}{2}(1-x)}, x\in [0,1]\}$;  (see Figure \ref{spi-gen}).

In $\delta\times I \t I$ consider for all $n\in\N^*$, the annulus leaves denoted $A^{i_{n}}$ connecting $\delta\times \{0\}\times i_{n}$ to $\delta\times \{1\}\times i_{n-1}$, i.e $\dps A^{i_{n}}= \bigcup _{x\in I} \delta_{x}^{-\frac{1}{2^{n}}x+i_{n }}$ (note that $A^{1/2}= T\bsl A$).\\
Foliate each solid cylinder bounded by $\delta\times \{0\}\times I_{n+1}\cup A^{i_{n+1}}\cup \delta\times \{1\}\times I_{n}\cup A^{i_{n}}$, $n\in \mathbb{N}$ by a foliation isotopic to $\G \t I$ with respect to the foliation set on the annuli $\delta\times \{0\}\times I_{n}, n\in \mathbb{N}$, and $\delta\times \{1\}\times I_{n}, n\in \mathbb{N}$.

\begin{RK}
This choice of foliated solid cylinder induces a clockwise foliation while the other choice (joining $\delta\times \{0\}\times I_{n}$ to $\delta\times \{1\}\times I_{n+1}$, $n\in \mathbb{N}$) induces a anti-clockwise foliation.
\end{RK}

When $n$ tends towards the infinity we add the torus leaf $T\t \{1\}$, because $A^{1}$ connects $\delta\times \{0\}\times 1$ to $\delta\times \{1\}\times 1$ since $\dps\lim_{n\rightarrow+\infty} i_{n} =1$.

\begin{DEF}
We call that component \textbf{\emph{generalized spiraling component}}, denoted $\mathcal{S}_{*}(\G,h)$, or $\mathcal{S}_{*}$ when there is no ambiguity.
\end{DEF}

Note that since the identification is by a foliation preserving homeomorphism; if $\G$ has no Reeb annuli,  $\mathcal{S}_{*}(\G,h)=\Ss_{1}(f,h)$, where $f$ is a suspension homeomorphism defining $\G$.\\

The two boundary components are :\\
 $T\backslash A \cup A $, where $A$ is transverse and $T\bsl A$ is tangent to the foliation.\\
 The torus leaf $T\times \{1\}$.\\

Let us describe the induced foliation by $\mathcal{S}_{*}(\G,h)$ on the properly embedded transverse annulus $X=\delta\times \{0\}\times I$.\\
$\partial X$ is included in leaves, i.e that foliation admits circle boundary leaves.\\
It has infinitely many circles leaves in its interior. Indeed, the leaf of $\delta\times \{0\}\times \{0\}$ is an half infinite cylinder, and its intersects $X$ in each $\delta\times \{0\}\times \{i_{n}\}, n\in \mathbb{N}$.\\
The induced foliation on each annulus $\delta\times \{0\}\times I_{n}, n\in \mathbb{N}$ (included in $X$), is $\mathcal{G}$; hence we have the following Remark.\\

 \begin{RK}
 The previous construction of $\Ss _{*}(\G,h)$ when $\G$ admits at least one Reeb annulus contradicts part (i) of Theorem $4.2.15$ of \cite {HH}, which says that a foliation of an annulus tangent to the boundary can only admits finitely many Reeb components.\\
\end{RK}

 \begin{RK}\label{tore-trans-spir}
 The induced foliation by $\Ss _{*}(\G,h)$ on $\dps T^{1/2}= \bigcup _{x\in I} \delta _{x}^{1/2}$ is the annulus foliation $\G$ whose boundary leaves are identified.
 \end{RK}

Note that by considering the induced foliation by $\mathcal{S}_{*}(\G,h)$ on the manifold homeomorphic to $T\t I$ bounded by $T\t \{1\}$ and $T^{1/2}$, we see that spiraling extend a given foliation on a $2$-torus with at least a circle leaf (possibly with spirals and Reeb annuli) to a foliation of $T\t I$ where $T\t \{1\}$ is a torus; which was our first goal.

\begin{figure}[htb!]
\includegraphics[width=9cm]{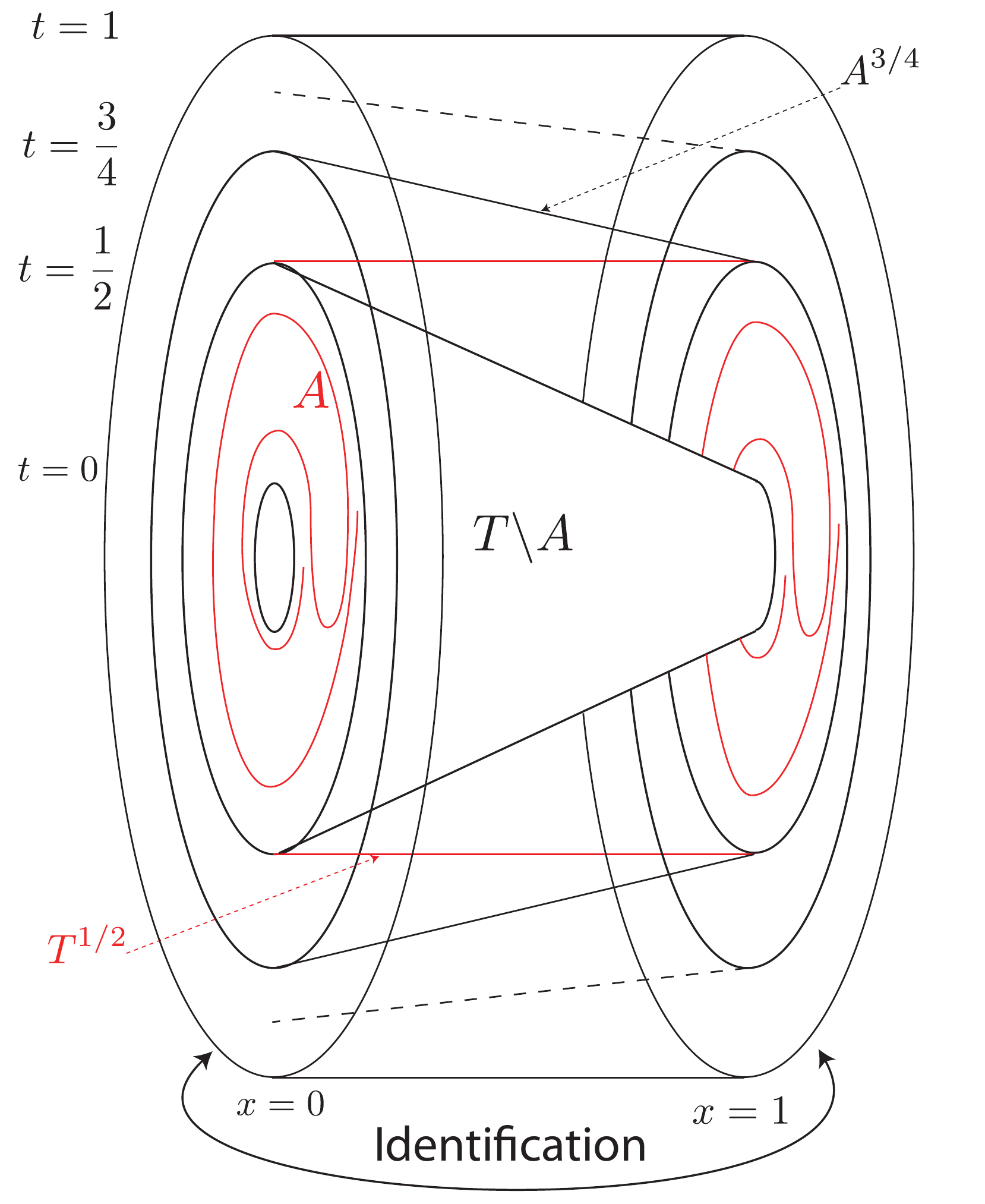}
\caption{Generalized spiraling : $\Ss_{*}$}\label{spi-gen}
\end{figure}

\section{Foliations near torus leaves}\label{torus}
We first prove the equivalence between trivial spiraling and turbulization (Lemma \ref{eq-turb-spir}).\\
Then we prove Proposition \ref {torus-leaf}.

\subsection{Equivalence between trivial spiraling and turbulization}
 \begin{LM}\label{eq-turb-spir}
$\mathcal{S}_1(Id)$ is isotopic to $\mathcal{T}$.
\end{LM}

\begin{proof}
We start from a component $\mathcal{T}$ of turbulization, and we are going to find a torus $T_1$ foliated by a tangent annulus and a transverse annulus, and then we can see that $\mathcal{S}_{1}\subset  \mathcal{T}$ up to isotopy. So $\mathcal{S}_1(Id)$ is isotopic to $\T$ since they foliate the same $3$-manifold.\\

We consider $I\t\mathbb{S}^{1}\times \mathbb{S}^{1} \cong \{ (x,\theta,z), x\in I, \theta \in ]-\pi,\pi], z\in[0,1]\}/\!\raisebox{-.65ex}{\ensuremath{\sim}}$\\
 where $(\theta,x,1) \sim (\theta,x,0)$ foliated by a $\mathcal{T}$ component.\\
Now, let $z_0\in ]0,1[$, and let $L_{z_0}$ be the leaf of $(1,\theta,z_0)$, for $\theta \in ]-\pi,\pi]$. \\
Clearly, $L_{z_0}$ does not depend on $\theta$, because the foliation $\mathcal{T}$ with those coordinate is invariant by rotation around the $z$-axis.\\
Then, there exists $x_0\in ]0,1[$ such that the point $(x_0,\theta,1)\in L_{z_0}$, by following the leaf $L_{z_0}$ when $z$ grows. Note again that $x_0$ does not depend on $\theta \in ]-\pi,\pi]$.\\
Let $A_{h}= \{(x,\theta,0), x_0\leq x \leq 1, \theta \in ]-\pi,\pi]\}$ and $A_{v}= \{(1,\theta, z), \theta \in ]-\pi,\pi], 0\leq z\leq z_0  \}$, as in Figure \ref{tour-spi}.\\
Then set $T_{1}= \{(x,\theta,z)\in L_{z_0}, x_{0}\leq x \leq 1, z_0\leq z\leq 1, \}\bigcup A_{v}\bigcup A_{h}$.\\
We can easily see that $ A_{v}\bigcup A_{h}$ is transverse to the foliation $\mathcal{T}$, and that $\{(x,\theta,z)\in L_{z_0}, x_{0}\leq x \leq 1, z_0\leq z\leq 1, \}$ is tangent to it.\\
Of course between $T_{1}$ and the torus leaf, the non-compact leaves are all half infinite cylinders, which is the case when the spiraling has no holonomy. Hence the turbulization contains spiraling with trivial holonomy.\\

\begin{figure}[htb!]
\includegraphics[width=7cm]{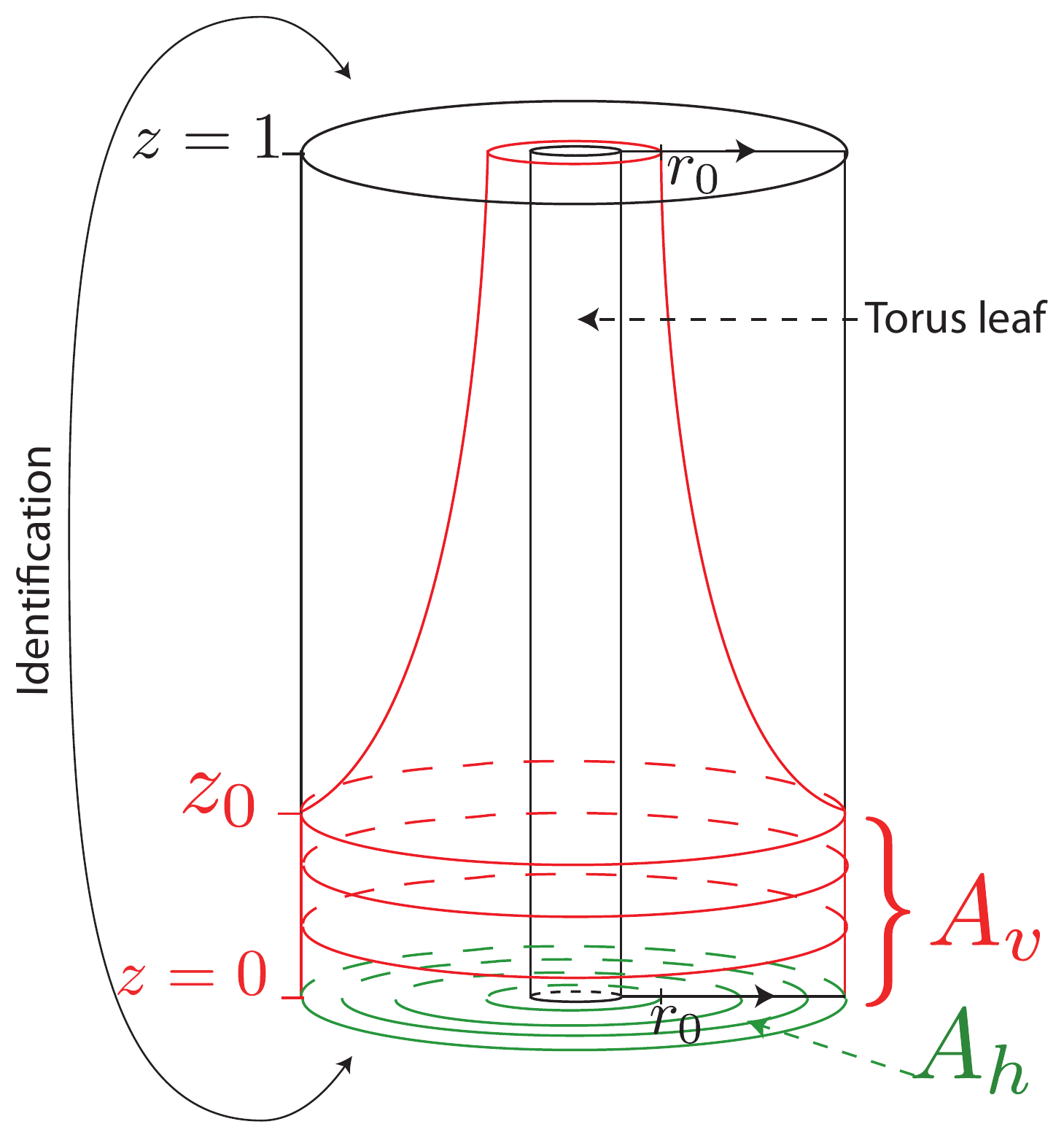}
\caption{}\label{tour-spi}
\end{figure}

Another way of doing it is by applying Remark \ref{tore-trans-spir} for the component $\Ss_{*}(\G,Id)=\Ss_{1}(Id)$ when $\G$ is a circle foliation as Figure \ref{spi-tour} shows it (of course in this case $T^{1/2}$ admits a circle foliation).

\begin{figure}[htb!]
\includegraphics[width=7cm]{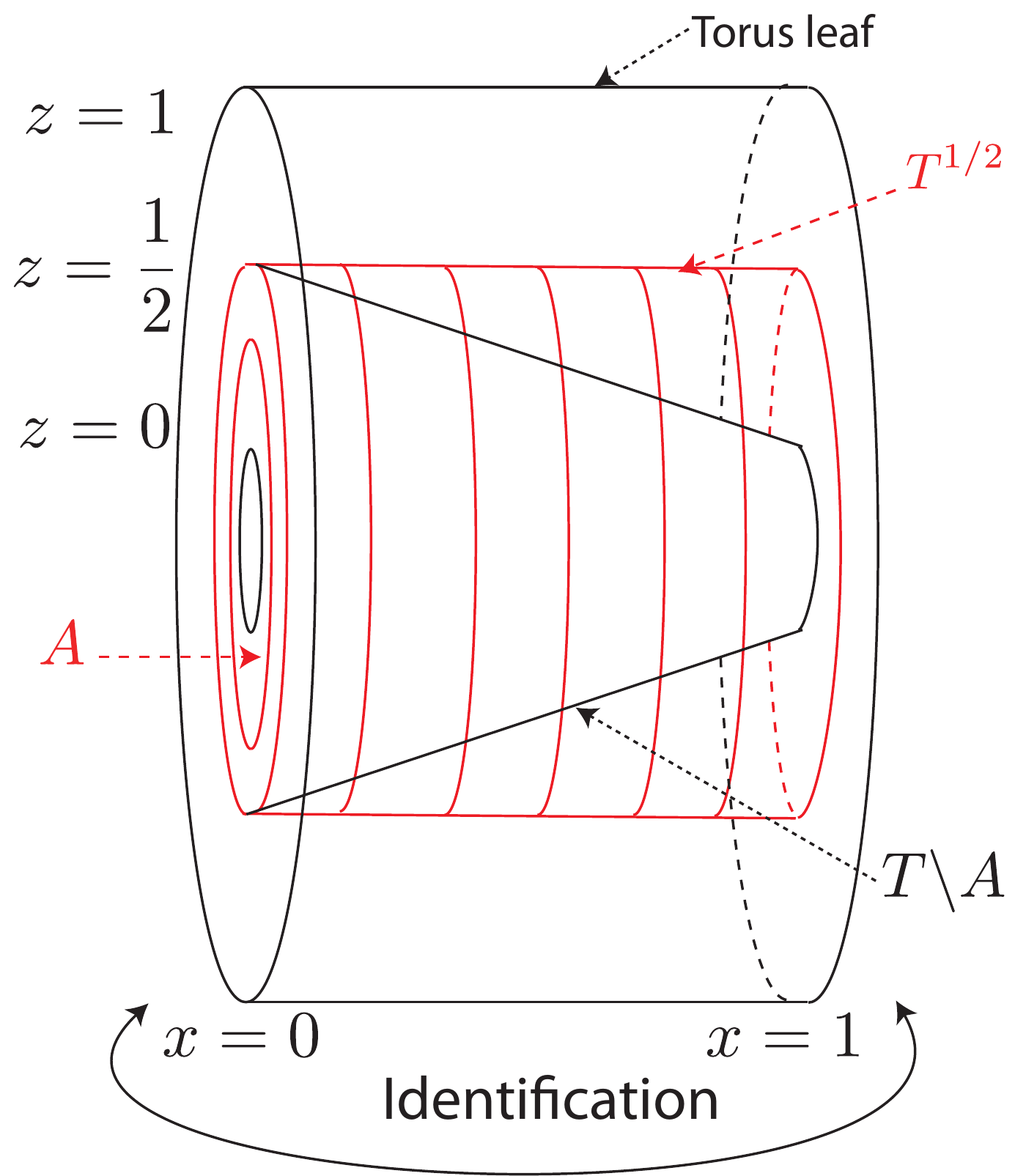}
\caption{}\label{spi-tour}
\end{figure}

\end{proof}

\subsection{Proof of Proposition \ref {torus-leaf}}

\begin{proof}
Of course if $\mathcal{F}$ admits one of those components,  $\mathcal{F}$ admits a torus leaf.\\
The converse is more interesting; it amounts to study the foliation in a neighborhood of a torus leaf.\\

The proof has two parts. First we choose a suitable neighborhood of a torus leaf (Claim \ref{dense}). Then we recognize the foliation of $T^{2}\t I$ as a $\Ss_{*}$ or a $\T_{*}$ component, using a properly embedded transverse annulus in this neighborhood.\\

Let $T$ be the torus leaf and consider a neighborhood of $T$ denoted by $V\cong T\times I$ where $T=T\times\{1\}$. \\
If $T$ is compressible, by the theorem of \cite{No}, the foliation in the $3$-manifold that $T$ bounds admits a Reeb component and so it admits a $\T$ component by Remark \ref{Reeb-turb}.\\

Thus we can assume that $T$ is incompressible, hence $T\times\{0\}$ is also incompressible.\\
Let us choose coordinates for $V$.\\
Let $W= I\t (\S^{1}\t I)$, i.e $W=\{ (x,(\theta, z)), x\in I, \theta\in ]-\pi,\pi], z\in I\}$ foliated by annulus leaves  $I\t(\S^{1}\t\{*\})$.\\
$V\cong W/\!\raisebox{-.65ex}{\ensuremath{\sim}}$, where $(0, (\theta,z)) \sim (1,h(\theta,z))$ and $h$ is a foliation preserving homeomorphism defined by $\F_{|V}$.\\
Set $X=\{(0,\theta,z), \theta\in ]-\pi,\pi], z\in I \}$.\\

First collapse all the $T\t I$ in $V$ whose foliation is $\{T\t \{t\}, t\in I\}$. Then we assume (since $M\not\cong T^{2}\t I$ and $M\not\cong T^{2}\t \S^{1}$ foliated by $\T^{2}\t \{*\}$) that all the torus leaves are isolated (i.e for each torus leaf in $V$ there exists a regular neighborhood of this torus leaf not admitting another torus leaf).

\begin{CL}\label{dense}
We can choose inside $T\t I$ a regular neighborhood $V'$ of a torus leaf $T'$ such that $V'\cong T'\t I$, $T'=T'\times\{1\}$, and there is no torus leaf in $\mathring{V'}$.
\end{CL}

\begin{proof}
If there is no interior torus leaf in $V$ we are done (choose $V'=V$) so we can assume that there is an interior torus leaf $T_{1}$ in $V$ and consider a thinner regular neighborhood of $T$ (still denoted $V$) where $T\t \{1\}=T$ and  $T\t \{0\}=T_{1}$. By continuing this process either we find such a $V'$, either this process never stops; that means that the set of torus leaves in the leaf space of $\F_{|V}$ is dense. \\
Hence we make the proof by contradiction and we suppose that such a $V'$ does not exist. We have seen above that it means that the set of torus leaves in the leaf space of $\F_{|V}$ is dense (between two torus leaves there always exists another torus leaf). \\
Recall that we can suppose that $T\t \{0\}$ is a torus leaf and that $T\t \{1\}=T$, otherwise the claim is true.\\
Consider the induced foliation by $\F$ on $X$. Call $C$ the set of circles of intersection between the torus leaves and $X$, and denote by $I_{a}^{b}=\{(0,0,z), z\in [a,b] \}$ for any real such that $0\leq a\leq b\leq 1$. Hence $C\cap I_{0}^{1}$ is dense in $I_{0}^{1}$.\\
That imposes that $X$ admits a circle foliation. Indeed, any spirals or Reeb annulus between two circle leaves contradicts the density.\\
Hence, there is two cases :\\
$V$ is trivially foliated by torus leaves which is impossible by assumption of collapsing.\\
$V$ is foliated by cylinder leaves and torus leaves.\\

In the latter case any cylinder leaf contradicts the density.\\
Indeed, up to isotopy and up to changing the coordinates, a cylinder leaf contains the annulus $A=\{(x,\theta,(b-a)x+a), x\in I, \theta \in ]-\pi,\pi]\}$ for given $a$ and $b$ in $I$.\\
Hence $C\cap I_{a}^{b}=\emptyset$ which contradicts the density, see Figure \ref{tore-dense}.\\
\begin{figure}[htb!]
\includegraphics[width=10cm]{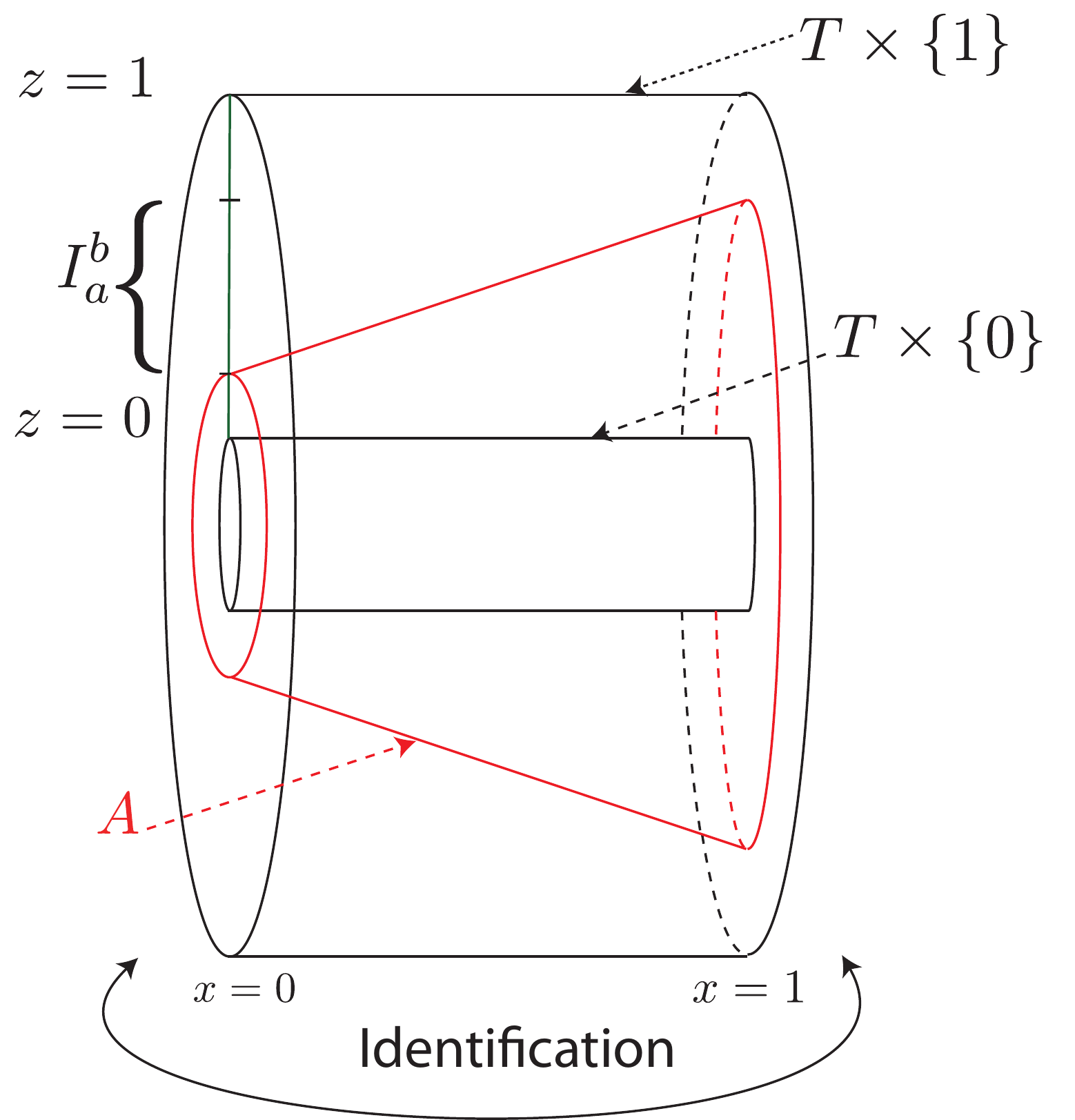}
\caption{Cylinder leaf contradicts the density}\label{tore-dense}
\end{figure}

This ends the proof of Claim \ref{dense}.
\end{proof}

By Claim \ref{dense}, we may assume that $V$ does not contain interior torus leaves.

If $T$ bounds a $\mathcal{L}$ component (see Definition \ref{def-L}, and also Figure \ref{L-T-I}), or the foliation of $Q$ (Waldhausen manifold) pictured in Figure \ref{fig:11}, then $\mathcal{F}$ trivially contains a $\mathcal{T}$ component.\\
Hence we may assume that $\F_{|V}$ is different from those two foliations.

We may recall that $T\times \{0\}$ is incompressible, so by a theorem proved by \cite{Ro} and independently by \cite{Th} we can isotope $T\times\{0\}$ such that it is everywhere transverse to the foliation or so that it is a leaf.\\
Hence, up to isotopy, we can assume that  all the $T\times \{t\}$ are transverse, for $t\in [0,1[$, and so we can consider the $1$-dimensional induced foliation on $T\times \{0\}$.\\
Since the foliation $\mathcal{F}$ is $\mathcal{C}^{2}$, so is the induced foliation on $T\times\{0\}$, and by a heorem of \cite{De}, either the induced foliation on $T\times\{0\}$ is dense (i.e all the leaves are lines); or it admits circle leaves (and some spirals limiting to those circles, or it is a circle foliation).\\

In the former case since the foliation on $T\times\{0\}$ is by parallel lines; one of the boundary component of $X$ is everywhere transverse to the foliation.\\
Thus, there is only one circle leaf, and up to isotopy, the only $\mathcal{C}^{2}$-foliation of the annulus with a boundary leaf and a transverse one is $\mathcal{C}$ (see Figure \ref{C}).\\
Since $\F$ induces this foliation on $X$ and since the foliation on $T\times\{0\}$ is by parallel lines, the induced foliation on the annuli $X_x=\{(x,\theta,z), \theta\in ]-\pi,\pi], z\in I \}$, for each $x\in I$, is also isotopic to $\C$ (note that $X=X_0$).\\
All the foliations possible on $T^2\t I$ are now characterized by the attachment possible between $X_0$ and $X_1$. All those foliations corresponds to a $\mathcal{T}_{*}$ component (see Figure \ref{T-*}).\\


The latter case where $T\times\{0\}$ admits at least one circle leaf, corresponds to a $\mathcal{S}_1$ or a $\mathcal{S}_*$ component depending on if there are circle foliation, spiral leaves or Reeb annuli. Recall that by Lemma \ref{eq-turb-spir} $\Ss_{1}(Id)$ is isotopic to $\T$. \\
Recall also that by Remark \ref{tore-trans-spir} a $\mathcal{S}_1$ or a $\mathcal{S}_*$ component can be seen as a foliation on $T^{2}\t I$ where $T^{2}\t \{0\}$ is everywhere transverse (with circle foliation, or admitting spiral leaves or Reeb annuli) and $T^{2}\t \{1\}$ is a leaf (choose $T^{2}\t \{0\}$ as $T^{1/2}$ of Remark \ref{tore-trans-spir}).\\
Moreover, since $\mathring{V}$ does not contain torus leaves, any circle leaf of $X$ is included in a cylinder leaf of $\F_{|V}$, hence there is an infinite number of circles leaves on $X$.\\
Now given such a foliation on $X$ and  on $T\times\{0\}$, the induced foliation by $\F$ on $W$ is up to isotopy (and changing the coordinates) the one of Figure \ref{seul-spir}.\\

\begin{figure}[htb!]
\includegraphics[width=12cm]{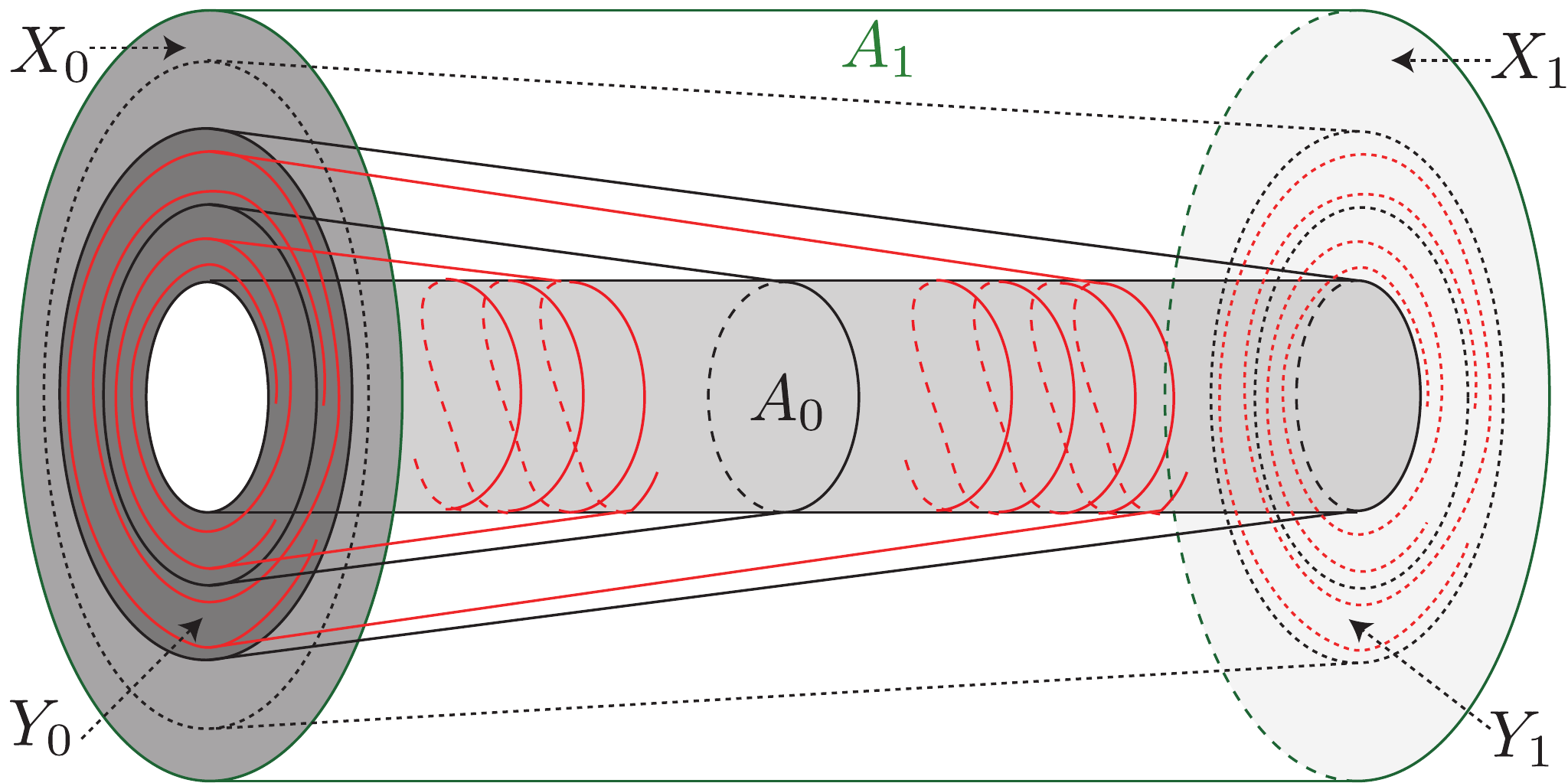}
\caption{Imposed foliation on $W$ }\label{seul-spir}
\end{figure}

Indeed, it admits an annulus leaf $A_{1}=\{ (x,\theta, 1), x\in [0,1[, \theta\in ]-\pi,\pi]\}$, $X_{1}$ has the foliation of $X=X_{0}$, it admits another transverse annulus $A_{0}=\{ (x,\theta, 0), x\in [0,1[, \theta\in ]-\pi,\pi]\}$, with the foliation of $T\times\{0\}$ split along a circle leaf that we denote by $\G$.\\
Note that in Figure \ref{seul-spir} we have chosen spiral leaves on $\G$ but we could have chosen circle foliation or foliation with Reeb annuli.
Moreover, in $W$ the only possibility to follow the foliation on those annuli is by following the projection of $A_{0}$ on a sub-annulus of $X_{0}$ denoted on Figure \ref{seul-spir} by $Y_{0}$.\\
Since $X_{1}$ has the foliation of $X_{0}$, there is a sub-annulus of $X_{1}$ foliated as $Y_{0}$ denoted by $Y_{1}$. Once again, the foliation of $Y_{1}$ can be followed in $W$ by following the projection of $Y_{1}$ on another sub-annulus of $X$. By continuing this process we obtain after gluing $X_{0}$ to $X_{1}$ by $h$, a $\Ss_{1}(f, h)$, or a $\mathcal{S}_*(\G,h)$ where $f$ is the suspension homeomorphism defining $\G$ when $\G$ has no Reeb annuli. (Note that $f=Id$ corresponds to a circle foliation on $\G$).\\
Note that $A_{0}/\!\raisebox{-.65ex}{\ensuremath{\sim}}$ corresponds to $T^{1/2}$ of Remark \ref{tore-trans-spir}.

\end{proof}

\section{Proposition \ref{sep_torus} and consequences}\label{prop-non-taut}

It is well known that Reeb's component (and Reeb annulus) are not taut.\\
\cite{B1}, generalized this fact to manifolds with at most one boundary component.\\
Here we give more details, and generalize it to manifolds with more boundary components, if we assume that the transverse orientation is the \textit{same} on each boundary component.\\
This is the goal of Proposition \ref{sep_torus} proved in Subsection \ref{preuve-sep-torus}.\\
In Subsection \ref{thin-wal} we will see that the hypothesis of Proposition \ref{sep_torus} are thin by giving interesting examples of foliations on Waldhausen manifold.\\
Finally in Subsection \ref{Non-taut} we will give a partial converse of Proposition \ref{sep_torus} which is Theorem \ref{tore_Go} (\cite{Go}) and Corollary \ref {gen-goodman}.\\

\subsection{Proof of Proposition \ref{sep_torus} }\label{preuve-sep-torus}

\begin{RK}\label{trans-or}
Note that in an orientable manifold, if a foliation is transversely orientable, then all the leaves are orientable.\\
However, the converse is not true~: there exists a foliation of $T^3$ with all the leaves orientable but which is not transversely orientable (see the foliation $\mathcal{L}_1$ in Subsection \ref{T3}).\\
Nevertheless, this foliation is not taut and if we assume that a foliation of an orientable manifold is taut and that all the leaves are orientable then this foliation is transversely orientable.
\end{RK}

\begin{DEF}
Assume that a manifold $M$ with non-empty boundary admits a transversely orientable foliation $\mathcal{F}$ such that the boundary of $M$ is a union of leaves. Then we say that \textbf{\emph{$\partial M$ has the same transverse orientation}} if the transverse orientation on those boundary leaves point all inward or point all outward.\\

\end{DEF}

 Proposition \ref{sep_torus} is a direct consequence of the following proposition.
\begin{PR}\label{sep_torus-arc}
Let $M$ be a $3$-manifold with a transversely orientable foliation $\mathcal{F}$, and $n\in\N$.\\
If the boundary of $M$ is a union of leaves $\dps \bigcup _{i=1...n} T_{i}$ with the same transverse orientation, or if  $\mathcal{F}$ contains a compact separating leaf $T_{0}$, then for all $i\in \{1,...,n\}, T_{i}$ does not admit a transverse loop or properly embedded transverse arc. 
\end{PR}

\begin{proof}
We are going to show that for every properly embedded arc $\gamma : I \rightarrow M$ with endpoints in a separating compact leaf or in a boundary leaf $T$, there exists a point of $\gamma(I)$ where the foliation is tangent to $\gamma$.\\

 This implies the proposition, assuming first that $\partial M = \emptyset$; because any closed curve transverse to a separating compact leaf $T$, intersects at least two times $T$, hence the closed curve is a union of arcs with endpoints in $T$, in each side of $T$. So we will only study one side of $T$ with an arc which meets $T$ only on its endpoints.\\
 If there is only one boundary leaf, this is exactly what we want to have.\\
If there are at least two compact boundary leaves, the endpoints of $\gamma$ may be on two differents boundary leaves; and since the transverse orientation is the same on those two leaves, the following applies similarly.\\

Up to isotopy, we assume that the induced orientation by the non-zero continuous vector field is a normal vector field to the leaves noted $N_x$ for each $x\in M$.\\
Recall that $M$ is a Riemannian manifold, hence in the tangent space of $M$ we can consider the following angles.\\
Let $h$ be the map~:\\
$$
 \left\{ \begin{array}{ll}
h~:\gamma(I)\rightarrow [0,\pi]\\ 
\\
x\mapsto (T\gamma_x,N_x)\\
  
\end{array} \right.
$$
\\

Where $(T\gamma _x,N_x)$ is the non-oriented angle between the tangent vector to $\gamma$ in $x$ (noted $T\gamma_x$) and the normal vector $N_x$.\\
$\mathcal{F}$ is transversely oriented (i.e $N$ is nowhere zero and continuous) and $\gamma$ is embedded, so $T\gamma$ is continuous; hence, $h$ is continuous.\\
Without loss of generality we can say that the transverse orientation on $T$ is such that $h(\gamma(0))>\pi/2$. Thus, $h(\gamma(1))<\pi/2$, because at $\gamma(1)$ the arc gets out of $T$, and $N_{\gamma(1)}=N_{\gamma(0)}$, see Figure \ref{pointe}. Indeed, the leaves are orientable, and if there are at least two boundary leaves we have supposed that the transverse orientation is the same.\\

By the Intermediate value Theorem ($h$ is continuous), there exists\\ $x\in~\gamma(]0,1[)$ such that $h(x)=\pi/2$. That means that $T\gamma_x$ is in the tangent plane of the leaf passing through $x$, i.e $\gamma$ is tangent to $\mathcal{F}$ in $x$.\\
Hence $\mathcal{F}$ cannot be taut; which ends the proof of Proposition \ref{sep_torus}.

\begin{figure}[htb!]
\includegraphics[width=10cm]{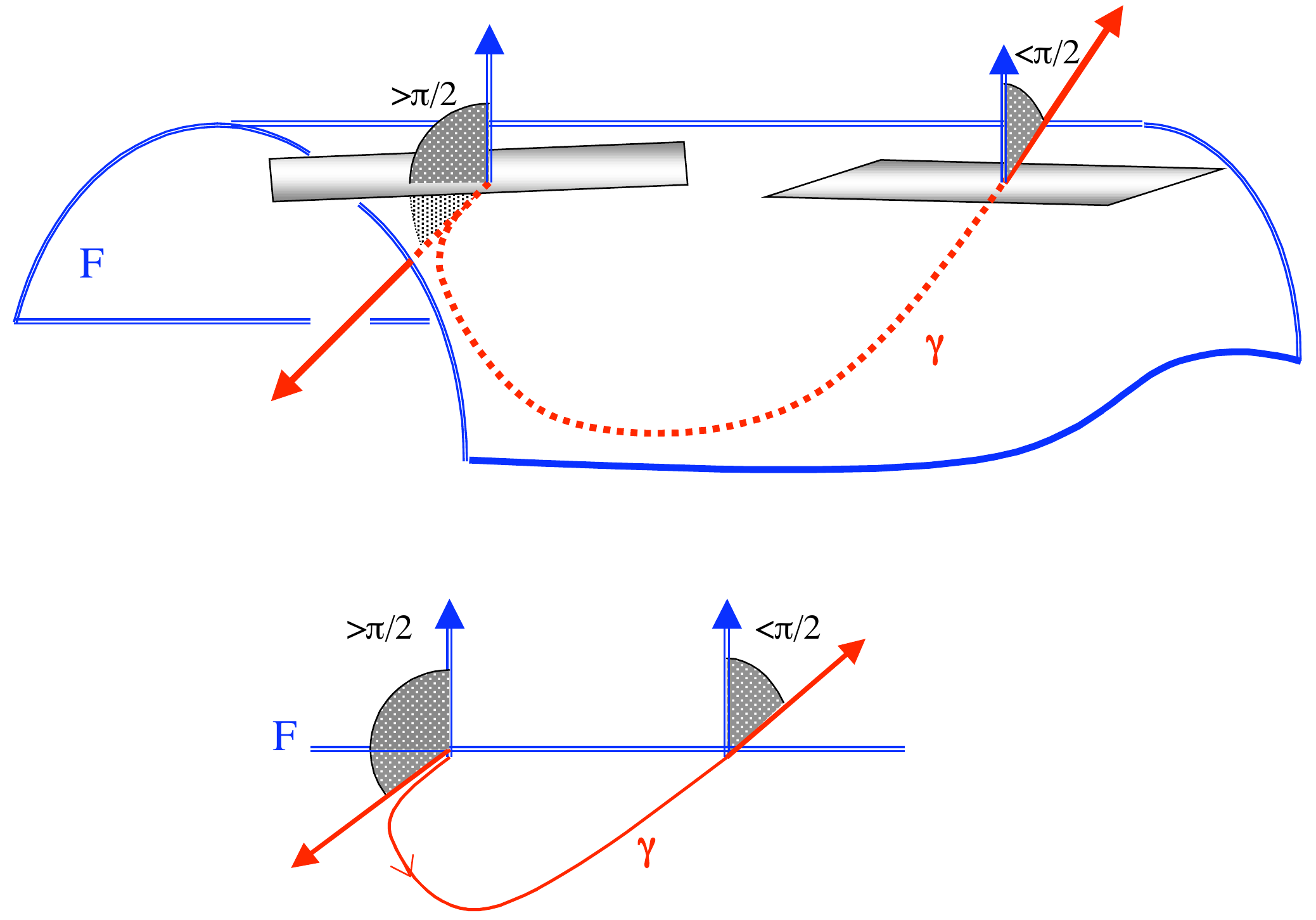}
\caption{}\label{pointe}
\end{figure}

\end{proof}

\begin{RK}
The above proposition similarly implies that Reeb annuli are not taut.
\end{RK}

\begin{RK}
Note that the assumptions of Proposition \ref{sep_torus}  are sharp. Indeed :
\begin{enumerate}
\item If there are a transverse boundary component and a tangent one, then the foliation may be taut.\\
For example, the $\mathcal{T}$ component is taut, i.e we can find a properly embedded arc $\gamma$ as in Figure \ref{fig:5}, with an endpoint on the transverse boundary torus, and the other on the torus leaf.\\
\item When there are at least two boundary leaves without the same transverse orientation, Proposition \ref{sep_torus} is not true. \\
Trivially, the foliation $S\times I$ is taut; where $S$ is any closed surface. But the transverse orientation on $S\times\{0\}$ is opposite to the one on $S\times\{1\}$ (one points inward and the other outward ). \\
\item The assumption of transverse orientation in Proposition \ref{sep_torus} is necessary as suggested by Lemma \ref{wal} on the Waldhausen manifold.\\

\end{enumerate}

\end{RK}

\subsection{Waldhausen manifold}\label{thin-wal}
\begin{LM}\label{wal}
Waldhausen manifold admits :
\begin{enumerate}
\item A taut, non-transversely orientable foliation with a single torus boundary leaf, and all the leaves are compact.
\item A taut, non-transversely orientable foliation with a single torus boundary leaf, and all the interior leaves are non-compact.
\end{enumerate}
\end{LM}

\begin{proof}

Recall that Waldhausen manifold $Q$ is the twisted product of the Klein bottle with an interval :\\
$Q=K\widetilde{\times}I$, where $K$ is the Klein bottle.\\
$Q$ has one torus boundary component $T$.\\
Let us represent $Q$ as follows~:\\
Consider $W=I\times \mathbb{S}^1\times I\cong\{(x,\theta, z), x\in [0,1], \theta\in ]-\pi,\pi],z\in [0,1]\}$.\\

Now $ Q\cong W/\!\raisebox{-.65ex}{\ensuremath{\sim}}$ where $(x,\theta,0) \sim (1-x,-\theta,1)$.\\
 
Part $(1)$ of Lemma \ref{wal} is easy to construct and is represented in Figure~\ref{compact-Q}.\\
With the above representation, there is a Klein bottle leaf which is $K=\{(1/2, \theta, z), \theta\in ]-\pi,\pi], z\in [0,1]\}/\!\raisebox{-.65ex}{\ensuremath{\sim}}$,\\
  and the other compact leaves are torus leaves which are for each $x\in [0,1]$~: \\
  $ \dps T_x=\big\{\{(x, \theta, z), \theta\in ]-\pi,\pi], z\in [0,1]\}\cup \{(1-x, \theta, z), \theta\in ]-\pi,\pi], z\in [0,1]\}\big\}/\!\raisebox{-.65ex}{\ensuremath{\sim}}$.\\

\begin{figure}[ht!]
\includegraphics[width=10cm]{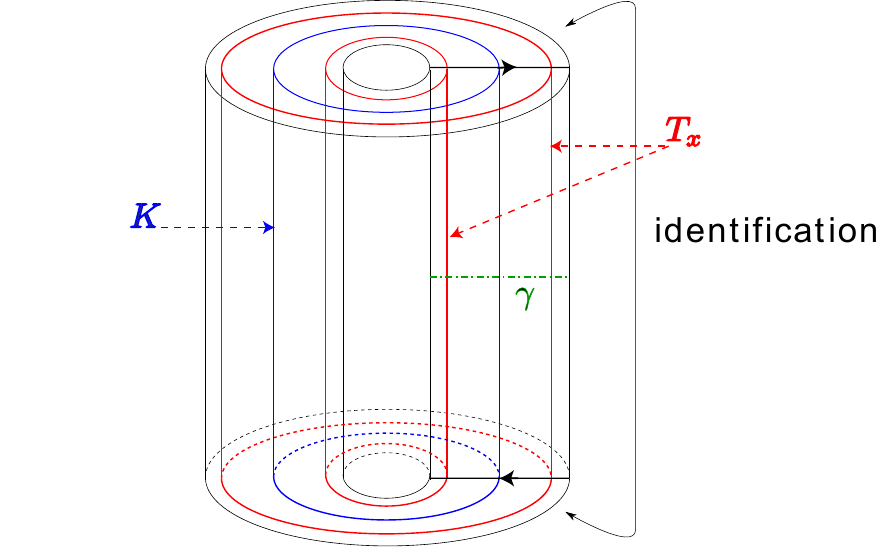}
\caption{Compact foliation of $Q$}\label{compact-Q}
\end{figure}

This foliation is taut, because for example $\gamma=\{(x,0,1/2), x\in [0,1]\}$ is a properly embedded transverse arc, with both endpoints in the torus boundary leaf (see Figure \ref{compact-Q}).\\
It is non-transversely oriented since it admits a non-orientable leaf (recall that $Q$ is orientable).\\
Hence Part (1) of Lemma \ref{wal} is proven.\\

Part $(2)$ of Lemma \ref{wal} needs more work and is the following construction.\\

For each $z\in [0,1]$, we set $A_z=\{(x, \theta, z), x\in [0,\frac{1}{2}], \theta\in ]-\pi,\pi]\}$; and $A'_{z}=\{(x, \theta, z), x\in [\frac{1}{2},1], \theta\in ]-\pi,\pi]\}$.\\
Consider a (clockwise) spiral foliation (see Figure \ref{dir}) on each annulus $A_{z}\cup A_{z}'=\{(x, \theta, z), x\in I, \theta\in ]-\pi,\pi]\}$. Denote the circle leaves by $C_z=\{(0, \theta, z),  \theta\in ]-\pi,\pi]\}$ and $C'_{z}=\{(1, \theta, z),  \theta\in ]-\pi,\pi]\}$.
For each $z\in I$, we let $C^{1/2}_{z}=\{(\frac{1}{2}, \theta, z),  \theta\in ]-\pi,\pi]\}$, they are all circles transverse to the foliation.\\
That induces on $A_{z}$ and on $A'_{z}$ a the foliation $\mathcal{C}$ (see Figure \ref{C}). \\

\begin{figure}[ht!]
\includegraphics[width=7cm]{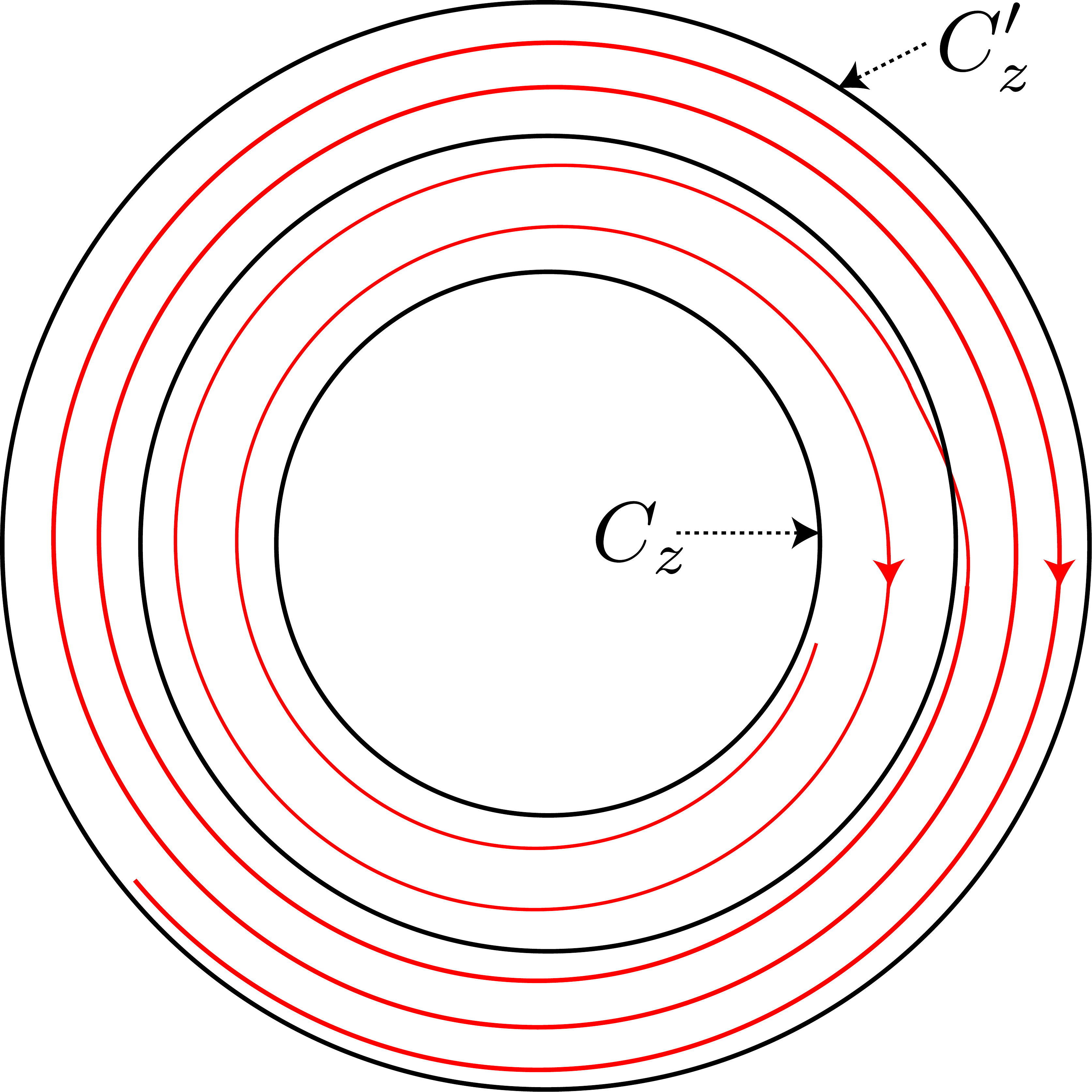}
\caption{Induced foliation on $A_{z}\cup A'_{z}, z\in [0,1] $.}\label{fig:20}
\end{figure}

That gives a taut product foliation on $\dps W=\bigcup _{z\in I} A_{z}\cup A'_{z}$ denoted $\widehat{\mathcal{F}}$.\\
 
 Now we want to use this foliation of $W$ to obtain by identification a foliation on $Q$.\\
We may recall that $ Q\cong W/\!\raisebox{-.65ex}{\ensuremath{\sim}}$ where $(x,\theta,0) \sim (1-x,-\theta,1)$.\\
So we only need to check that the foliation $\widehat{\mathcal{F}}$ in $W$ is preserved by the identification.\\
More precisely, it remains to check that any leaf of $A_{0}$ is identified on any leaf of $A'_{1}$ and any leaf of $A'_{0}$ is identified on any leaf of $A_{1}$.\\
In particular that means that for each $ \theta\in ]-\pi,\pi]$, the leaf of $A_{0}$ at $(\frac{1}{2}, \theta, 0)$ must be identified on the leaf of $A'_{1}$ passing at $(\frac{1}{2}, -\theta, 1)$.\\

One way to do it is as follows~:\\
Let $f$ be a diffeomorphism such that :\\
$$
 \left\{ \begin{array}{ll}
f~:A_{1}\cup A'_{1}\rightarrow A_{0}\cup A'_{0}\\ 
\\
f(A_{1})=A'_{0}, f(A'_{1})=A_{0} \\
\\
\forall (1/2, \theta, 1) \in C^{1/2}_{1}, f(1/2, \theta, 1)=(1/2,-\theta,0)\in C^{1/2}_{0}\\
\end{array} \right.
$$
and such that $f$ preserves the foliation, i.e $f$ maps a half-spiral leaf on a half-spiral leaf.\\
Note that the definition of $f$ induces $f(C_{1})=C'_{0}, f(C'_{1})=C_{0}$. \\
We can consider $ Q\cong W/\!\raisebox{-.65ex}{\ensuremath{\sim '}}$ where $((x,\theta),1) \sim ' (f(x,\theta),0)$ with the induced foliation by $\widehat{\mathcal{F}}$ because this foliation is preserved by the identification.\\
We denote by $\mathcal{F}$ this new foliation on $Q$.\\
Note that this new representation of $Q$ with $\sim '$ is isotopic to the first one with $\sim$.

$\mathcal{F}$ is taut, because for example $\gamma= \{(x,0,1/2), x\in I\}$ is a properly embedded transverse arc to $\mathcal{F}$.\\

$\F$ admits a single torus boundary leaf, which is $\displaystyle (\bigcup _{z\in  [0,1]} C_z\cup C'_z )/\!\raisebox{-.65ex}{\ensuremath{\sim '}}$ and the interior leaves are non-compact (they all contain an embedded $\R\t I$).

The proof of Claim \ref{cl-trans-or} ends the proof of Part (2) of Lemma \ref{wal}.

 \begin{CL}\label{cl-trans-or}
 $\mathcal{F}$ is not transversely oriented.
\end{CL}

\begin{proof}
Let $\widehat{L}$ be the leaf of $\widehat{\F}$ in $W$ containing the arc $\widehat{C}=\{(\frac{1}{2}, 0, z),  z\in [0,1]\}$, and consider a regular neighborhood $\widehat{B}$ of $\widehat{C}$ in $\widehat{L}$.\\
Set $a_{i}=\partial \widehat{B}\cap A_{i}$, and $a_{i}'=\partial \widehat{B}\cap A_{i}'$.\\
$\widehat{B}/\!\raisebox{-.65ex}{\ensuremath{\sim '}}$ is homeomorphic to a Mobius band since $f(a_{1})=a'_{0}$ and $f(a'_{1})=a_{0}$ (see Figure \ref{fig:10}), and by construction is included in a leaf of $\mathcal{F}$. Since $Q$ is oriented, $\F$ is non-transversely oriented.

\end{proof}

In conclusion, $\mathcal{F}$ is taut, non-transversely oriented with a torus boundary leaf and with non-compact interior leaf, as in Figure \ref{fig:10}, which ends the proof of Part (2) of Lemma \ref{wal}.\\

\begin{figure}[ht!]
\includegraphics[width=10cm]{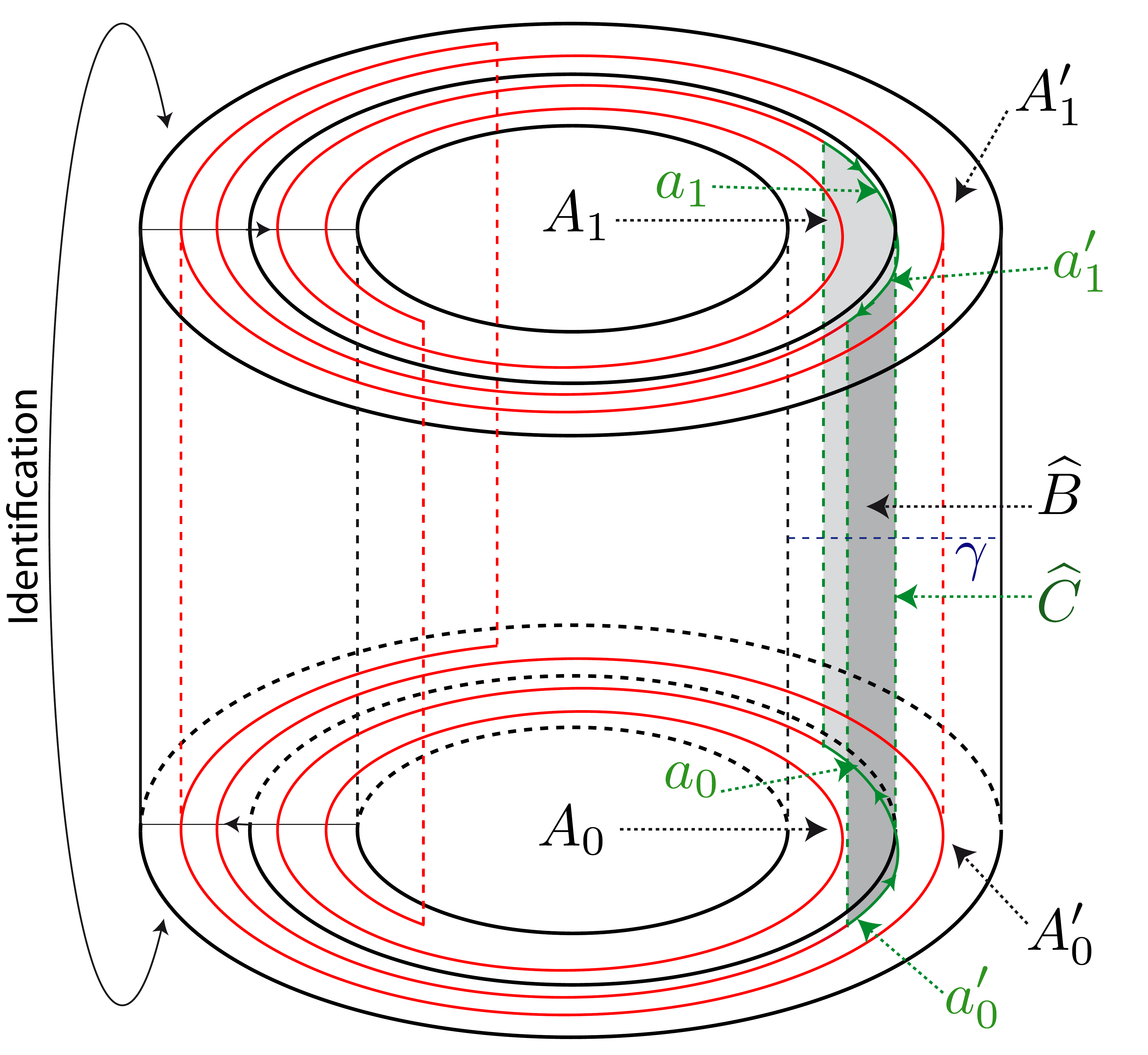}
\caption{Foliation $\mathcal{F}$ on $Q$}\label{fig:10}
\end{figure}

\end{proof}

\subsection{Partial converse: existence of torus leaf}\label{Non-taut}

Now we suppose that a transversely oriented foliation is non-taut and see that it admits a torus leaf but we cannot conclude if it  is separating or not, that is why it is a partial converse to Proposition  \ref{sep_torus}.

\begin{THM}[\cite{Go}]\label{tore_Go}
If a leaf of a transversely orientable $\mathcal{C}^1$-foliation of a closed $3$-manifold does not intersect a closed transverse curve then it is a torus leaf.
\end{THM}

Therefore, if a foliation is not taut then it admits a torus leaf.

\begin{Q}
We may wonder if it is still true when the foliation is only supposed to be $\mathcal{C}^0$.
\end{Q}

\begin{COR}\label{gen-goodman}
Consider a transversely oriented $\mathcal{C}^1$-foliation on $M$ tangent to the boundary (possibly $\partial M =\emptyset$), then the following assertions are true.
\begin{enumerate}
\item If a leaf does not admit a properly embedded transverse arc or transverse loop then it is a torus.
\item If a leaf is separating, then it is a torus leaf.
\item If the boundary components of $M$ are a union of leaves admitting the same transverse orientation, then all are tori.
\end{enumerate}
\end{COR}

\begin{proof}
When $\partial M \not =\emptyset$, we consider the double of $M$, i.e $D(M)=\displaystyle M\bigcup _{\partial M}M$ (the union of two copies of $M$ with opposite orientation). Notice that any closed transverse loop passing trough a leaf in $D(M)$ would induce a closed transverse curve, or a properly embedded transverse arc in $M$. Now we assume that a leaf does not admit a properly embedded transverse arc or transverse loop, and we apply Theorem \ref{tore_Go} to $M$ (or do $D(M)$ if $\partial M \not =\emptyset$), so this leaf is a torus, so part (1) is true.\\

If a separating compact surface is a leaf, then by Proposition \ref{sep_torus-arc}, there is no transverse loop passing through it, so this is a torus by part (1), so part (2) is true.\\

By applying Proposition \ref{sep_torus-arc} and part (1), we obtain part (3).
\end{proof}

\section{Separating compact leaf}\label{non-taut-sec}

As Theorem \ref{tore_Go} says, a non-taut foliation admits a torus leaf. Then there are two possibilities. This torus leaf can be separating or non-separating. The former case is explored in this section while the latter case is the aim of Section \ref{non-sep-torus}.\\
Note that there are three types of non-taut foliations admitting a separating torus leaf depending on if we can modify the foliation so that it becomes taut.
\begin{itemize}
\item Foliations admitting a Reeb component which can be deleted to obtain another taut foliation (example in $\S^{2}\t \S^{1}$).
\item Foliations admitting a Reeb component which cannot be deleted (example in $\S^{3}$).
\item Non-taut and Reebless foliations (example among graph manifolds).
\end{itemize}

\subsection{Non-taut foliation admitting Reeb component}
Consider a non-taut foliation of a manifold $M$ containing a Reeb component $R$ and denote by $T=\partial R$. In order to know if we can delete a neighborhood of $R$ and replace it by a trivially foliated solid torus $\D^{2}\t \S^{1}$ we need to know how $R$ is attached in $M$ and how the foliation looks like in a neighborhood of $T$ in $M\bsl \mathring{R} $.\\
This is the aim of Lemma \ref{del-Reeb}. Next we will see the opposite process which consist on considering a taut foliation and adding Reeb component as Proposition \ref{non-taut} suggests it.

In the light of Proposition \ref{torus-leaf} in a neighborhood of $T$ in $M\bsl \mathring{R} $ there exists either (generalised) spiraling component or (generalised) turbulization component, bounded by $T$ or bounded by a torus leaf $T'$ included in a neighborhood of $T$ in $M\bsl \mathring{R} $. Hence up to deleting the foliation of $T^{2}\t I$ bounded by $T$ and $T'$ we can consider that $T$ bounds a (generalised) spiraling component or (generalised) turbulization component (that changes the foliation but not the manifold $M$).\\
We want to replace the Reeb component $R$ by a trivially foliated solid torus, i.e foliated by meridians disks $\D^{2}\t \S^{1}$. That imposes that $T$ must bound in $M\bsl \mathring{R} $ a turbulization component $\T$ (see Figure \ref{fig:5}) because all the other components do not induce a circle foliation on the transverse boundary torus.\\
Moreover if the circles $C$ of the circle foliation induced by $\T$ on the transverse torus, bound meridian disks in $R$ then we can delete $R\cup \T$ and replace it by the trivially foliated solid torus $\D^{2}\t \S^{1}$ by gluing meridians disks on the circles $C$; and once again that changes the foliation but not the manifold $M$.\\

\begin{figure}[ht!]
\includegraphics[width=12cm]{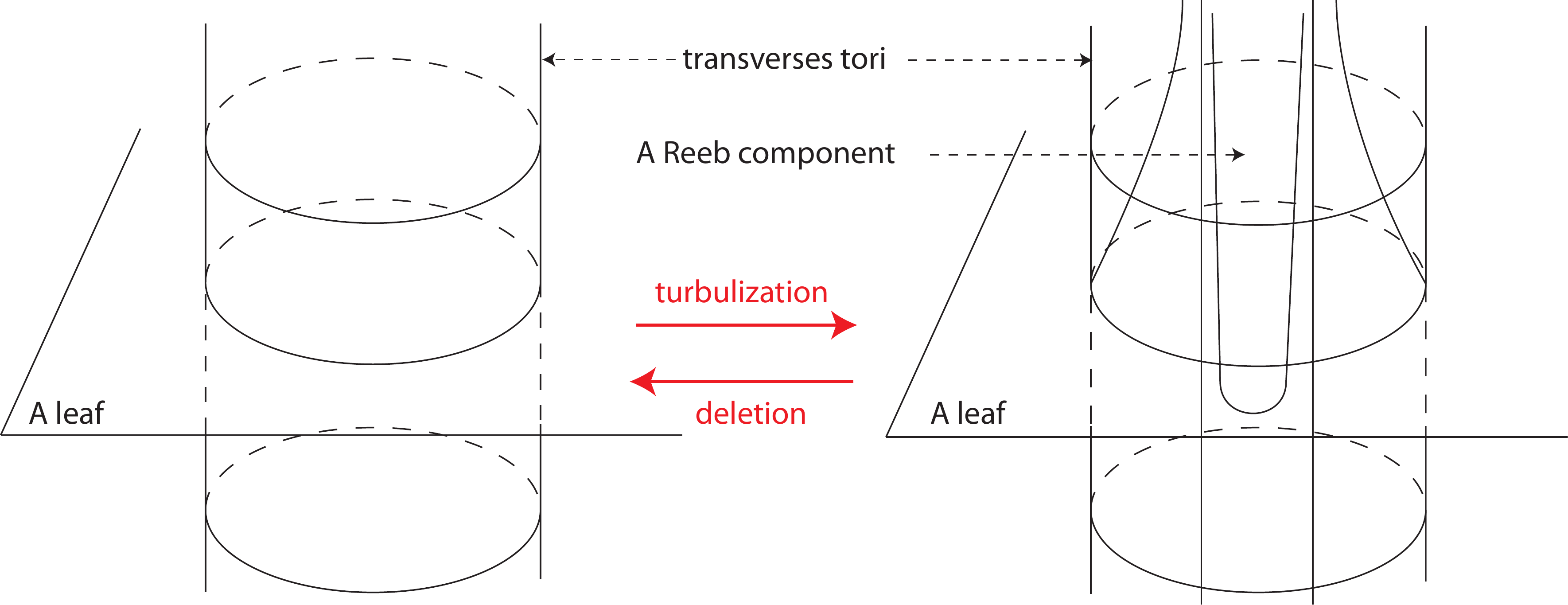}
\caption{Foliation $\mathcal{F}$ on $Q$}\label{idee-tour}
\end{figure}

Note that if the circles $C$ do not bound meridian disks in $M$ as in Figure \ref{S3} we cannot delete the Reeb component, as in the case of the Reeb foliation of $\S^{3}$ (foliation obtained by gluing two Reebs components to obtain $\S^{3}$).\\
Indeed the boundary of the meridians disks of a Reeb component of $\S^{3}$ are longitudes for the other boundary component.\\
 
 \begin{figure}[ht!]
\includegraphics[width=10cm]{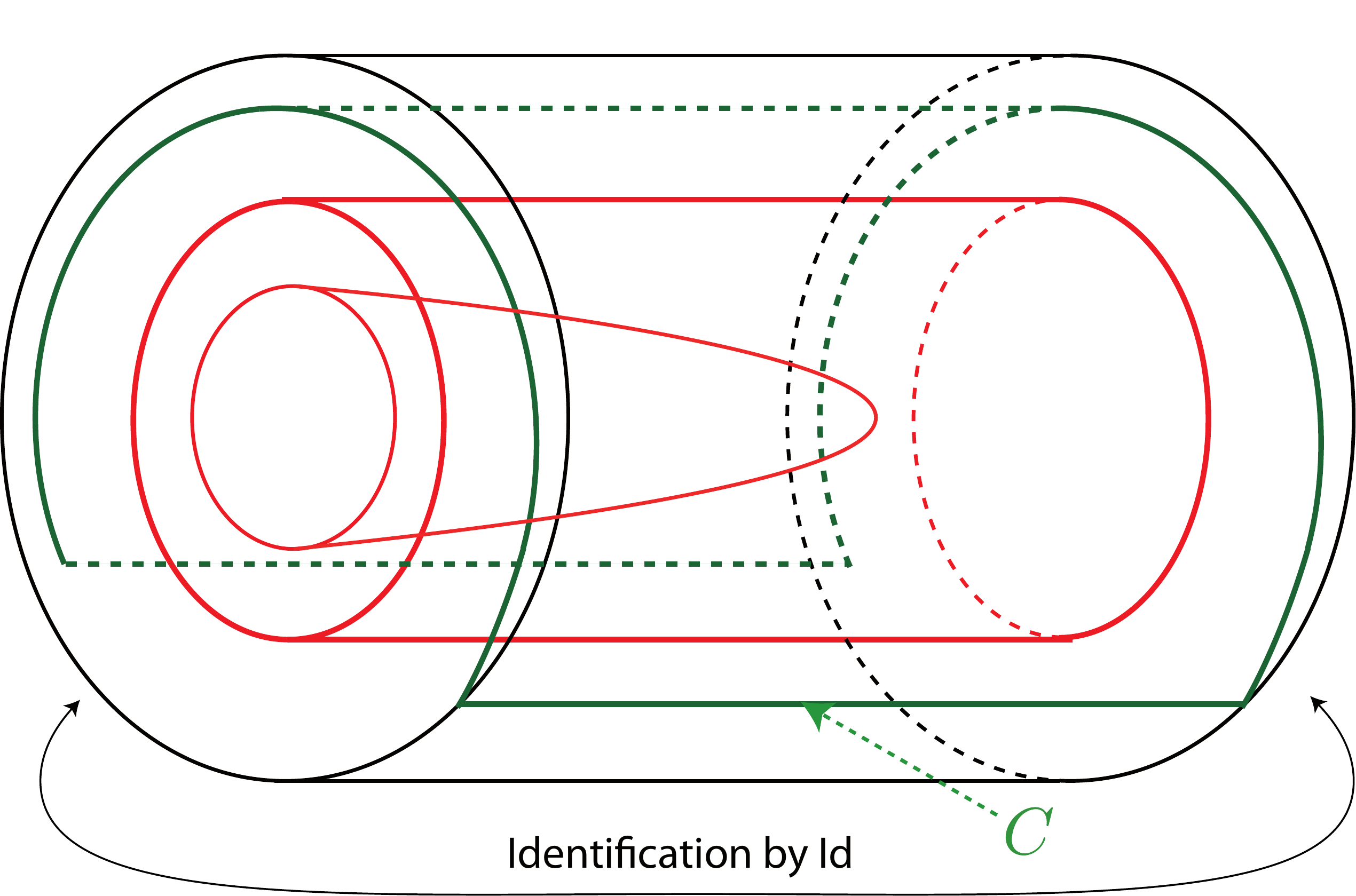}
\caption{Part of the Reeb foliation of $\S^{3}$}\label{S3}
\end{figure}
 
Hence we have proved the following Lemma :

\begin{LM}\label{del-Reeb}
A Reeb component $R$ can be deleted if and only if up to deleting a $T^{2}\t I$, $\partial R$ bounds a $\T$ component in $M\bsl \mathring{R} $ whose circles $C$ of the circle foliation induced by $\T$ on the transverse torus bound disks in $\T\cup R$
\end{LM}

\begin{RK}
Note that the Reeb foliation of $\S^{2}\t \S^{1}$ (foliation obtained by gluing two Reeb components to obtain $\S^{2}\t \S^{1}$) can be transformed by applying two times this process and we obtain the product foliation $\S^{2}\t \S^{1}$. Indeed that gives two solid tori trivially foliated by disks glued along their circle boundary two obtain sphere leaves.
\end{RK}

\begin{PR}\label{non-taut}
From each transversely oriented taut foliation $\F$ on a closed $3$-manifold $M$, ($M\not\cong \mathbb{S}^2\times\mathbb{S}^1$ trivially foliated), we can construct a non-taut foliation on $M$ (with a Reeb component) and a non-taut foliation without Reeb component on $M\backslash V$, where $V$ is a solid torus.\\

\end{PR}
  
\begin{proof}
By definition, there exists a closed transerve curve, say $\gamma$. Choose a small enough regular neighborhood of $\gamma$, denoted $V\cong \mathbb{D}^2\times\mathbb{S}^1$, so that the induced foliation by $\mathcal{F}$ on $V$ is the trivial foliation $\mathbb{D}^2\times\mathbb{S}^1$.\\

Now consider $M \backslash \mathring{V}$.\\
By construction the foliation induced on $\partial V$ is $(\partial \mathbb{D}^2)\times\mathbb{S}^1$.\\
Then we can apply the process of turbulization in $(\partial V)\times I$ by pasting a $\mathcal{T}$ component, to obtain a foliation $\mathcal{F}'$ of $M \backslash \mathring{V}$ with one torus boundary leaf.\\
Then, Proposition \ref{sep_torus} implies that $\mathcal{F}'$ is not taut.\\

This process of turbulization gives a Reeb component if and only if $M\cong \mathbb{S}^2\times\mathbb{S}^1$ with the product foliation.\\
Indeed suppose that our construction leads to a Reeb component $R$, i.e $M \backslash \mathring{V}=R$ so contains a $\T$ component. Then $M \backslash (\mathring{V}\cup \T)$ is a solid torus foliated by $\mathbb{D}^2\times\mathbb{S}^1$, homeomorphic to $M \backslash \mathring{V}$.\\
Then $M$ is a union of two solid tori, and since $V$ is foliated by disks, the transverse circles leaves of $\T$ bounds disks in $V$ and in $R$, so the identification of the two solid tori pastes the boundary of the meridians disks, hence $M \cong \mathbb{S}^2\times\mathbb{S}^1$.\\
The converse is trivial.\\
 
In conclusion, we have constructed a non-taut foliation on $M \backslash \mathring{V}$ without Reeb component.
By gluing  a Reeb component trivially to this torus leaf we obtain a non taut foliation with a Reeb component on $M$.\\
Hence we have proved Proposition \ref{non-taut}.

\end{proof}




\subsection{Non-taut and Reebless foliations}
Foliations admitting a Reeb component are not taut, but the converse is false: there are many non-taut and Reebless foliations.\\
There are two kinds of examples :
\begin{enumerate}
\item Non-taut and Reebless foliations on manifolds admitting a taut foliation;
\item Non-taut and Reebless foliations on manifolds without taut foliations.
\end{enumerate}

Many examples for Point (1) are constructed by Proposition \ref{non-taut}.\\
A simple example is the following. Consider the manifold $M=S_g \times \mathbb{S}^1$, where $S_g$ is a closed compact surface of genus $g$, for $g\geq 0$, with the trivial product foliation. Note that it is taut, so we can apply Proposition \ref{non-taut}, and we construct a non-taut and Reebless foliation on $M\bsl \mathring{V}$, where $V$ is a solid torus (see next figure for the case where $g=1$).\\
Note that by gluing two such foliations along their boundary torus leaves we obtain a separating torus leaf in a non-taut Reebless foliation.

  \begin{figure}[htb!]
\includegraphics[width=11cm]{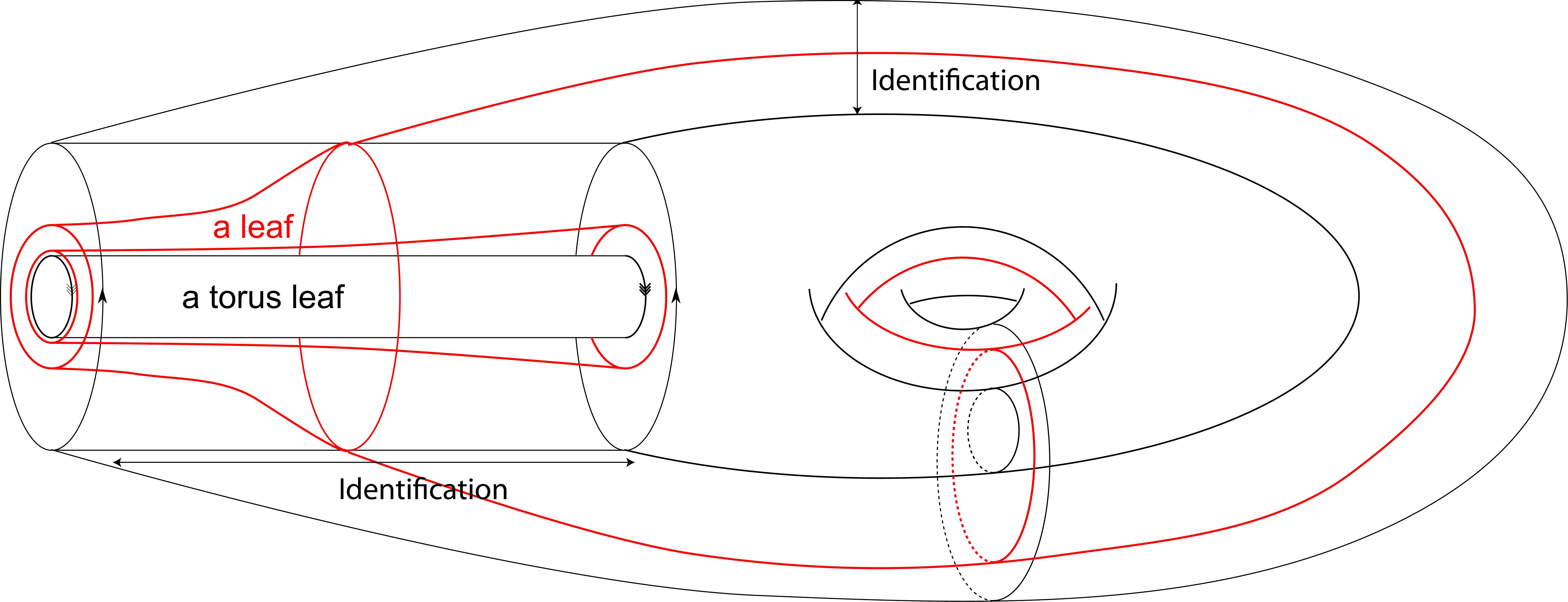}
\caption{Non taut foliation on $M\bsl \mathring{V}$, when $g=1$}\label{fig:16}
\end{figure}

Nevertheless, note that $Q$ admits a non-taut, transversely orientable Reebless foliation, it not obtained via Proposition \ref{non-taut}. This is the one constructed by R. \cite{Ro}, called type $II_b$ component, and given in Figure \ref{fig:11}.\\
\begin{figure}[ht!]
\includegraphics[width=6cm]{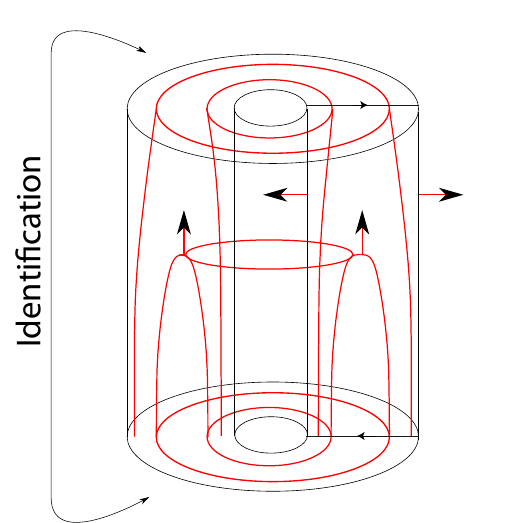}
\caption{Non-taut Reebless and transversely oriented foliation on $Q$}\label{fig:11}
\end{figure}

Proposition \ref{non-taut} shows that there are a lot of non-taut, Reebless foliations, since any taut foliation on a manifold $M$ gives rise to a non-taut Reebless foliation (on $M\bsl\mathring{V}$, where $V$ is a solid torus). \\

A more interesting question is the existence of non-taut Reebless foliations in a manifold not admitting a taut foliation, (Point (2)); i.e among homology spheres by Theorem \ref{Gabai}.\\
\cite{BNR} gave examples of such foliations on graph manifolds.

\begin{THM}[\cite{BNR}]
 There exist infinitely many manifolds without taut foliations admitting Reebless foliation (hence non-taut). Those are graph manifolds constructed by gluing two Seifert fibered manifold, each based on the disc with two exceptional fibers.
\end{THM}

\begin{Q}
There are infinitely many Seifert fibered homology spheres not admitting a taut foliations by \cite{GM}. Do they admit a non-taut  Reebless foliation? 
\end{Q}

\section{Non-separating torus leaf.}\label{non-sep-torus}

We have seen that a foliation with a separating torus leaf cannot be taut. \\
The case of non-separating torus leaves is very different since they can lie in a taut foliation or in a non-taut foliation.\\
The goal of Section \ref{non-sep-torus} is to understand the reason. 

In this section, we first give an example of a non-taut $\mathcal{C}^1$-foliation admitting a non-separating torus leaf, and then we give some constructions of taut and non-taut foliations admitting a non-separating torus leaf. \\
We will see that the key point amounts to do a \textit{good} spiraling (opposite direction of rotation in a neighborhood of the torus leaf) to obtain a taut foliation or a \textit{bad} spiraling (same direction of rotation in a neighborhood of the torus leaf) to obtain a non-taut foliation.\\
Then we prove Theorem \ref{main}.\\

\subsection{Example of non-taut foliation on $T^{3}$}\label{T3}
We study the well-known example of $T^3$ (where $T^3\cong \mathbb{S}^1\times \mathbb{S}^1\times \mathbb{S}^1$).\\
Here we give two examples of non-taut foliations with non-separating torus leaves on $T^{3}$. A non-transversely oriented one ($\L_{1}$); and a transversely oriented one ($\L_{2}$).\\

Let us represent $T^3$ as follows~:\\
Set $W=\{(x,\theta, z), x\in [0,1], \theta\in ]-\pi,\pi],z\in [0,1]\} \cong I\times \mathbb{S}^1\times I$.\\

Now $ W/\!\raisebox{-.65ex}{\ensuremath{\sim}} \cong T^{2}\times I$ where $(x,\theta,0)\sim (x,\theta,1)$; and $T^3$ is obtained by identifying the two following torus boundary components to obtain a non- separating torus $T\subset T^{3}$~:\\
 $T_0=\{(0,\theta, z), \theta\in ]-\pi,\pi], z\in [0,1]\}/\!\raisebox{-.65ex}{\ensuremath{\sim}}$ and\\
  $T_1=\{(1,\theta, z), \theta\in ]-\pi,\pi], z\in [0,1]\}/\!\raisebox{-.65ex}{\ensuremath{\sim}}$.\\

\begin{DEF}\label{def-L}
Foliate each $A_z=\{(x,\theta, z), x\in [0,1], \theta\in ]-\pi,\pi]\}$, for $z\in [0,1]$, by a Reeb annulus. That induces a foliation $\mathcal{L}$ on $T\times I\cong W/\!\raisebox{-.65ex}{\ensuremath{\sim}}$ because this foliation is invariant by $\sim$.
\end{DEF}
 In this foliation $\mathcal{L}$, $T_0$ and $T_1$ are leaves, which implies a foliation $\mathcal{L}_1$ on $T^3$.\\

 \begin{figure*}[htb!]
\begin{minipage}[b]{0.48\linewidth}
\centering

\centerline{\includegraphics[width=5cm]{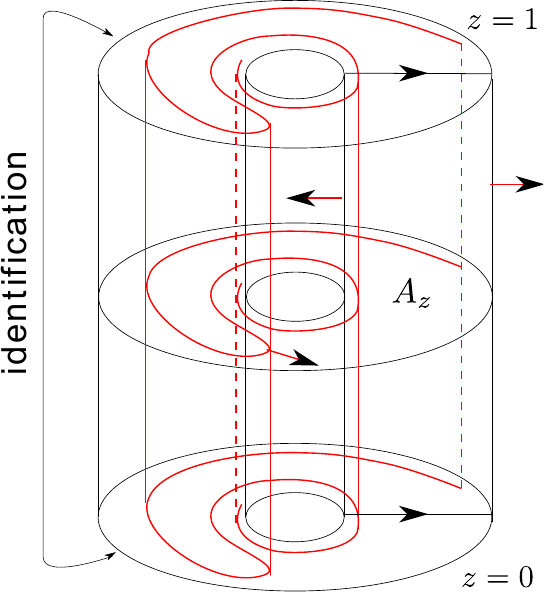}}
\centerline{\footnotesize{(a) Foliation $\mathcal{L}$ of $T\times I$.}}
\end{minipage}
\hfill
\begin{minipage}[b]{0.48\linewidth}
\centering

\centerline{\includegraphics[width=5cm]{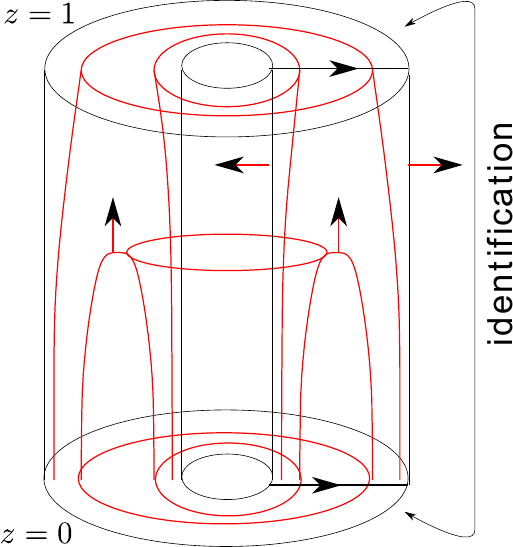}}
\centerline{\footnotesize{(b) Other representation of $\mathcal{L}$ (from \cite{Ro})}}
\end{minipage}
\caption{}
\label{L-T-I}
\end{figure*}

$\mathcal{L}_1$ is not taut, because any transverse loop passing through the torus leaf, would induce (after isotopy and splitting) a transverse arc on $A_z$, for some $z\in [0,1]$, with endpoints in $\partial A_z$, which is impossible since a Reeb annulus is not taut.\\
Moreover, $\mathcal{L}_1$ is not transversely orientable because there is no way of extending continuously a transverse vector field on $T$.\
Indeed,  $\mathcal{L}$ is transversely oriented. But on each $A_z$, the two boundary leaves have the same orientation. Hence on $ W/\!\raisebox{-.65ex}{\ensuremath{\sim}} \cong T^{2}\times I$, the two torus boundary leaves have the same orientation (for example, in Figure \ref{L-T-I}, they both point outward).\\
Thus, when gluing the two boundary leaves by a reversing orientation homeomorphism, the transverse orientations cannot match.\\

Note that $\mathcal{L}_1$ does not contain any non-orientable leaf and is not transversely orientable, which is the counterexample expected in Remark \ref{trans-or}.\\

Now we give the second example of non-taut foliation with non-separating torus leaf, but which is transversely oriented.\\

Consider the construction above, and glue with a reversing orientation homeomorphism two copies of $\mathcal{L}$, denoted by $\L$ and $\L^{*}$, where we add a $*$ to all the notations when we are in $\L^{*}$. The annuli $A_{z}$ and $A_{z}^{*}$ are attached along their boundary, so that the transverse orientation matches. We obtain a transversely oriented foliation $\mathcal{L}_2$ of $T^3$ (see the induced foliation by $\mathcal{L}_2$ on $A_z'=A_{z}\cup A_{z}^{*}$ for some $z\in [0,1]$ on Figure \ref{fig:13}).

\begin{figure}[ht!]
\includegraphics[width=5cm]{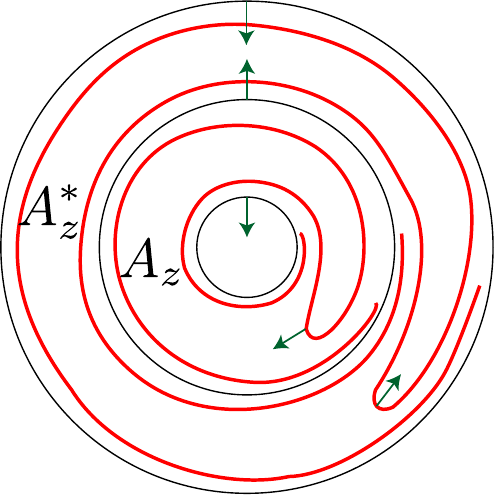}
\caption{Foliation on $A_z'$ induced by $\mathcal{L}_2$}\label{fig:13}
\end{figure}

In conclusion, $\mathcal{L}_2$ admits two non-separating torus leaves, is not taut and Reebless (no leaf is homeomorphic to $\mathbb{R}^2$), and is transversely orientable. Hence this is the expected foliation.\\
Note that by gluing together an even number of such components $\mathcal{L}$, this give an infinite number of such foliations.\\ 
 
Note that obviously $T^3$ admits a taut foliation, which is the product foliation; but we will see another interesting taut foliation constructed with  \textit{good} spiraling in Subsection \ref{taut-spi}, (see Figure \ref{taut-T-I}).

\subsection{Good orientation vs bad orientation}\label{taut-spi}
In this subsection we first give an example of construction where we can obtain a taut or a non-taut foliation depending on if we do a \textit{good} or a \textit{bad} orientation. \\
Theorem \ref{main} is a generalization of this fact.\\

 When $M$ is a manifold with two torus boundary components, then we denote $M/ \partial$ the manifold obtained by identifying the two boundary  torus components by the trivial homeomorphism.\\
 Let us study an interesting example~:  $M\cong F_g\times \mathbb{S}^1/\partial $ where $F_g$ is a twice punctured genus $g$ compact orientable surface.\\

When $g=0$, we obtain $T^3$. We have already seen in Subsection \ref{T3} an interesting Reebless, and non-taut foliation on $T^3$; here we will construct a taut one with non-compact leaves.\\

We set $M'= F_g\times \mathbb{S}^1$, with a fixed orientation and denote $\partial M' =T_{-}\cup T_{+}$, (a union of two tori).\\
We denote by $T$ the non-separating torus resulting from the identification of $T_{-}$ and $T_{+}$ in $M=M'/\partial $.\\

Consider on $M'$ the product foliation $\F'$. That induces on $T_{-}$ and $T_{+}$ a circle foliation. We want $T$ to be a torus leaf, so we are going to apply the process of turbulization (or equivalently spiraling by Lemma \ref{eq-turb-spir}) on $T_{-}$ and $T_{+}$. This amounts to glue two $\T$ components, and depending on the gluing, we can construct a taut foliation (Figure \ref{fig:32}), or a non-taut foliation (Figure \ref{fig:33}) by gluing two copies of $\mathcal{T}$ differently on the two torus boundary components.\\
In Figure \ref{fig:31} we have fixed a transverse orientation, and we explicit the two choices of turbulization. The key point is that the transverse orientation on the leaf attaching on the two transverse tori is the same, since any leaf of $\F'$ admits one boundary component on each transverse torus.

 \begin{figure*}[htb!]
\begin{minipage}[b]{0.48\linewidth}
\centering
\centerline{\includegraphics[width=6.5cm]{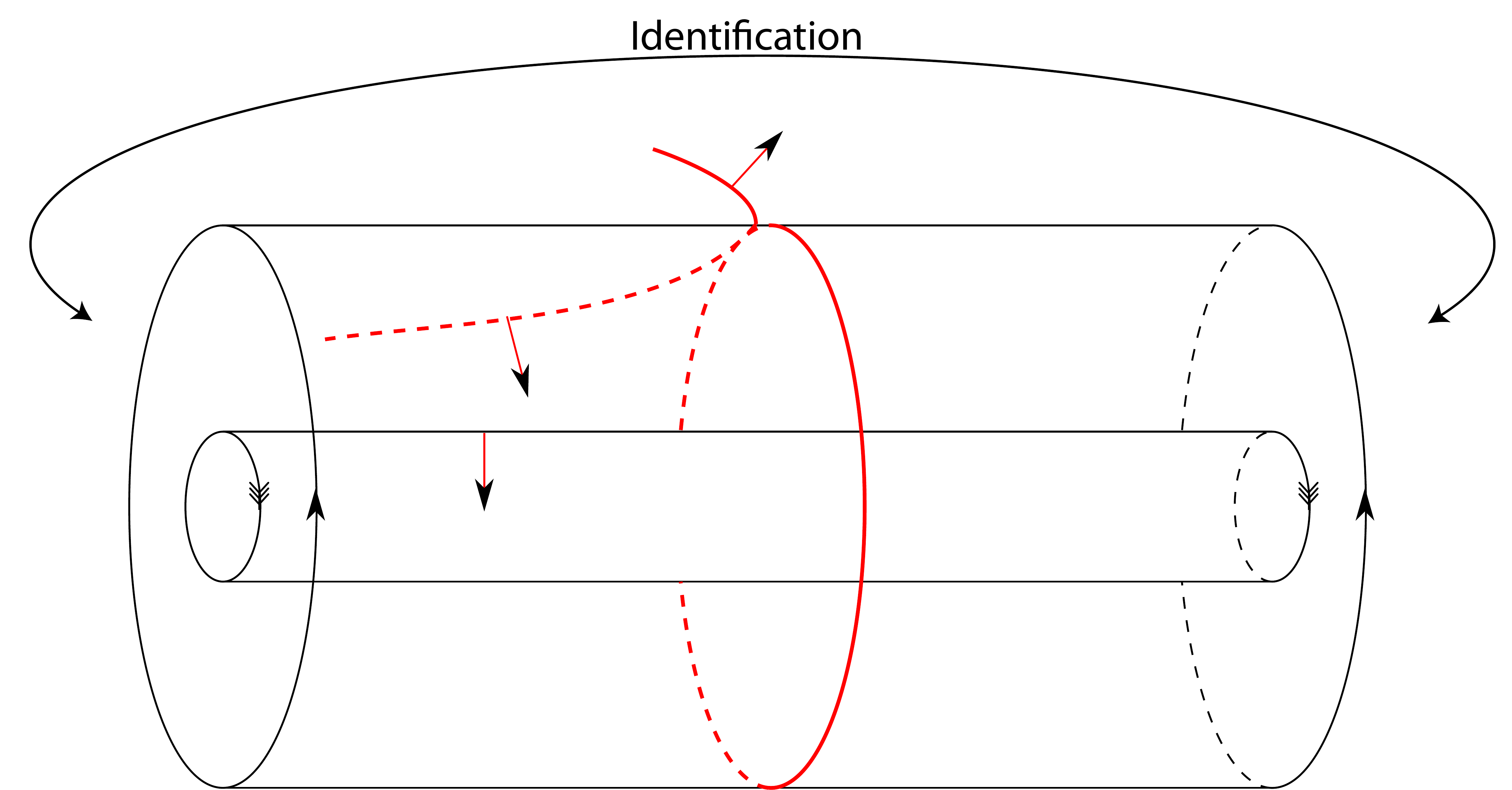}}
\centerline{\footnotesize{(a)$\mathcal {T}^+$}}
\end{minipage}
\hfill
\begin{minipage}[b]{0.48\linewidth}
\centering
\centerline{\includegraphics[width=6.5cm]{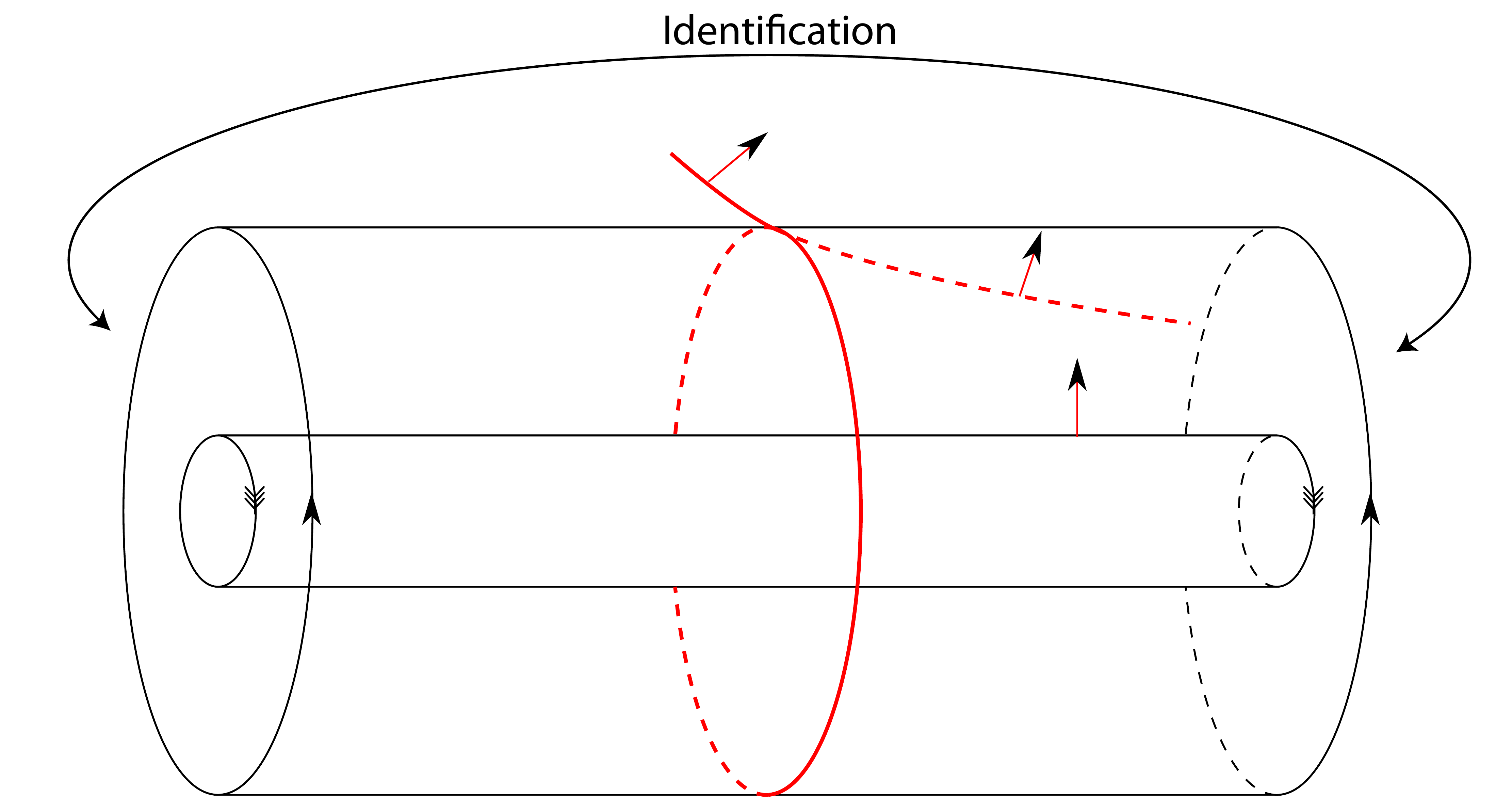}}
\centerline{\footnotesize{(b)$\mathcal {T}^-$}}
\end{minipage}
\caption{}
\label{fig:31}
\end{figure*}

Now we fix a transverse orientation on $\F'$.\\
We want to attach two components $ \mathcal{T}^+$ or $ \mathcal{T}^-$ on the boundary components of $M'$. \\
Let us denote $T_1$ and $T_2$ the two new torus boundary components after the pastings.\\
 There are two choices :
 \begin{enumerate}
\item We glue $ \mathcal{T}^+$ on one boundary component and $\mathcal{T}^-$ on the other. That induces opposite transverse orientation on $T_1$ and $T_2$, (one points inward and the other points outward), and so a taut transversely oriented foliation on $M'$ (choose for example the arc $\gamma$ in Figure \ref{fig:32}).\\
\item We glue $\T$, where $ \T\in\{\mathcal{T}^+, \mathcal{T}^-\}$ on each boundary component, that induces the same transverse orientation on $T_1$ and $T_2$. \\
This foliation is transversely oriented, so by Proposition \ref{sep_torus} it is non-taut.
\end{enumerate}

\begin{figure}[htb!]
\includegraphics[width=6cm]{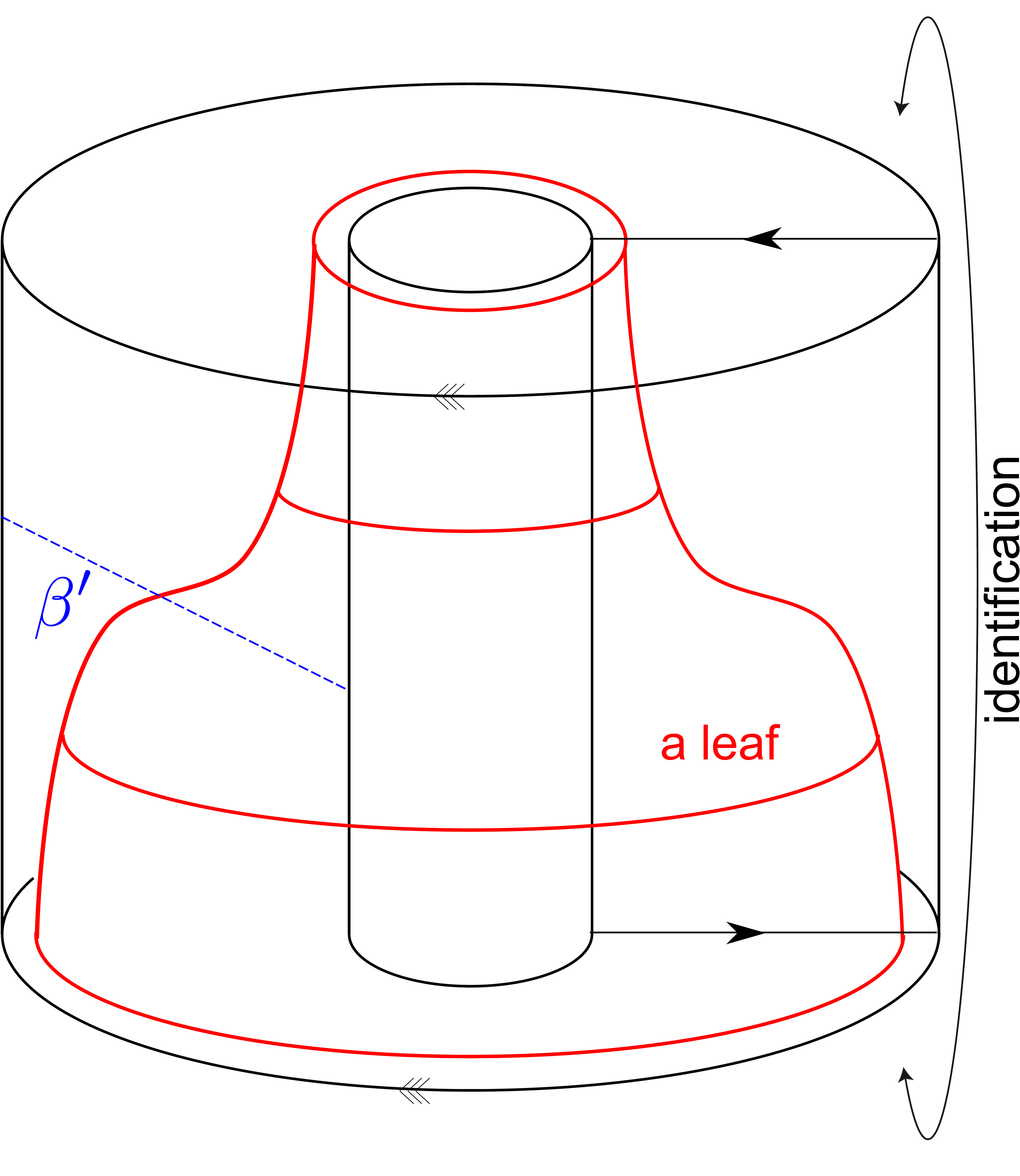}
\caption{Case (1) when $g=0$, i.e on $T\times I  $ }\label{taut-T-I}
\end{figure}

Now we identify $T_1$ and $T_2$.\\
 In the first case that induces a taut foliation admitting a non-separating torus leaf on $M=F_g\times \mathbb{S}^1/\partial $ (with non-compact leaves, but $T$).\\

 In the second case that induces a non-taut, and non-transversely oriented foliation admitting a non-separating torus leaf on $M=F_g\times \mathbb{S}^1/\partial $.\\
Indeed, this foliation is not transversely oriented, because the transverse orientation on $T_1$ and $T_2$ is the same, and the identification $T_1$ and $T_2$ reverse the orientation, so the transverse orientation on $T$ is not well defined.\\
Note that the case $g=0$ is exactly the foliation of $T^{3}$ of Subsection \ref{T3} (see Figure \ref{L-T-I}).\\
Note also that by gluing trivially two copies of such a foliation along the boundary torus leaves, we obtain a non-taut and transversely oriented foliation with two non-separating torus leaves.

 \begin{figure}[htb!]
\includegraphics[width=14cm]{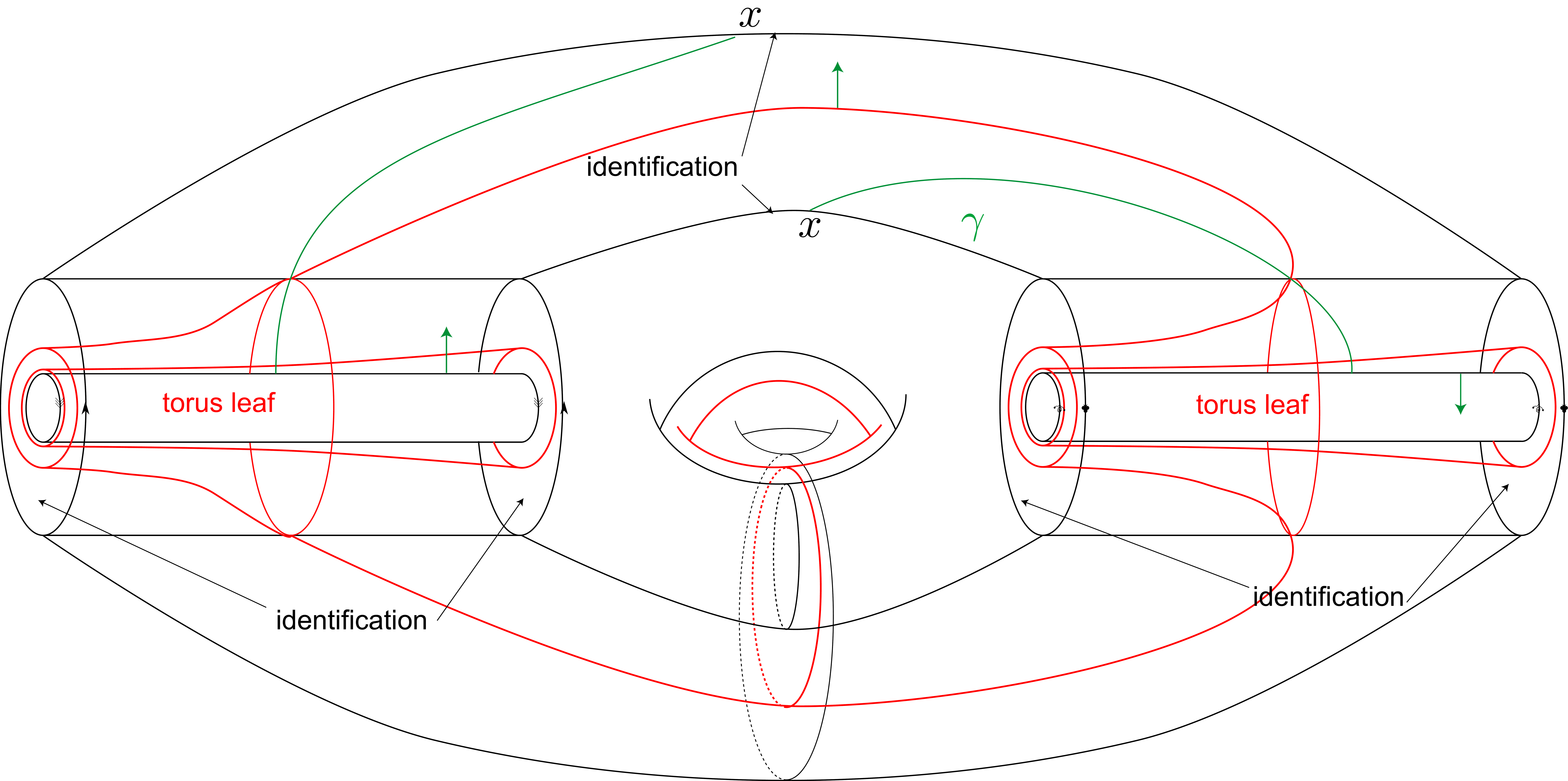}
\caption{Case (1) when $g=1$ : taut foliation}\label{fig:32}
\end{figure}

  \begin{figure}[htb!]
\includegraphics[width=14cm]{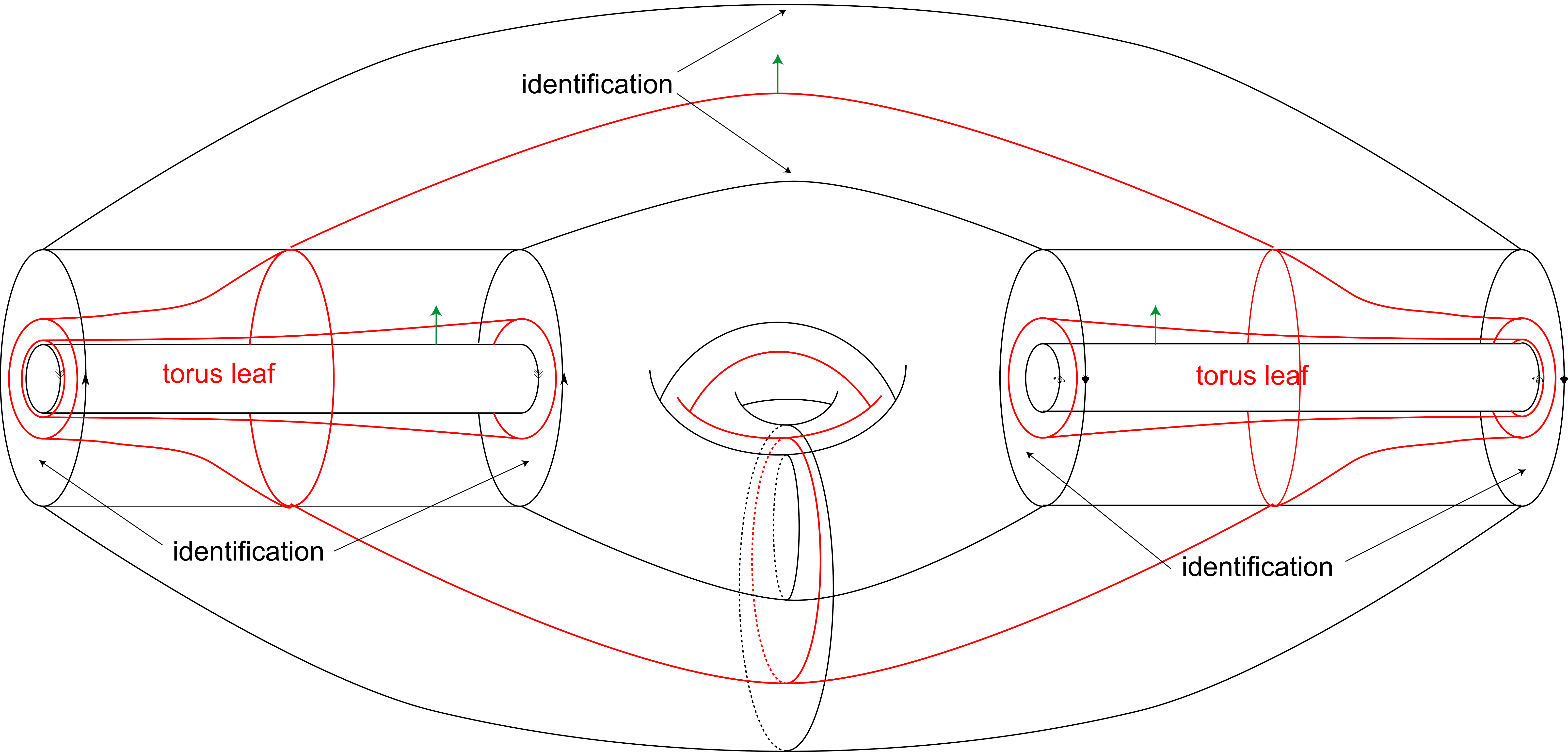}
\caption{Case (2) when $g=1$ : non-taut foliation}\label{fig:33}
\end{figure}

That example leads us to make the following definition :

\begin{DEF}
Let $M$ be a manifold with a transversely oriented foliation $\F$ such that the boundary of $M$ is a union of torus leaves.\\
We say that $\F$ has a \textbf{\emph{bad orientation}} if the transverse orientation on each boundary torus leaf is the same.\\
Otherwise we say that $\F$ has a \textbf{\emph{good orientation}}.
\end{DEF}

\begin{THM} \label{bad-good-2-comp} 
Let $M$ be a manifold with a transversely oriented $\C^{1}$-foliation~$\F$.\\
Assume that the boundary of $M$ is a union of two torus leaves.\\
Assume also that $\F$ does not admit neither interior torus leaf, nor embedded annulus whose induced foliation by $\F$ is a Reeb annulus. \\
Then, $\F$ is taut if and only if $\F$ has a good orientation.\\
\end{THM}

\begin{proof}
Proposition \ref{sep_torus} gives exactly that if $\F$ has a bad orientation then $\F$ is non-taut.\\
This is equivalent to say that if $\F$ is taut then $\F$ has a good orientation.\\
It remains to show that if $\F$ has a good orientation then $\F$ is taut. This is the goal of the following.\\
Let us denote by $T'_{1}$ and $T'_{2}$ the two tori boundary leaves. Choose an embedded torus $T_{i}$ in a regular neighborhood of $T'_{i}$ for $i=1,2$, and denote by $N(T_{i}')$ the regular neighborhood of $T_{i}'$ bounded by $T_{i}'$ and $T_{i}$, for $i=1,2$.\\
 By the Theorem of \cite{Ro} and \cite{Th} we can assume that $T_{1}$ and $T_{2}$ are transverse to $\F$.\\
 Fix an orientation on $M$. Up to considering the opposite transverse orientation on $\F$, we can assume that the transverse orientation on $T'_{1}$ points in $M$ and points out of $M$ on $T'_{2}$.\\
 If we denote by $N= M\bsl (\mathring{N(T_{1}')\cup N(T_{2}')})$ the oriented manifold homeomorphic to $M$, bounded by $T_{1}$ and $T_{2}$, and the induced foliation by $\F$ on $N$ by $\G$, then $\G$ does not admit torus leaves, so by Corollary \ref {gen-goodman} it is taut.\\
 
 \begin{CL}\label{arc}
There exists a properly embedded arc $\gamma : I \rightarrow N$ transverse to $\G$ with an endpoint on $T_{1}$ and another on $T_{2}$.
 \end{CL}

\begin{proof}
Since $\G$ is taut, either we find a properly embedded transverse arc, or a closed transverse curve to each leaf.\\
If there exists a properly embedded transverse arc, we note that it must have one endpoint on $T_{1}$ and another on $T_{2}$. Indeed if both endpoints are on the same boundary component then by Proposition \ref{sep_torus-arc}, it cannot be a transverse arc.\\
If we find a closed transverse curve, it can be chosen so that it meets all the leaves, then we can cut this curve in two points to obtain a transverse arc, and isotope it so that the endpoints meet $T_{1}$ and $T_{2}$, and keep being transverse to $\G$. Indeed, it suffices to pick one leaf $F_{1}$ meeting $T_{1}$ and one (other) leaf $F_{2}$ meeting $T_{2}$, and cut the closed loop at the points it meet $F_{i}$, $i=1,2$, and push the endpoints to the boundary of $F_{i}$, $i=1,2$ by small isotopies transverse to the leaves in a neighborhood of $F_{i}$, $i=1,2$. So we have proved Claim \ref{arc}.\\
\end{proof}

Therefore, up to considering $t\rightarrow \gamma (1-t)$, we can assume that $\gamma (0)\in T_{1}$ and $\gamma (1)\in T_{2}$. Moreover we make the confusion between $\gamma$ and $\gamma (I)$. Let $J=[0,\epsilon]$, where $\epsilon >0$ is small enough.\\
It remains to understand why we can extend $\gamma$ in $M$ to obtain a properly embedded arc transverse to $\F$.\\

\begin{CL}\label{possible}
If $\F$ is transversely oriented with a good orientation, the only possibilities for $\G$ to be transverse to the $T_{i}$, $i=1,2$ are the one on Figure \ref{good-or}.
\end{CL}

  \begin{figure}[htb!]
\includegraphics[width=10cm]{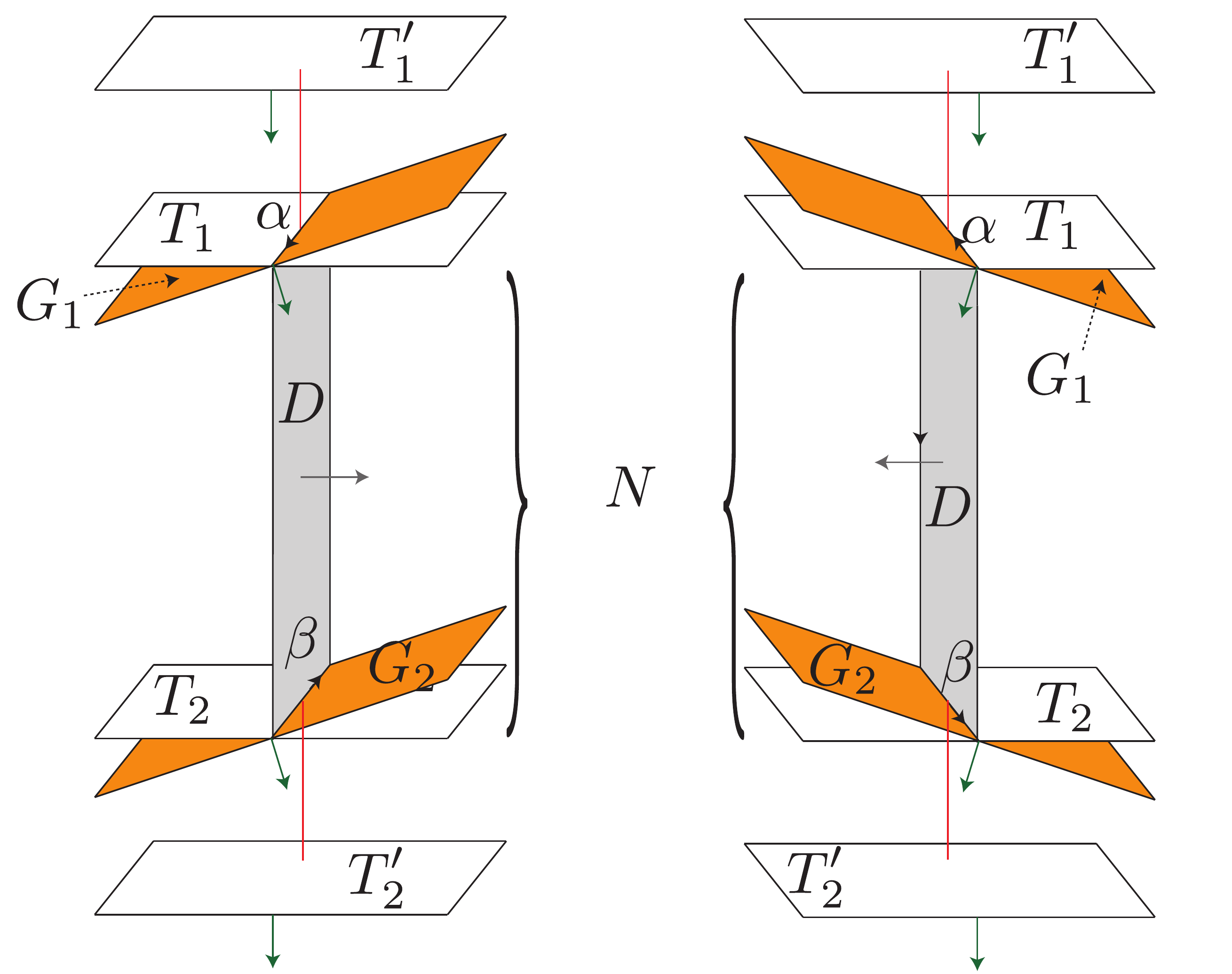}
\caption{Coherent orientation }\label{good-or}
\end{figure}

 \begin{figure}[htb!]
 \includegraphics[width=10cm]{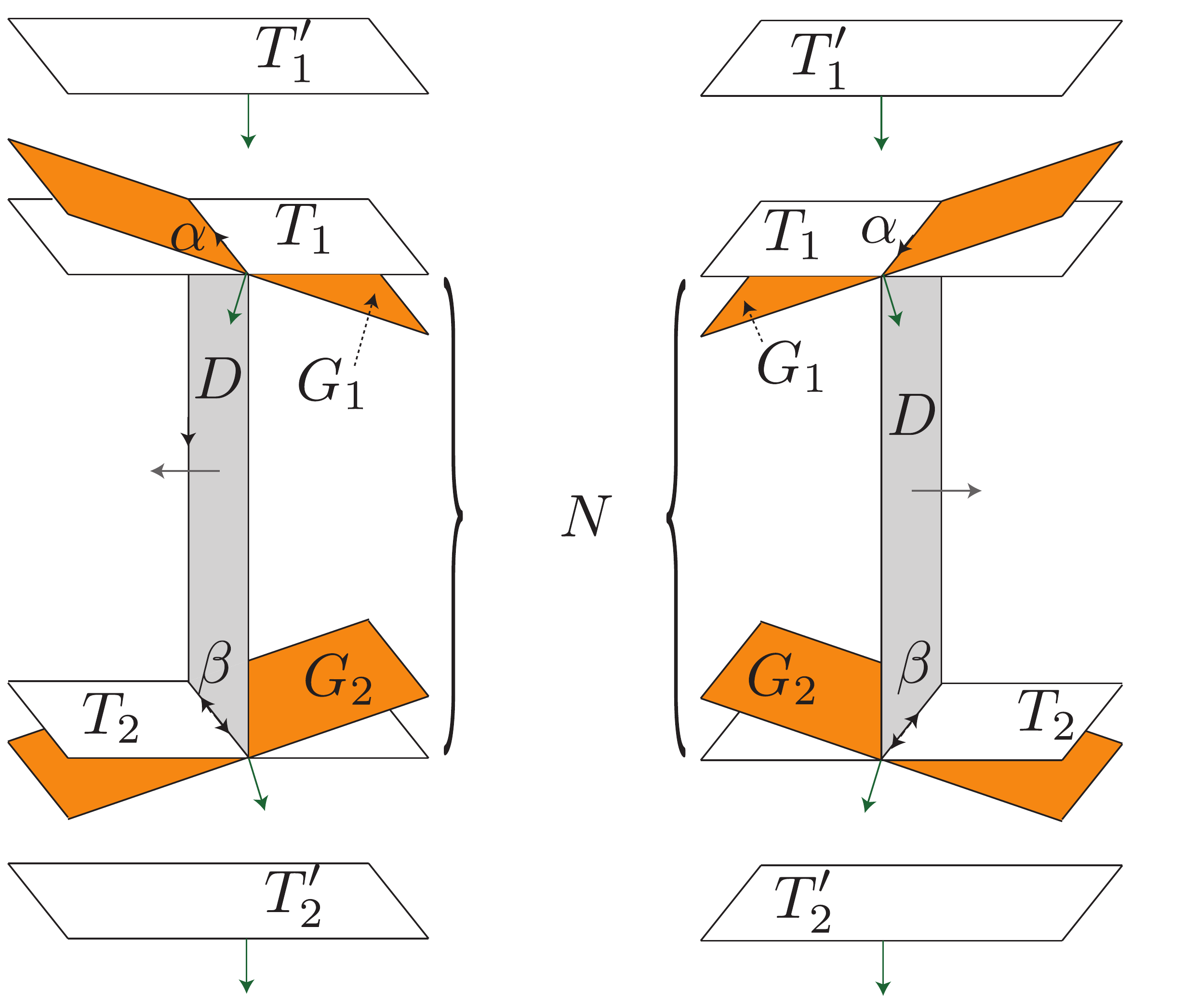}
\caption{Non-coherent orientation }\label{bad-or}
\end{figure}

\begin{proof}
Recall that we have always assumed that $M$ is orientable.\\
We can choose a small disk, denoted $D\cong \gamma\t J$ in $N$ such that for each $t_{0}\in I$, $\gamma (t_{0})\t J$ is included in a leaf of $\G$. It admits an arc $\alpha =\gamma (0) \t J \subset \partial D$, so $\alpha\subset T_{1}\cap G_{1}$, where $G_{1}\in \G$, and an arc $\beta =\gamma (1)\t J  \subset \partial D$, so $\beta\subset T_{2}\cap G_{2}$, where $G_{2}\in\G$.\\
Since $\F$ is transversely oriented, so is $\G$, and the transverse orientation of $\G$ induces an orientation on $D$ which must be coherent, because $N$ is oriented.\\
Indeed, the transverse orientation of $G_{1}$, induces an orientation on $\alpha\subset \partial G_{1}$, denoted $\vec{\alpha}$. This orientation induces also an orientation on $D$, because $\alpha\subset\partial D$.\\
Similarly, the transverse orientation of $G_{2}$, induces an orientation on $\beta\subset \partial G_{2}$, denoted $\vec{\beta}$. But since $D$ and $N$ are oriented, the induced orientation on $\partial D$ imposes $\vec{\beta}=-\vec{\alpha}$.\\

Moreover, there are two ways of being transverse to each $T_{i}$ ($i=1,2$), which gives four possibilities. Figure \ref{good-or} showes the two possibilities with a coherent orientation, between the one induced on $D$ and the one induced by $\partial G_{1}$ and $\partial G_{2}$.\\
Figure \ref{bad-or} shows the other two possibilities where the induced orientation on $D$ by $G_{1}$ is not coherent with the induced orientation on $\beta$ by $G_{2}$ in $N$; which ends the proof of Claim \ref{possible}
\end{proof}

Now we use Proposition \ref{torus-leaf} to understand the foliation of $N(T_{i}')$ in $M$, for $i=1,2$. Indeed, since there is no interior leaves in $\F$, each $T_{i}'$ bounds a $\Ss _{*}$ or a $T_{*}$ component (without embedded Reeb annulus). Hence we can easily find an extension of $\gamma$ in $M$, as in Figure \ref{good-or} which ends the proof of Theorem \ref{bad-good-2-comp}. 
\end{proof}

\noindent Now we prove Theorem \ref{main}.
\begin{proof}
If there are two torus boundary leaves, then this is Theorem \ref{bad-good-2-comp}.\\
Otherwise, it remains to understand that for each torus leaf we can find another torus leaf with opposite orientation (we suppose there is a good orientation). So by Theorem \ref{bad-good-2-comp}, we find a properly embedded arc; which proves Theorem \ref{main}.
\end{proof}

\begin{RK}
Note that given a manifold with a transversely oriented $\C^{1}$-foliation without embedded Reeb annulus, and with interior torus leaves, we can split along all the torus leaves. If we obtain a connected manifold we can apply Theorem \ref{main} to know if it is taut. Of course the connectedness is crucial.
\end{RK}

\begin{RK}
Note that Gabai's spiraling constructs a taut foliation with a non-separating torus leaf, by splitting along it, and Gabai's process imposes a good orientation by considering the component $\Ss (f,h)$ with only a transverse annulus on one boundary component and the remaining of the boundary component of $\Ss (f,h)$ is tangent to the foliation.
\end{RK}

\bibliographystyle{plainnat}
\bibliography{reference}

\end{document}